\documentclass[11pt]{article}
\usepackage{dsfont}
\usepackage{multicol,amssymb,amsmath}
\usepackage{mathrsfs}
\usepackage{amsthm}
\usepackage{graphicx}
\usepackage{mathtools}
\usepackage{xcolor}
\usepackage{enumerate}
\usepackage[margin=15truemm]{geometry}
\usepackage{tikz}
\usepackage{ifthen}
\usetikzlibrary{positioning}
\usetikzlibrary{svg.path}
\usetikzlibrary{
	math,
	bbox,
	knots,
	hobby,
	decorations.pathreplacing,
	shapes.geometric,
	calc
}
\newtheorem{thm}{Theorem}[section]
\newtheorem{lem}[thm]{Lemma}
\newtheorem{prop}[thm]{Proposition}

\newtheorem{cor}[thm]{Corollary}
\theoremstyle{definition}\newtheorem{defn}[thm]{Definition}
\newtheorem{eg}[thm]{Example}
\newtheorem*{rmk}{Remark}
\newtheorem*{thm*}{Theorem}
\newtheorem*{cor*}{Corollary}
\newcommand{\HHH}{{\mathcal H}}
\newcommand{\HH}{{H}}

\newcommand{\KK}{{\mathcal K}}
\newcommand{\FF}{{\mathcal F}}
\newcommand{\FFF}{{\mathbf F}}
\newcommand{\RR}{{\mathbb R}}
\newcommand{\EE}{{\mathbb E}}

\newcommand{\CC}{{\mathbb C}}
\newcommand{\NN}{{\mathbb N}}

\newcommand{\ZZ}{{\mathbb Z}}

\newcommand{\nn}{\mathfrak{n}}
\newcommand{\mm}{\mathfrak{m}}

\newcommand{\frakA}{\mathfrak{A}}

\begin{document}
	\title{On von Neumann algebras generated by free Poisson random weights}
	\author{Zhiyuan Yang\footnote{ 
			Texas A\&M University, College Station, TX, 77843, United States. Email: zhiyuanyang@tamu.edu}}
	\maketitle
	\begin{abstract}
		We study a generalization of free Poisson random measure by replacing the intensity measure with a n.s.f. weight $\varphi$ on a von Neumann algebra $M$. We give an explicit construction of the free Poisson random weight using full Fock space over the Hilbert space $L^2(M,\varphi)$ and study the free Poisson von Neumann algebra $\Gamma(M,\varphi)$ generated by this random weight. This construction can be viewed as a free Poisson type functor for left Hilbert algebras similar to Voiculescu's free Gaussian functor for Hilbert spaces.  When $\varphi(1)<\infty$, we show that $\Gamma(M,\varphi)$ can be decomposed into free product of other algebras. For a general weight $\varphi$, we prove that $ \Gamma(M,\varphi) $ is a factor if and only if $ \varphi(1)\geq 1 $ and $ M\neq \CC $. The second quantization of subunital weight decreasing completely positive maps are studied. By considering a degenerate version of left Hilbert algebras, we are also able to treat free Araki-Woods algebras as special cases of free Poisson algebras for degenerate left Hilbert algebras. We show that the L\'{e}vy-It\^o decomposition of a jointly freely infinitely divisible family (in a tracial probability space) can in fact be interpreted as a decomposition of a degenerate left Hilbert algebra. Finally, as an application, we give a realization of any additive time-parameterized free L\'{e}vy process as unbounded operators in a full Fock space. Using this realization, we show that the filtration algebras of any additive free L\'{e}vy process are always interpolated group factors with a possible additional atom.
	\end{abstract}
	
		\tableofcontents
	\section{Introduction}
	
	Free Gaussian (semicircular) families have been intensively studied from both the probability and operator algebra perspectives, beginning with Voiculescu’s free Gaussian functor \cite{VDN92}, which identified the free group factors with the von Neumann algebras generated by free Gaussian families. Among the subsequent developments. Shlyakhtenko \cite{Shl97} introduced the free quasi-free Gaussian functor, which generalizes Voiculescu’s framework to the nontracial setting. In this setting, one considers a complex Hilbert space $\HHH$ with a modular structure in the sense of \cite{RV77}, meaning that $ \HHH $ is equipped with a densely defined closed conjugate linear involution $S$. Shlyakhtenko introduced the free Araki-Woods algebras associated with $ \HHH $ and $S$. These algebras are generated by the field operators $ X(h):=l(Sh)+l^*(h) \in B(\FF(\HHH))$ where $l(Sh)$ and $l^*(h)$ are the creation and annihilation operators on the full Fock space $\FF(\HHH)$ for $ h $ in the domain of $S$.
	
	While free Poisson families have also been widely studied, there has been little emphasis on the von Neumann algebras they generate and the second quantization maps between these algebras. In this paper, we explore this perspectives. To be precise, we consider a similar functor for left Hilbert algebras, or equivalently, von Neumann algebras with n.s.f. (normal semifinite faithful) weights which have natural modular structures according to the Tomita-Takesaki theory. Instead of defining the field operator $ X(h) $ to be the sum of a creation and an annihilation operator, we will let $ X(h) $ be the sum of a creation, an annihilation, and a preservation operator (which arises from the algebraic structure). Specifically, for a left Hilbert algebra $\frakA$, we will consider the von Neumann algebra $\Gamma(\frakA)$ generated by $X(h):= l(h)+l^*(Sh)+\Lambda(\pi_{l}( h )) $ for $h\in \frakA$, where $ \Lambda(\pi_{l}( h )): f_1\otimes \cdots \otimes f_n\mapsto (hf_1)\otimes f_2\otimes \cdots\otimes f_n $ is the preservation operator for $ \pi_l(h)\in B(\overline{\frakA}) $, and $ S$ the involution of $ \frakA $. The next paragraph will discuss why the operators $X(h)$ lead to the formation of free Poisson families.
	
	In free probability, it is well-known that an operator of the form $ l(h)+l^*(h)+ \Lambda(T)$ on the full Fock space naturally gives us a freely infinitely divisible variable with respect to the vacuum state (see \cite{GSS92}\cite{speicher1998combinatorial}\cite{NS06}\cite{MS06}). And in particular, when $\frakA$ is a von Neumann algebra $M$ with a faithful normal state and  $h\in M_{s.a.}$, the operator $ X(h):=\ell(h)+\ell^*(h)+\Lambda(\pi_l(h)) $ has centered compound free Poisson distribution whose intensity measure is the distribution of $h$. Therefore, the linear map $X:\frakA\to \Gamma(\frakA)$ can be considered as a noncommutative version of (centered) free Poisson random measure.
	Indeed, similar to classical Poisson random measure, \cite{BT05} (see also \cite{Ans02}\cite{BP14}\cite{AG15}) defined a free Poisson random measure over a localizable measure space $ (\Theta,\mathcal{E},\nu) $ to be a map $ Y $ from $ \mathcal{E}_0: = \{E\in \mathcal{E}: \nu(E)< \infty\} $ to a $W^*$-probability space $(\mathcal{A},\tau)$ such that
	\begin{enumerate}
		\item For any $ E\in \mathcal{E}_0 $, $Y(E)\in \mathcal{A}$ is a self-adjoint element with the free Poisson distribution of intensity $ \nu(E) $.
		\item For disjoint sets $E_1,\cdots,E_n$, $ Y(E_1),\cdots,Y(E_n) $ are freely independent, and $ Y(\cup_{i=1}^{n}E_i) = \sum_{i=1}^{n}Y(E_i) $.
	\end{enumerate}
	One can check that if we consider the commutative Hilbert algebra $L^\infty(\Theta,\mathcal{E},\nu  ) \cap L^2(\Theta,\mathcal{E},\nu  )$, and denote the random variable $ Y(E) := l(\mbox{1}_E)+l^*(\mbox{1}_E)+\Lambda(\mbox{1}_E)+\nu(E)1 = X(\mbox{1}_E)+ \nu(E)1$ on the full Fock space over $ L^2(\Theta,\nu) $, then the family $\{Y(E)\}_{E\in \mathcal{E}_0}$ satisfy all the conditions above (see e.g. \cite{BLKL18}). In particular, the map $E\mapsto X(\mbox{1}_E)$ can be considered as a centered version of free Poisson random measure. Now, when we replace $ (\Theta,\mathcal{E},\nu) $ by a von Neumann algebra with weight $ (M,\varphi) $ (with the corresponding left Hilbert algebra denoted $\frakA_\varphi$), we call the linear map $ X:\frakA_\varphi\to \Gamma(\frakA_\varphi)$ the centered free Poisson random weight with the intensity weight $ \varphi $. In this case, we also denote $ \Gamma(M,\varphi):= \Gamma(\frakA_\varphi) $. We will see that the random weight $X$ has properties similar to those of the free Poisson random measure:
	\begin{enumerate}
		\item For each projection $ p\in M $ with $ \varphi(p)<\infty $, $ X(p) $ is a self-adjoint operator with centered free Poisson distribution of intensity $\varphi(p)$.
		\item For orthogonal projections $p_1,\cdots,p_n\in M$, $X(p_1),\cdots,X(p_n)$ are freely independent, and $ X(\sum p_i)=\sum X(p_i) $.
	\end{enumerate}
	We also point out that in classical probability, this type of Poisson random weight is studied in the recent preprint \cite{chen2023noncommutative} by Chen and Junge, and therefore our free Poisson random weight can be considered as a free analogue of the construction in \cite{chen2023noncommutative}.
	
	We are mainly interest in the properties of the free Poisson von Neumann algebra $ \Gamma(M,\varphi):= \Gamma(\frakA_\varphi) $ generated by the free Poisson random weight $X$. Additionally, we study the second quantization construction within the free Poisson framework. Notably, we find that the appropriate morphism for constructing second quantization is given by normal, subunital, weight-decreasing, completely positive maps. We summarize our main results as following.
	
	\begin{thm*}
		Let $\varphi$ be a n.s.f. weight on a von Neumann algebra $ M $ and $ \Gamma(M,\varphi):=\Gamma(\frakA_\varphi) \subseteq B(\FF(L^2(M,\varphi))) $ be the free Poisson von Neumann algebra generated by the centered free Poisson random weight $X$.
		\begin{enumerate}
			\item (Standardness) The vacuum vector $\Omega$ of the Fock space is cyclic and separating for $\Gamma(M,\varphi)$. In particular, the vacuum state $ \varphi_\Omega = \langle \Omega,\cdot \Omega\rangle $ is faithful and normal on $ \Gamma(M,\varphi) $.
			\item (Factorality) $\Gamma(M,\varphi)$ is a factor if and only if $ \varphi(1)\geq 1 $ and $ M\neq \CC $. And when $ \Gamma(M,\varphi) $ is a factor, it has type $ \text{II}_1 $ or $ \text{III}_\lambda $ with $ \lambda\neq 0 $.
			\item When $ \varphi(1)=\alpha <\infty $, we have the isomorphism
				$$ \Gamma(M,\varphi) \simeq \begin{cases} \left(\underset{\alpha}{L(\ZZ)\ast (M,\frac{1}{\alpha}\varphi)}\right)\oplus \underset{1-\alpha}{\CC},\quad \varphi(1)=\alpha \leq 1\\
				L(\ZZ)\ast p \left((M,\frac{1}{\alpha}\varphi)\ast ( \overset{p}{\underset{1/\alpha}{\CC}}\oplus\overset{ 1-p}{\underset{1-1/\alpha}{\CC}} )\right) p, \quad  \varphi(1) = \alpha> 1.\end{cases}$$
			\item (Second quantization) If $ \psi$ is a n.s.f. weight on an another von Neumann algebra $ N $, and $ T:M\to N $ is a normal subunital weight decreasing completely positive map, then there exists a unique unital state preserving completely positive map $ \Gamma(T):\Gamma(M,\varphi)\to \Gamma(N,\psi) $ preserving the Wick products.
		\end{enumerate}
	\end{thm*}

	Another perspective we want to study is the Fock space model for (not necessarily compactly supported) freely infinitely divisible distributions. In studying these distributions, we find it natural to consider a degenerate variant of left Hilbert algebras, which we refer to as pseudo left Hilbert algebras (see Definition \ref{defn pseudo left Hilbert algebra}). This slight generalization allows us to unify both semicircular and free Poisson variables within the same free Poisson algebra $\Gamma(\frakA)$. Moreover, we observe that the L\'{e}vy-It\^{o}'s decomposition can be interpreted as a decomposition of a pseudo left Hilbert algebras constructed via free cumulants. Using this, we give a prove of free L\'{e}vy-It\^{o}'s decomposition using pseudo left Hilbert algebras.
	
	If we assume that the weight $\varphi$ is tracial, we are also allowed to deal with freely infinitely divisible distributions with noncompact support by considering operators affiliated with the free Poisson algebra $\Gamma(M,\varphi)$. Unlike the classical case, a freely infinitely divisible distribution has compact support if and only if its L\'{e}vy measure has compact support. Therefore, in the free L\'{e}vy-It\^{o}'s decomposition, the only essential unbounded part is the free compound Poisson variable for which we will use a different operator model which has a more clear structure for the study of affiliated operators (see Section 4 for details).
	
	As a application, we prove a convergence for higher variantion of a free L\'{e}vy process in the sense of \cite{AW18}, which leads to the following corollary about the filtration algebras of a free L\'{e}vy process.
	
	\begin{cor*}
		Let $ Z_t $ be a additive free L\'{e}vy process with L\'{e}vy triplet $ (a,b,\rho) $ in some $W^*$-probability space, then for any $t>0$, the filtration von Neumann algebra $ M_{t} = W^*((Z_s)_{s\leq t}) $ is an interpolated free group factor (with a possible additional atom)
		$$ M_{t} \simeq \begin{cases}
			L(\FFF_{\infty}), \quad \mbox{if }b\neq 0 \mbox{ or } \rho(\RR)=\infty,\\
			L(\FFF_{ 2t\rho(\RR) }), \quad \mbox{if }b= 0 \mbox{ and } 1\leq t\rho(\RR)<\infty\\
			\underset{t\rho(\RR)}{L(\FFF_{ 2})}\oplus \underset{1-t\rho(\RR)}{\CC}, \quad \mbox{if }b= 0 \mbox{ and } t\rho(\RR)<1.
		\end{cases} $$
	\end{cor*}
	
	The paper is organized as follows. Section 2 recalls some basic facts about free cumulants, Fock space, (pseudo) left Hilbert algebras, weights on von Neumann algebras, and completely positive maps. In Section 3, we construct the free Poisson random weight and the free Poisson von Neumann algebra for a (pseudo) left Hilbert algebra or a n.s.f. weight. In Section 4, we describe the isomorphism class of free Poisson algebras when the weight is finite. Section 5 is devoted to the second quantization of normal subunital weight-decreasing completely positive maps. In Section 6, we discuss factoriality and type classification of the free Poisson algebra. As an application, in Section 7, we prove a multivariable version of L\'{e}vy-It\^{o}'s decomposition in a tracial probability space. In section 8, we give a realization of a (unbounded) freely L\'{e}vy process as affiliated operators of the free Poisson algebra and compute the isomorphism class of its filtration algebras.

	\section{Preliminaries}
 	\subsection{Freely independence and free cumulants}
 	A \textit{non-commutative probability space} is a pair $(\mathcal{A},\varphi)$, where $\mathcal{A}$ is a unital (*)-algebra over $\CC$ and $\varphi:\mathcal{A}\to \CC$ is a linear functional such that $ \varphi(1)=1 $. If $\mathcal{A}$ is a von Neumann algebra and $\varphi$ is a faithful normal state, then $(\mathcal{A},\varphi)$ is called a \textit{$W^*$-probability space}. Throughout this paper we mainly focus on $W^*$-probability spaces.
 	
 	A family $\{\mathcal{A}_i\}_{i\in I}$ of unital subalgebras of $ \mathcal{A} $ are called free (or freely independent) from each other if for all $ n\geq 1 $, $ \varphi(x_1x_2\cdots x_n) =0$ whenever $ \varphi(x_j)=0 $ for all $j=1,\cdots,n$ and $ x_j\in A_{i_j} $ with $ i_1\neq i_2\neq\cdots \neq i_n $. We say a family of elements $\{x_i\}_{i\in I} $ are free (resp. $*$-free) if the algebras generated by each $x_i$ (resp. $\{x_i,x_i^*\}$) are free.
 	
	A \textit{partition} $\pi$ of a finite set $ S $ is a collection $ \pi= \{V_1,\cdots,V_s\} $ of disjoint nonempty subsets of $S$ such that $ \cup_{i=1}^{s}V_i = S $. Those $ V_i $'s are called blocks of $\pi$, and the number of the blocks in $ \pi $ is denoted by $ |\pi| $. We denote by $ \mathcal{P}(S) $ the set of all partitions of $S$. If $ S $ is a totally ordered finite set, a partition $ \pi\in \mathcal{P}(S) $ is called \textit{noncrossing} if there are no different blocks $V_p$, $ V_q $ with $s_1,s_2 \in V_p$, $t_1,t_2\in V_q$ such that $ s_1<t_1<s_2<t_2 $. We will denote by $ \mathcal{NC}(S) $ the set of noncrossing partitions of $ S $. When $ S $ is the set of integers $ [n]:=\{1,2,\cdots,n\} $, we will also write $ \mathcal{P}(n):= \mathcal{P}([n]) $ and $ \mathcal{NC}(n):=\mathcal{NC}([n]) $.
	
	The set of partitions $ \mathcal{P}(S) $ can be equipped with the reverse refinement order, i.e. for two partitions $\pi,\sigma\in \mathcal{P}(S)$ we say $ \pi\geq \sigma $ if each block of $ \sigma $ is contained in a block of $ \pi $.
	
	\begin{defn}
		For a fixed integer $ n\geq 1 $, consider two copies of the ordered set $ [n]=\{1,\cdots,n\} $, $ [\bar{n}]=\{\bar{1},\cdots,\bar{n}\} $ and their union $S=\{1,\bar{1},\cdots,n,\bar{n}\}$ with the order $ 1<\bar{1}<\cdots<n<\bar{n} $. For a given $ \pi \in \mathcal{NC}(n) $, its \textit{Kreweras complement} $K(\pi)\in \mathcal{NC}(\bar{n})\simeq \mathcal{NC}(n)$ is defined as the biggest element in $ \mathcal{NC}(\bar{n}) $ such that $ \pi\cup \sigma \in \mathcal{P}(S)$ is still a noncrossing partition. 
	\end{defn}
	
	The Kreweras complement $ K:\mathcal{NC}(n)\to \mathcal{NC}(n) $ gives a self-bijection of $ \mathcal{NC}(n) $ and has the important property that $ |\pi|+|K(\pi)|= n+1$ for any $ \pi\in \mathcal{NC}(n) $. Its square $ K^2:\mathcal{NC}(n)\to \mathcal{NC}(n) $ is given by a `cyclic permutation', in the sense that for a noncrossing partition $ \pi=\{ \{i^1_1,\cdots,i^1_{k_1}\}, \cdots,\{i^s_1,\cdots,i^s_{k_s} \}\}\in \mathcal{NC}(n) $, $ K^2(\pi)=\gamma(\pi):=\{ \{\gamma(i^1_1),\cdots,\gamma(i^1_{k_1})\}, \cdots,\{\gamma(i^s_1),\cdots,\gamma(i^s_{k_s}) \}  \} $ where $ \gamma\in S_n $ is the cyclic permutation $ \gamma=( n,n-1,\cdots,1 ) $. In particular, for any $ \pi\in\mathcal{NC}(n) $, we also have $ K(\pi)=\gamma( K^{-1}(\pi) ) $.
	
	The $n$-th moments of $(\mathcal{A},\varphi)$ is the $n$-linear functional $ M_{n}:\mathcal{A}^n\to \CC $ such that $ M_{n}(x_1,\cdots,x_n) = \varphi(x_1\cdots x_n) $. More generally, for a partition $ \pi\in \mathcal{P}(n)$, we define the multilinear functional $ M_{\pi}:\mathcal{A}^n\to \CC $ as $$ M_{\pi}( x_1,\cdots,x_n ):= \prod_{V\in \pi} M_{|V|}( x_i: i\in V )= \prod_{V\in \pi} M_{|V|}( x_{v_1},\cdots,x_{v_r} ),  $$
	where $ V=\{v_1,\cdots,v_r\} $ and $ v_1<\cdots<v_r $.
	
	\begin{defn}
		For $n\geq 1$, the \textit{free cumulants} $R_n$ is the $n$-linear functional $ R_n:\mathcal{A}^n\to \CC $ defined recursively by the relations
		$$ M_{n}(x_1,\cdots,x_n)=\sum_{\pi\in \mathcal{NC}(n)}R_\pi(x_1,\cdots,x_n) ,\quad \forall x_1,\cdots,x_n\in \mathcal{A}, $$
		where for each $ \pi\in \mathcal{NC}(n) $,
		$$ R_{\pi}( x_1,\cdots,x_n ):= \prod_{V\in \pi} R_{|V|}( x_i: i\in V ).  $$
	\end{defn}

	Using free cumulants, we have another characterization of freeness: two families $F_1,F_2\subseteq \mathcal{A}$ are free if and only if $ R_n(x_1,\cdots,x_n)=0 $ whenever $ x_j\in F_1\cup F_2 $ and for some $j_1,j_2$, $ x_{j_1}\in F_1 $ and $ x_{j_2}\in F_2 $.
		
	Recall that a linear functional $\psi$ on $\mathcal{A}$ is \textit{tracial} if $ \psi(xy)=\psi(yx) $ for all $x,y\in \mathcal{A}$. In particular, for a tracial state, we always have that $M_n$ is cyclic symmetric $M_n(x_1,\cdots,x_n)= M_n(x_n,x_1,\cdots,x_{n-1})$, which also implies that $R_n$ is cyclic symmetric.
	
	When $ x,y\in \mathcal{A} $ are free, the moments and free cumulants of their product $ xy $ can be calculated via the Kreweras complements.
	
	\begin{thm}[\cite{nica1996multiplication}]\label{thm moments of product}
		Let $ (\mathcal{A},\varphi) $ be a non-commutative probability space with two free unital subalgebras $ \mathcal{A}_1,\mathcal{A}_2 $. If $ x_1,\cdots,x_n\in \mathcal{A}_1 $ and $ y_1,\cdots,y_n\in \mathcal{A}_2 $, then
		$$ R_n(x_1y_1,\cdots,x_ny_n)=\sum_{\pi\in \mathcal{NC}(n)}R_{\pi}(x_1,\cdots,x_n)R_{K(\pi)}(y_1,\cdots,y_n),$$
		and
		\begin{equation}\label{moments of product}
			M_{n}(x_1y_1,\cdots,x_ny_n)=\sum_{\pi\in \mathcal{NC}(n)}R_{\pi}(x_1,\cdots,x_n)M_{K(\pi)}(y_1,\cdots,y_n)= \sum_{\pi\in \mathcal{NC}(n)}M_{K^{-1}(\pi)}(x_1,\cdots,x_n)R_\pi(y_1,\cdots,y_n). 
		\end{equation} 
	\end{thm}
	For a proof, see \cite{nica1996multiplication}\cite{NS06}.
	
	When there are several algebras or states, we also write $ M_{\pi}=M_{\pi,\mathcal{A}}=M_{\pi,\varphi} $ and $ R_{\pi}=R_{\pi,A}=R_{\pi,\varphi} $ to emphasize that the corresponding moments and cumulants are taken for the pair $(\mathcal{A},\varphi)$.
	
	\subsection{Free cumulants of operators on full Fock spaces}
	Throughout the paper, we will always assume the inner product $\langle \cdot,\cdot\rangle$ of a Hilbert space is linear on the second component and conjugate linear on the first component.
	
	The \textit{full Fock space} $ \FF(\HHH) $ over a (complex) Hilbert space $ \HHH $ is the Hilbert space
	$$ \FF(\HHH) = \CC\Omega \oplus \bigoplus_{n=1}^{\infty}\HHH^{\otimes n}, $$
	where $ \Omega $ is a distinguished vector of norm one and is called \textit{the vacuum vector}. The vector state $\varphi_{\Omega}:= \langle \Omega, \cdot\Omega\rangle $ on $B(\FF(\HHH))$ is called \textit{the vacuum state}.
	
	If $ \HHH,\KK $ are two Hilbert spaces and $ T \in B(\HHH,\KK)$ is a contraction, then we define the operator $ \FF(T) \in B(\FF(\HHH),\FF(\KK))$ as
	$$ \FF(T)\xi_1\otimes \cdots \otimes \xi_n = (T\xi_1)\otimes\cdots \otimes (T\xi_n),\quad \forall \xi_i\in \HHH.$$ If $T,S$ are two contractions then we obviously have $ \FF(ST)=\FF(S)\circ \FF(T) $.
	
	For $ \xi\in \HHH $, the (left) \textit{creation operator} $ l(\xi) $ is the bounded operator defined as
	$$ l(\xi)\Omega = \xi,\quad l(\xi)\eta_1\otimes \cdots \otimes \eta_n = \xi\otimes \eta_1\otimes \cdots \otimes \eta_n,\quad \forall n\geq 1, \eta_i\in \HHH.  $$
	The adjoint of $l(\xi)$ is called the (left) annihilation operator of $\xi$, which satisfies
	$$ l^*(\xi)\Omega=0,\quad l^*(\xi)\eta_1\otimes \cdots \otimes \eta_n = \langle \xi,\eta_1 \rangle \eta_2\otimes \cdots \otimes \eta_n,\quad \forall n\geq 1, \eta_i\in \HHH. $$
	Note that while $ l(\xi) $ is linear on $\xi$, $ l^*(\xi) $ is actually conjugate linear on $\xi$.
		
	For any (possibly unbouned) linear operator $T$ on $ \HHH $, we define the gauge (preservation) operator $ \Lambda(T) $ on $\FF(\HHH)$ as
	$$ \Lambda(T)\Omega = 0,\quad \Lambda(T)\eta_1\otimes \cdots \otimes \eta_n = (T\eta_1)\otimes \eta_2\otimes \cdots \otimes \eta_n,\quad \forall n\geq 1, \eta_i\in D(T). $$	
	When $ T $ is densely defined and closed, $ \Lambda(T) $ is also densely defined and closable and moreover $ \Lambda(T^*)\subseteq \Lambda(T)^* $. For simplicity, we will again denote the closure of $ \Lambda(T) $ by $ \Lambda(T) $. If $ T $ is bounded, then $ \|\Lambda(T)\|\leq \|T\| $.
	
	The free cumulants of the creation, annihilation and gauge operators are given by the following formula (Proposition 13.5 \cite{NS06}).
	\begin{thm}\label{cumulants of operators on Fock spaces}
		Let $\HHH$ be a Hilbert space and consider the non-commutative probability space $ (B(\FF(\HHH)), \varphi_\Omega) $. Then for all $ \xi,\eta\in \HHH $ and $ T_1,\cdots,T_{n-2}\in B(\HHH) $,
		$$ R_n( l^*(\xi),\Lambda(T_1),\cdots,\Lambda(T_{n-2}),l(\eta) )=\langle \xi,T_1T_2\cdots T_{n-2}\eta \rangle, $$
		while all other free cumulants with arguments from the set $ l(\HHH)\cup l^*(\HHH)\cup \Lambda(B(\HHH))$ are $0$.
	\end{thm}
	
\subsection{Pseudo left Hilbert algebras and weights}
We first recall the basic properties of left Hilbert algebras.
\begin{defn}
	A \textit{left Hilbert algebra} is an (associative) involution algebra $\frakA$ over $\CC$ with the (antilinear) involution operator $S:\frakA\to \frakA$ and an (sesquilinear) inner product $ \langle \cdot,\cdot \rangle $ such that
	\begin{enumerate}[1)]
		\item For each $\xi\in \frakA$, the left multiplication $ \pi_l(\xi): \frakA\ni \eta\mapsto \xi \eta\in \frakA$ is a bounded operator.
		\item For all $\xi,\eta,\zeta\in \frakA$, $\langle \xi \eta,\zeta \rangle = \langle \eta,(S\xi)\zeta \rangle$.
		\item The involution $S$ is preclosed.
		\item The subalgebra $\frakA^2=\text{span}\{ \xi \eta:\xi ,\eta\in \frakA \}$ is dense in the Hilbert space $ \overline{\frakA} $.
	\end{enumerate}
\end{defn}

The main purpose of condition 4) is to guarantee that the representation $ \pi_l $ of $ \frakA $ is nondegenerate. We are in fact interested in a slightly more general version of left Hilbert algebras for which we do not assume 4).

\begin{defn}\label{defn pseudo left Hilbert algebra}
	A \textit{pseudo left Hilbert algebra} is an involution algebra $\frakA$ over $\CC$ with the (antilinear) involution operator $S:\frakA\to \frakA$ and a sesquilinear inner product $\langle \cdot,\cdot\rangle$ such that
	\begin{enumerate}[1)]
		\item For each $\xi\in \frakA$, the left multiplication $ \pi_l(\xi): \frakA\ni \eta\mapsto \xi \eta$ is a bounded operator.
		\item For all $\xi,\eta,\zeta\in \frakA$, $\langle \xi \eta,\zeta \rangle = \langle \eta,(S\xi)\zeta \rangle$.
		\item The involution $S$ is preclosed.
	\end{enumerate}
\end{defn}

Let $ \overline{\frakA} $ be the completion of $\frakA$ (hence $ \overline{\frakA} $ is a Hilbert space) and denote the closure of $S$ on $ \overline{\frakA} $ by $S$ again. We also denote $D(S)\subseteq \overline{\frakA}$ be the domain of $S$. For each $ \xi\in \overline{\frakA} $, denote also the unbounded right (left) mulplication operator $\pi_r(\xi):\eta\mapsto \eta\xi$, (resp. $ \pi_l(\xi):\eta \mapsto \xi\eta $).
\begin{defn}
	Let $\frakA$ be a pseudo left Hilbert algebra. A vector $ \xi\in \overline{\frakA} $ is called \textit{left (right) bounded} if $ \pi_l(\xi) $ ($\pi_r(\xi)$) is bounded. We denote $ \mathfrak{B} $ ($ \mathfrak{B}' $) the set of all left (right) bounded vectors. We denote also $ \frakA' = D(S^*)\cap  \mathfrak{B}' $ and $ \frakA'' = D(S)\cap \mathfrak{B} $. We say $\frakA$ is full if $ \frakA = \frakA'' $.
\end{defn}

We will see that actually a full pseudo left Hilbert algebra is simply a direct sum of the actual left Hilbert algebra $ (\frakA^2)'' $ and the subspace $ D(S)\cap (\frakA^2)^{\perp} $ with the trivial multiplication (i.e. $\xi\eta = 0$). Since the part $ D(S)\cap (\frakA^2)^{\perp} $ has no information about multiplication, it is uniquely determined by the modular structure, or equivalently, the standard (real) subspace $\KK$ of the fixed points of $S$ in $ D(S)\cap (\frakA^2)^{\perp} $. Therefore, a full pseudo left Hilbert algebra is nothing but a full left Hilbert algebra together with a Hilbert space equipped with a standard subspace. One advantage of considering a pseudo left Hilbert algebra is that we can consider the free Araki-Woods algebras as a special case of free Poisson algebras (which will be defined later.) Another advantage is that pseudo left Hilbert algebra in fact appears naturally when studying the L\'{e}vy-It\^{o}'s decomposition of infinitely divisible variables.

One of the crucial results about left Hilbert algebras is that $ \frakA' := D(S^*)\cap \mathfrak{B}' $ is a right Hilbert algebra (i.e. an involutive algebra with involution $F=S^*$ such that the right multiplication is bounded and satisfies the right version of conditions 2)-4)). It can be shown that $\frakA' = \frakA'''$ and $ \frakA''=\frakA'''' $. Thus one can consider $ \frakA'' $ as a form of closure of $ \frakA $ (in fact, $ \frakA $ is dense in $ \frakA'' $ in vector norm as well as in the strong operator topology). We will see later that exactly the same kind of properties also holds for pseudo left Hilbert algebras.

\subsubsection{n.s.f. weights and left Hilbert algebras}
Let us also briefly recall some basic facts about (nondegenerate) left Hilbert algebras and weights on von Neumann algebras. We will mainly follow \cite{Tak03}, and we refer to \cite{Tak03} and \cite{stratila2020modular} for the details.

\begin{defn}
	A \textit{weight} on a von Neumann algebra $M$ is a map $ \varphi:M_+\to [0,\infty] $ such that $ \varphi(\lambda x+y)=\lambda \varphi(x)+\varphi(y) $ for all $ x,y\in M_+ $ and $\lambda\geq 0$. We say that $ \varphi $ is faithful if $ \varphi(x)\neq 0 $ for all positive $x\neq 0$; semifinite if the span of the set of elements with finite weight $ \{ x\in M_+: \varphi(x)<\infty \} $ generate $ M $; normal if for every bounded increasing net $ (x_i)_i $ in $ M_+ $, $ \varphi(\sup x_i)=\sup \varphi(x_i) $. In short, we say that $ \varphi $ is a n.s.f. weight if $\varphi$ is normal, semifinite, and faithful.
\end{defn}

Each left Hilbert algebra is naturally associated with a von Neumann algebra $ \mathcal{R}_l(\frakA) := \pi_{l}(\frakA)'' \subseteq B(\overline{\frakA})$ with a n.s.f. weight $ \varphi $ determined by $ \varphi( \pi_{l}(\xi)^*\pi_{l}(\xi) ) = \|\xi\|^2 $. And the commutant of the von Neumann algebra $ \mathcal{R}_{l}(\frakA) $ is exactly $ \mathcal{R}_r(\frakA'):=\pi_{r}(\frakA')'' $ and in particular we always have $ \mathcal{R}_{l}(\frakA) = \mathcal{R}_{l}(\frakA'') $. In particular, for the study of von Neumann algebra, we can always assume that $ \frakA $ is full.

Conversely, each von Neumann algebra $M$ with a n.s.f. weight $\varphi$ arises from a full left Hilbert algebra $\frakA_\varphi $: Denote
$$ \nn_{\varphi} = \{ x\in M: \varphi(x^*x)<\infty \},\quad \mm_\varphi = \nn_{\varphi}^*\nn_{\varphi}. $$
We also have $ \mm_\varphi = \mbox{span}\{\mm_\varphi^+\} = \mbox{span}\{ x^*x: x\in \nn_{\varphi} \} $. Since $ \nn_{\varphi} $ is a left ideal of $M $, we have $ \mm_\varphi \subseteq \nn_\varphi $. Also, $ \mm_\varphi $ is operator norm dense in $ \nn_{\varphi} $ as every $ x\in \nn_\varphi $ can be approximated by $ (xu_i)_{i} $ where $ (u_i)_i $ is a right approximate identity of $ \nn_{\varphi} $.

Let $ L^2(M,\varphi) $ be the Hilbert space associated with the positive sesquilinear form $ (x,y)\mapsto \varphi( x^*y ) $ on $ \nn_{\varphi} $, and $ \eta_{\varphi}:\nn_{\varphi}\to L^2(M,\varphi) $ be the canonical embedding map. (Let $\pi_\varphi$ be the natural left representation of $M$ on $L^2(M,\varphi)$, then the triple $(\pi_\varphi,L^2(M,\varphi),\eta_{\varphi})$ is also called a semi-cyclic representation of $ M $.) The left Hilbert algebra associated with $ (M,\varphi) $ is $ \frakA_\varphi = \eta_{\varphi}( \nn_{\varphi}\cap \nn_{\varphi}^* )\subseteq L^2(M,\varphi) $ with multiplication $ \eta_{\varphi}(x)\eta_{\varphi}(y):=\eta_{\varphi}(xy) $ and involution $ S_\varphi: \eta_{\varphi}(x)\mapsto \eta_{\varphi}(x^*) $.

For a full left Hilbert algebra $\frakA$, let $S = J\Delta^{1/2}$ be the polar decomposition of $S$. Since $S^2=1$ and $S$ is antilinear, $J$ is antiunitary with $ J^*=J^{-1} $, and the right polar decomposition of $S$ is $S=\Delta^{-1/2}J$. The Tomita-Takesaki theory then says that $$ J\frakA = \frakA',\quad \Delta^{it}\frakA=\frakA,\quad \Delta^{it}\frakA'=\frakA' $$ and $$ J\pi_{l}(\xi)J = \pi_{r}(J\xi),\quad \pi_{l}(\Delta^{it}\xi) = \Delta^{it}\pi_{l}(\xi)\Delta^{-it}, \quad \forall u\in \frakA.$$
In particular, we have $ J\mathcal{R}_{l}(\frakA)J = \mathcal{R}_{r}(\frakA') $ and $ \Delta^{it}\mathcal{R}_{l}(\frakA)\Delta^{-it}=\mathcal{R}_{l}(\frakA) $. The antiunitary operator $ J $ is called the modular conjugation of $\frakA$. And when $\frakA = \frakA_\varphi$ for a n.s.f. weight $\varphi$, the one parameter automorphism group $ \{\sigma_{t}\}_{t\in \RR}\subseteq \text{Aut}(M)$ defined as $\sigma^{\varphi}_t ( x):= \Delta_\varphi^{it}x\Delta_\varphi^{-it} $ is called the modular automorphism group of $ M $.

For each pair $ (M,\varphi) $, $ M' = \mathcal{R}_{r}(\frakA'_\varphi) $ has a natural weight $\varphi'$ induced by the right Hilbert algebra $ \frakA'_\varphi $ such that $ \varphi'(\pi_{r}(x)^*\pi_{r}(y) ) = \langle x,y \rangle$ for all $x,y\in \frakA'_\varphi$. In fact, one can easily check that $\varphi'(J_\varphi x J_\varphi) = \overline{\varphi(x)} $ for all $x\in \mm_{\varphi}$.

\subsubsection{decomposition of full pseudo left Hilbert algebras}
We now turn to show that a full pseudo left Hilbert algebra decomposes into the direct sum of a part with standard subspace and a part of an actual left Hilbert algebra . Recall first that $S$ is defined to be the closure of its restriction on $\frakA$. In particular, $ \frakA $ is a core of $ S $, i.e. $ \frakA $ is dense in $ D(S) $ in the graph norm $ \|\xi\|_S:= \|\xi\|+\|S\xi\| $.

\begin{lem}\label{perp of A2 has trivial multiplication}
	Let $ \frakA $ be a pseudo left Hilbert algebra,
	then $ \frakA^2 $ is a left Hilbert algebra. For all $ \xi\in \frakA$, $\eta\in (\frakA^2)^{\perp} $, $ \xi \eta = 0 $. If furthermore $ \eta\in \frakA $, then $ \eta\xi=0 $.
\end{lem}
\begin{proof}
	a) To show that $ \frakA^2 $ is a Hilbert algebra, we only need to show that $ (\frakA^2)^2=\frakA^4 $ is dense in the closure of $ \frakA^2 $. Let $ \xi\in (\frakA^4)^\perp\cap \overline{\frakA^2} $, then for all $\xi_1,\xi_2,\xi_3,\xi_4\in \frakA$, we have
	$$ \langle \xi_4,\xi_1\xi_2\xi_3 \xi \rangle = \langle (S\xi_3)(S\xi_2)(S\xi_1)\xi_4,\xi \rangle=0,$$
	which implies $ \frakA^3 \xi =\{0\} $. Therefore, for all $ \xi_1,\xi_2\in \frakA $, we have
	$$ \langle \xi_1\xi_2\xi, \xi_1\xi_2 \xi\rangle = \langle \xi_2\xi,(S\xi_1)\xi_1\xi_2 \xi \rangle = 0,$$
	hence $ \frakA^2 \xi = 0 $. Repeat the same argument, we get $ \frakA \xi = 0 $ and therefore $ \xi\in (\frakA^2)^\perp $. But by assumption $ \xi\in \overline{\frakA^2}$, so we get $ \xi=0 $.
	
	b) For $ \eta\in (\frakA^2)^{\perp} $, we have by definition $ 0=\langle fg,\eta \rangle = \langle g,(Sf)\eta \rangle $ for all $  f,g\in \frakA  $, therefore we have $ (Sf)\eta = 0 $ for all $ f\in \frakA $, thus $ \xi \eta = 0 $ for all $ \xi \in \frakA $. To see that $ \eta\xi $ is also $0$ when $ \eta\in \frakA\cap (\frakA^2)^{\perp} $, we note that for all $ \zeta\in \frakA $, since $ (S\zeta)\eta =0 $, we have
	$$ \langle \zeta,\eta\xi \rangle = \langle (S\eta)\zeta,\xi \rangle = \langle S( (S\zeta)\eta ),\xi \rangle =0. $$
\end{proof}

\begin{rmk}
	From the previous lemma, it is already obvious that the right and left multiplication by a vector in $ (\frakA^2)^\perp $ should always be $0$, and only the $ \frakA^2 $ part has non-degenerate multiplication. However, to actually prove that $\frakA$ has the desired decomposition (when $\frakA$ is full), one major difficulty is to show that $ D(S) $ can be decomposed into $ D(S)\cap \overline{\frakA^2} $ and $ D(S)\cap (\frakA^2)^\perp $. To overcome this difficulty, we will follow \cite{Tak03} to first show that $ \frakA' $ (defined below) has the desired decomposition which then implies the same decomposition for $ \frakA''$ as $ \frakA'' = (\frakA')' $.
\end{rmk}

Let $ \mathcal{R}_l(\frakA)$ again be the von Neumann generated by $ \pi_l(\frakA) \subseteq B( \overline{\frakA})$ (note that this is a degenerate representation), then by the previous Lemma, $ \overline{\frakA^2} $ is an invariant subspace of $  \mathcal{R}_l(\frakA) $ and $ \mathcal{R}_l(\frakA) $ acts on $ ({\frakA}^2)^{\perp} $ as $0$. In particular, let $ p $ be the orthogonal projection from $\overline{\frakA}$ onto $ \overline{\frakA^2} $, then $ p $ is the unit of $ \mathcal{R}_l(\frakA) $.

Recall that for a right bounded vector $\xi\in \mathfrak{B}' $, $ \pi_r(\xi) $ is defined to be the multiplication on the right. In particular, we have $ \pi_r(\mathfrak{B}')\subseteq \mathcal{R}_l(\frakA)'\subseteq B(\overline{\frakA}) $. One can also easily check that $ \mathfrak{B}' $ is invariant under $ \mathcal{R}_l(\frakA)' $ and $a\pi_r(\xi) = \pi_r(a\xi)$ for all $ a\in \mathcal{R}_l(\frakA)' $ and $ \xi\in \mathfrak{B}' $.

We can now follow exactly the same arguments in [Lemma VI.1.8 to Lemma VI.1.13 \cite{Tak03}] which we summarize below.

\begin{lem}{\cite{Tak03}}\label{Lemmas of Takesaki}
	\begin{enumerate}[(1)]
		\item $ \pi_r(\mathfrak{B}')^* \mathfrak{B}' \subseteq \frakA'$, and $ S^*(\pi_r(\xi))^*\eta = \pi_r(\eta)^*\xi $ for all $\xi,\eta\in \mathfrak{B}'$.
		\item $\frakA'$ is a pseudo right Hilbert algebra.
		\item For each $ \eta\in D(S^*) $, $ \pi_{r}(\eta) $ and $ \pi_r(S^*\eta) $ are closable operators affiliated with $ \mathcal{R}_l(\frakA)' $, and $ \pi_r(\eta)\subseteq (\pi_r(S^*\eta))^* $, $ \pi_r(S^*\eta) \subseteq (\pi_r(\eta))^*$.
		\item Let $\eta\in D(S^*)$, and $ \pi_r(\eta)=uh=ku $ be the left and right polar decomposition. For each continuous function $ f\in C_c(0,\infty) $ supported on a compact subset of $ (0,\infty) $, we have $ f(h),f(k)\in \pi_r(\frakA') $, $ f(h)S^*\eta, f(k)\eta \in (\frakA')^2 \subseteq \mathfrak{B}'$, and
		$$ \pi_r(f(h)S^*\eta) = hf(h)u^*,\quad \pi_r(f(k)\eta) = kf(k)u,\quad S^*f(k)\eta = f(h)^*S^*\eta. $$
	\end{enumerate}
\end{lem}
A slight modification the proof of Lemma VI.1.13 \cite{Tak03} yield the following lemma.
\begin{lem}
	Let $ \eta\in D(S^*) $, and be $p$ the orthogonal projection onto $ \overline{\frakA^2} $, then $ p\eta\in D(S^*) $ and
	$$ S^*p\eta = pS^*\eta. $$
	In particular, $D(S^*)=(D(S^*)\cap \overline{\frakA^2} )\oplus (D(S^*)\cap {(\frakA^2)}^{\perp} )$. And for $ \eta \in \overline{\frakA} $, $ \eta \in D(S^*) $ if and only if $ p\eta,(1-p)\eta \in D(S^*) $.
\end{lem}
\begin{proof}
	We keep the notations in the last lemma. Let $ (f_n)_{n\geq 1} $ be an increasing sequence in $C_c(0,\infty)$ of positive functions such that $ f_n $ converges pointwisely to $1$ on $ (0,\infty) $, then we have strong operator limit $$ \lim_{n\to\infty} f_n(h) = s_r(\pi_r(\eta)),\quad \lim_{n\to\infty} f_n(k) = s_l(\pi_r(\eta)),$$
	where $ s_l(\pi_r(\eta)) $ and $ s_r(\pi_r(\eta)) $ are the left and right support of $ \pi_r(\eta) $. We claim that $ s_l(\pi_r(\eta))\eta = p\eta $. Indeed, since $ \pi_l(\frakA^2) $ acts trivially on $ (\frakA^2)^{\perp} $ and $ \overline{(\frakA^2)} $ is a non-degenerate representation of $ \pi_{l}(\frakA^2) $, we can choose a net $ (\xi_i) $ in $ \frakA^2 $ such that $ \{\pi_{l}(\xi_i)\} $ converges strongly to $p$. We then have
	$$ p\eta = \lim_{i} \pi_l(\xi_i)\eta = \lim_{i} \pi_r(\eta)\xi_i \in s_l(\pi_r(\eta))\overline{\frakA}, $$
	and in particular, $ s_l(\pi_r(\eta))p\eta =p\eta $. Since $ p $ is the unit of $ \mathcal{R}_l(\frakA) $, we obtain $ s_l(\pi_r(\eta))\eta =p\eta $. Similarly, we also have $ s_r(\pi_r(\eta))S^*\eta = pS^*\eta $. As $S^*$ is closed, the limits $ \lim_{n\to \infty} f_n(k)\eta= p\eta $ and $ \lim_{n\to \infty}f_n(h)\eta = pS^*\eta $ then imply that $ p\eta\in D(S^*) $ and $ S^*p\eta = pS^*\eta $.
\end{proof}

\begin{prop}\label{decomposition of pseudo hilbert algebra}
	For any pseudo left Hilbert algebra $\frakA$, $\frakA'=(\frakA^2)'\oplus ((\frakA^2)^\perp \cap D(S^*))$ with the multiplication $ (\xi_1+\eta_1)(\xi_2+\eta_2)=\xi_1\xi_2 $ for $ \xi_i\in (\frakA^2)' $ and $ \eta_i\in  (\frakA^2)^\perp \cap D(S^*)$.
	
	Similarly, $ \frakA'' = (\frakA^2)'' \oplus ((\frakA^2)^\perp \cap D(S))$ with the multiplication $ (\xi_1+\eta_1)(\xi_2+\eta_2)=\xi_1\xi_2 $ for $ \xi_i\in (\frakA^2)'' $ and $ \eta_i\in  (\frakA^2)^\perp \cap D(S).$
\end{prop}
\begin{proof}
	By Lemma \ref{perp of A2 has trivial multiplication}, for every $ \eta\in (\frakA^2)^{\perp} $, $\pi_r(\eta)=0$. Therefore, $ (\frakA^2)^\perp \subseteq \mathfrak{B}' $. Let $ w=\xi+\eta\in \frakA' $ with $ \xi\in \overline{\frakA^2} $ and $\eta\in (\frakA^2)^{\perp} $, by the previous lemma, $ \xi,\eta\in D(S^*) $. Therefore, $ \eta \in D(S^*)\cap \mathfrak{B}'=\frakA' $, which implies also $ \xi = w-\eta \in \frakA'\subseteq \mathfrak{B}' $. In particular, $ \xi\in \overline{\frakA^2}\cap D(S^*)\cap \mathfrak{B}' = (\frakA^2)' $. The multiplication rule then follows from the fact that $ \pi_r(\eta) = \pi_r(S^*\eta) = 0 $ for all $ \eta\in (\frakA^2)^\perp \cap D(S^*) $. Following the same arguments, one has also that $\eta\in \overline{\frakA}$ is in $D(S) $ if and only $ p\eta,(1-p)\eta\in D(S) $, and that $ \frakA'' = (\frakA^2)'' \oplus ( (\frakA^2)^\perp \cap D(S) ) $.
\end{proof}

Finally, we have the following modification of Theorem VI.1.26 \cite{Tak03} which shows the density of $\frakA$ in $\frakA''$.

\begin{prop}\label{approximate A'' by A}
	Let $ \frakA $ be a pseudo left Hilbert algebra. If $ \xi\in \frakA'' $, then for every $\varepsilon>0$, there exists a sequence $ (\xi_n ) $ in $\frakA$ such that
	$$ \lim_{n\to\infty} \|\xi-\xi_n\|=0, \quad \lim_{n\to\infty} \|S\xi-S\xi_n\|=0,\quad \text{and }\|\pi_l(\xi_n)\|\leq \max\{ \|\pi_l(\xi)\|,\varepsilon \}, $$
	and $ \pi_l(\xi_n) $ converges to $\pi_l(\xi)$ in the strong operator topology.
\end{prop}
\begin{proof}
	The same arguments in Theorem VI.1.26 \cite{Tak03} still work for pseudo left Hilbert algebra, except that we no longer require that $ \|\pi_l(\xi)\|\neq 0 $ (as for pseudo left Hilbert algebra it is possible that $ \pi_l(\xi)=0 $ for nonzero $\xi$). In particular, when $ \pi_l(\xi) =0$, one can only guarantee that $ \|\pi_l(\xi_n)\| $ is sufficiently small (i.e. less than $\varepsilon$).
\end{proof}

\begin{rmk}
	While $\frakA''$ has the decomposition $ \frakA'' = (\frakA^2)'' \oplus ( (\frakA^2)^\perp \cap D(S) ) $, in general there is no guarantee that $ \frakA $ also has similar decomposition. In fact, we do not even know whether the subspace $ \frakA\cap (\frakA^2)^\perp $ is nontrivial. Nevertheless, for our study of von Neumann algebras, the above density result shows that we can always assume that $\frakA$ is full (see Proposition \ref{density results}).
\end{rmk}

\subsection{Completely positive maps and bimodules}
We recall two equivalent description of bimodules of von Neumann algebras.

Let $M$ and $N$ be two von Neumann algebras. If $X$ is a self-dual right Hilbert $N$-module, we denote $\mathcal{L}_N(X)$ the algebra of all $N$-linear bounded operators on $X$.
\begin{defn}
	A $M$-$N$ $W^*$-bimodule is a pair $(X,\pi)$ with $X$ a self-dual right	$N$-module such that the linear span of $\{ \langle \xi,\eta \rangle_N,\xi,\eta\in X \} $ is $\sigma$-weakly dense in $N$, and $\pi:M\mapsto \mathcal{L}_N(X)$ a unital normal injective homomorphism.
\end{defn}

\begin{defn}
	A correspondence from $M$ to $N$ is a unital $*$-representation	$(\pi, \HHH)$ 
	of the algebraic tensor	product	$M\odot	N^{op}$	on a Hilbert space	$\HHH$ such that the restriction to both $M=M\odot 1$ and $N^{op}= 1\odot N^{op}$	is normal. We shall write $x\xi y =\pi(	x)\pi(y^{op})\xi$ for $x\in M, y\in N$ and $\xi\in \HHH$, where $y^{op}=Jy^*J\in N'\subseteq B(L^2(N))$ and we identify $N'$	with $N^{op}$.	
\end{defn}

These two notions are equivalent in the following ways: Given a $M$-$N$ $W^*$-bimodule $X$, let $\HHH(X)$ be the Hilbert space obtained by taking the completion of $ X\odot L^2(N) $ (quotient the null space) with respect to the
inner product
$$ \langle a_1\otimes \xi_1,a_2\otimes \xi_2 \rangle:= \langle \xi_1,\langle a_1,a_2 \rangle_N\xi_2 \rangle,\quad \forall a_i\in X, \xi_i\in L^2(N), $$
then $ \HHH(X) $ is a correspondence from $M$ to $N$ with $N^{op}$ acting on the second component $L^2(N)$. Conversely, if $\HHH$ is a correspondence from $M$ to $N$, then the space $ X(\HHH):= \mathcal{L}_N(L^2(N),\HHH) $ of all $N$-linear bounded operators from $ L^2(N) $ to $\HHH$ is a $M$-$N$ $W^*$-bimodule. We refer to \cite{anantharaman1990relative} for more details about the relationship between these two notions.

While the $W^*$-bimodule perspective is not directly relative to our Fock space construction, it is convenient to have it when discussing the second quantization construction in later sections.

If $\HHH$ is a correspondence from $M$ to $N$, let $\overline{\HHH}$ be the conjugate Hilbert space of $\HHH$, then $ \overline{\HHH} $ is a natural correspondence from $ N $ to $M$ with action $ y\hat{\xi}x := \overline{ x^* \xi y^*} $ for $x\in M$, $u\in N$, and $ \hat{\xi}\in \overline{\HHH} $. Therefore the $M$-$N$-bimodule $ X(\HHH) $ is also associated with a $N$-$M$-bimodule $X(\overline{\HHH})$. In general, there is no direct way of describing $ X(\overline{\HHH}) $ in terms of $ X(\HHH) $. However, if the bimodule or corresondence) is given by a normal completely positive map $T:M\to N$, then $X(\overline{\HHH})$ can infact be described by a dual map $ T^\star:N\to M $ in the sense of Accardi and Cecchini \cite{AC82} (which was then generalized by \cite{petz1984dual}).

If $T:M\to N$ is a normal completely positive map from $M$ to $N$, then $T$ is naturally associated with a $M$-$N$-bimodule $X_T$ obtained taking the completion of $M\odot N$ (quotient the null space) with respect to the $N$-valued inner product $ \langle m_1\otimes n_1,m_2\otimes n_2 \rangle_N:= n_1^*T(m_1^*m_2)n_2\in N $. Similarly, $T$ is naturally associated with the
correspondence $\HHH_T$ obtained by taking the competion of $ M\odot L^2(N) $ (quotient the null space) with respect to the (scalar-valued) inner product $ \langle m_1\otimes \xi_1,m_2\otimes \xi_2 \rangle_T:= \langle \xi_1,T(m_1^*m_2)\xi_2 \rangle $. For $m\in M,n\in N,\xi\in L^2(N)$, we will denote by $ m\otimes_T n\in X_T $ the image of $m\otimes n$ under the quotient and similarly by $ m\otimes_T \xi\in \HHH_{T} $ the image of $ m\otimes \xi $ under the quotient. We can easily check that $X_T\simeq X(\HHH_{T}):= \mathcal{L}_N(L^2(N),\HHH_{T})$ with the action
$$ (m\otimes_T n)\xi := m\otimes_T(n\xi )\in \HHH_{T},\quad \forall m\otimes_{T} n\in X_T,\xi\in L^2(N). $$

\subsubsection{dual of normal positive maps}
Suppose that $\varphi$ and $ \psi $ are two n.s.f. weights on $M$ and $N$. It is shown in \cite{petz1984dual} (as a generalization of \cite{AC82}) that if $T:M\to N$ is a normal positive linear map, and there exists a constant $C\geq 0$ such that $ \psi\circ T \leq C\varphi $, then we can construct a unique dual normal positive map $ T^\ast: N\to M $ from $N$ to $M$. To describe $ T^\ast $, we first recall a duality between $ M $ and $ \mm_{\varphi} $ for a fixed n.s.f. weight $\varphi$ on $ M $.

Let $ \varphi $ be a n.s.f. weight on $M$ and $ \varphi' $ the weight $ \varphi'(J_\varphi x J_\varphi) = \overline{\varphi(x)} $ on $M'$. It is proved by Haagerup \cite{HAAGERUP1975302} that for each $ y'\in \mm_{\varphi'}= J \mm_{\varphi} J \subseteq M'$, there is a unique normal linear functional $ \theta(x')\in M_* $ such that if $ y' = \pi_r(\xi_1)^*\pi_r(\xi_2) $ with $ \xi_1,\xi_2\in \mathfrak{B}'\subseteq L^2(M) $, then
$$ \theta( y')(x) = \langle \xi_1, x\xi_2 \rangle,\forall y\in M. $$
In fact, $ \theta:\mm_{\varphi'}\to M_* $ is linear completely positive and has dense image. Moreover, for self-adjoint $ y'\in \mm_{\varphi'} $, one has $ \|\theta(y')\| = \inf\{ \varphi'(x'_1)+\varphi'(x_2): x'_i\in (\mm_{\varphi'})_+, y'=x'_1-x'_2 \} $. For more details about $ \theta $, we refer to \cite{HAAGERUP1975302}, see also Lemma VII.1.8. \cite{Tak03} and Lemma 1.8. \cite{stratila2020modular}.

If we compose $ \theta $ with the isomorphism $ M\ni y\mapsto y^{op} = J y^* J\in M $, then we obtain a bilinear form
$$ [x,y]_{\varphi} := \theta(y^{op})(x),\quad \forall x\in M, \forall y\in \mm_{\varphi}. $$
Since for each $y\in \mm_{\varphi}$, $ \theta(y^{op}) \in M_*$ is in the predual, we call the norm on $ \mm_{\varphi}$, $ \| y\|_{L^1}:= \| \theta(y^{op}) \| $ the $ L^1 $-norm of $ y $. We will denote $ L^1(\varphi) $ the completion of $ \mm_{\varphi} $ with respect to $ \|\cdot\|_{L^1} $.

\begin{defn}{(\cite{petz1984dual})}
	Let $\varphi$ be a n.s.f. weight on $ M $. The \textit{biweight} $ [\cdot,\cdot ]_\varphi: M\times L^1(\varphi)\to \CC $ is a bilinear form on $ M $ and $ L^1(\varphi) $ such that
	$$ [x,y]_\varphi := \theta(y^{op})(x)= \langle J\xi_2, x J \xi_1 \rangle,\quad \forall x\in M, \forall y= \pi_l(\xi_1)^*\pi_l(\xi_2)\in \mm_{\varphi}, \xi_i\in \mathfrak{B}.$$
\end{defn}

Notice that for all $ y\in \mm_{\varphi} $, in fact $[1,y]_\varphi = \varphi(y)$. We also observe that $ [\cdot,\cdot]_\varphi $ is symmetric in the sense that if both $ x,y $ are in $ \mm_{\varphi} $, then $ [x,y]_\varphi = [y,x]_\varphi $. Indeed, suppose that $ x= \pi_l(\xi_1)^*\pi_l(\xi_2) $ and $ y= \pi_l(\eta_1)^*\pi_l(\eta_2)$ with $ \xi_i,\eta_i\in \mathfrak{B}\subseteq L^2(M) $, then
\begin{align*}
	[x,y]_\varphi &= \langle J \eta_2,xJ\eta_1 \rangle = \langle \pi_l(\xi_1)J\eta_2,\pi_l(\xi_2)J \eta_1 \rangle = \langle \pi_r(J\eta_2)\xi_1,\pi_r(J\eta_1)\xi_2 \rangle\\
	&=\langle J\pi_l(\eta_2)J\xi_1,J\pi_r(\eta_1)J\xi_2\rangle = \langle \pi_r(\eta_1)J\xi_2,\pi_l(\eta_2)J\xi_1 \rangle\\
	&= \langle J\xi_2,yJ\xi_1 \rangle = [y,x]_\varphi.
\end{align*}
Therefore, $[x,y]_\varphi$ is well-defined if one of $ x $ and $y$ is in $\mm_{\varphi}$. Also, for positive elements, we can also consider the following extension of $  [\cdot,\cdot]_\varphi $: for $\forall x\in M_+,y\in N_+$
$$ [x,y]_\varphi := \begin{cases}
	 \langle J\eta, x J \eta \rangle,&\text{if } x\in M^+, y= \pi_l(\eta)^*\pi_l(\eta)\in \mm_{\varphi}^+,\\
	  \langle J\xi,yJ\xi \rangle,&\text{if } y\in M^+, x= \pi_l(\xi)^*\pi_l(\xi)\in \mm_{\varphi}^+,\\
	 +\infty, &\text{otherwise}.
\end{cases} $$

\begin{prop}{(\cite{petz1984dual})}\label{petz lemma}
	\begin{enumerate}[(1)]
		\item If $ (x_n)$ is an increasing net of positive elements, then $ \lim_n [x_n,y]_\varphi=[ \lim_n x_n,y ]_\varphi $ for all $y\in M$.
		\item The biweight $ [\cdot,\cdot]_\varphi $ is a duality under which $ M $ is the conjugate space of $ L^1(\varphi) $.
		\item If $\omega:\mm_{\varphi}\to \CC$ is a positive linear map such that $ \omega \leq C \varphi $ on $ \mm_{\varphi}^+ $, then there exists a unique $ m\in M_+ $ such that $ \|m\|\leq C^{1/2} $ and $ [m,x]_\varphi = \omega(x) $ for all $ x\in \mm_{\varphi} $.
	\end{enumerate}
\end{prop}
The following is the main theorem of \cite{petz1984dual} with a slight modification.

\begin{thm}{(\cite{petz1984dual})}\label{duality of positive maps}
	Let $ \varphi,\psi $ be two n.s.f. weights on von Neumann algebras $ M $ and $ N $, and $ T:M\to N $ be a positive map such that $ T(\mm_{\varphi}^+)\subseteq \mm_\psi^+ $ and $\psi\circ T\leq C\varphi$ for some constant $C>0$. 
	\begin{enumerate}[(1)]
		\item There exists a unique normal positive map $ T^\star: N\to M $ such that for $ m\in \mm_\varphi $ and $n\in N$,
		$$ [ T^\star(n), m]_\varphi = [ n,T(m) ]_\psi,$$
		$ T^\star( \mm_{\psi} )\subseteq \mm_{\varphi} $, and $ \varphi\circ T^\star \leq \| T(1) \|\psi $.
		\item $ T|_{\mm_\varphi}=T^{\star \star}|_{\mm_\varphi} $. In particular, if $T$ is normal or $ \varphi(1)<\infty $, then $ T=T^{\star \star} $.
	\end{enumerate}
	
	Furthermore,
	\begin{enumerate}[(1)]
		\setcounter{enumi}{2}
		\item if $T$ is completely positive, then $ T^\star $ is also completely positive. Let $ T_2: L^2(M,\varphi)\to L^2(N,\psi) $ be the extension of the bounded linear map $ \eta_{\varphi}(m)\mapsto \eta_{\psi}(T(m)) $, then $ T^\star_2 = J_\varphi (T_2)^* J_\psi$.
		\item if $ T $ is subunital (i.e. $T(1)\leq 1$), then $ T^\star $ is weight decreasing.
		\item if $ T $ is weight decreasing, then $ T^\star $ is subunital.
		\item if $T$ is normal, then $T$ is unital (weight preserving) if and only if $ T^\star $ is weight preserving (resp. unital).
	\end{enumerate}
\end{thm}
Note that in (3), when $ T $ is completely positive, then by the Kadison-Schwarz inequality, $ T(x)^*T(x) \leq \|T(1)\|T(x^*x) $ for each $ x\in M $. For $x\in \nn_\varphi$, $ \langle \eta_\psi(T(x)), \eta_\psi(T(x)) \rangle = \psi( T(x)^*T(x) )\leq \|T(1)\|\psi(T(x^*x)) \leq C\|T(1)\| \varphi(x^*x)=C\|T(1)\| \langle \eta_\varphi(x),\eta_\varphi(x)\rangle$. Therefore, the map $ \eta_\varphi(x)\mapsto \eta_\psi(T(x)) $ indeed extends to a bounded linear map $ T_2: L^2(M,\varphi)\to L^2(N,\psi) $.
\begin{proof}
	\begin{enumerate}[(1)]
		\item For any $n\in N_+$, consider the positive linear map $ \mm_{\varphi}\ni x\mapsto [n,T(x)]_\psi $. Since for positive $x$, $ [n,T(x)]_\psi\leq \|n\|\|T(x)\|_{L^1(\psi)}=\|n\|\psi(T(x))\leq C\|n\|\varphi(x) $, by Proposition \ref{petz lemma}, there exists a $T^\star(n)\in M_+$ such that $ [n,T(x)]_\psi = [T^\star(n),x]_\varphi $ for all $x\in m_\varphi$. Clearly $ T^\star(n) $ depends linearly on $n$, and thus we can extends $ T^\star $ a positive linear map on $ N_+ $. To see that $T^\star$ is normal, we may follow \cite{AC82}: Let $ (n_i)\subseteq N_+ $ be an increasing net with $ \lim_i n_i = n $. For all $ m\in \mm_{\varphi} $,
		\begin{align*}
			[\lim_{i} T^\star(n_i),m ]_\varphi= \lim_{i} [T^\star(n_i),m]_\varphi=\lim_{i} [n_i,T(m)]_\psi= [n,T(m)]_\psi= [T^\star(n),m]_\varphi
		\end{align*}
		Since $ \mm_{\varphi} $ is dense in $L^1(\varphi)$, we obtain $\lim_{i} T^\star(n_i)= T^\star(n)$. Finally, let $ (e_i)  $ be an increasing net in $\mm_{\varphi}^+$ converging to $1$, then for $ n\in N_+ $, $ \varphi\circ T^{\star}(n) = [ 1,T^\star(n) ]_\varphi = [\lim e_i,T^\star(n)]_\varphi= \lim [T(e_i),n ]_\psi \leq \|T(1)\|\varphi(n) $.
		\item Since the biweight is symmetric, for all $ m\in \mm_{\varphi} $, $ n\in \mm_{\psi} $, we have $ [n,T^{\star\star}(m)]_\psi=[T^\star(n),m]_\varphi=[n,T(m)]_\psi $. In particular, $ T^{\star\star} = T $ on $ \mm_{\varphi} $.
		\item The completely positivity follows from direct computation, see Proposition 7 \cite{petz1984dual}. Let $\xi_1,\xi_2\in \frakA_\varphi$, $ \xi = J_\varphi ( (S_\varphi \xi_1)  \xi_2)\in J_\varphi \frakA_\varphi^2=(\frakA'_\varphi)^2 $ and $ \eta = (S_\psi\eta_1)\eta_2\in \frakA_\psi^2 $. We have
		\begin{align*}
			&\langle \xi,T^{\star}_2(\eta) \rangle = \langle  \pi_r( J_\varphi S_\varphi \xi_1 )J_\varphi\xi_2, \eta_\varphi(T^{\star}(\pi_l(\eta)) ) \rangle= \langle  \pi_r( S_\varphi^* J_\varphi \xi_1 )J_\varphi\xi_2, \eta_\varphi(T^{\star}(\pi_l(\eta)) ) \rangle = \langle J_\varphi \xi_1, T^{\star}(\pi_l(n))J_\varphi \xi_2\rangle\\
			=&[T^{\star}(\pi_l(\eta)),\pi_l( \xi_2 )^*\pi_l(\xi_1) ]_\varphi = [\pi_l(\eta),T( \pi_l( \xi_2 )^*\pi_l(\xi_1) ) ]_\psi = \langle J_\psi \eta_2,  T( \pi_l( \xi_2 )^*\pi_l(\xi_1) )J_\psi \eta_1 \rangle\\
			=& \langle J_\psi \eta_2, \pi_r(J_\psi \eta_1) T_2( (S_\varphi \xi_2 )\xi_1 ) \rangle = \langle J_\psi( (S_\psi \eta_1) \eta_2),T_2J_\varphi \xi \rangle = \langle J_\psi T_2 J_\varphi \xi,\eta  \rangle.
		\end{align*}
		Since $ (\frakA'_\varphi)^2 $ and $ \frakA_\psi^2 $ is dense in $ L^2 $, we have $ T^{\star}_2 = (J_\psi T_2 J_\varphi)^*  $.
		\item This directly follows from (1).
		\item Let $(e_i)$ be a increasing net in $ \mm_\psi^+ $ strongly converging to $1$. We have $ \|T^\star(1)\|\leq \liminf_{i }\|T^\star(e_i)\| = \liminf_{i} \sup_{n\in \mm_{\varphi}^+, \varphi(n)\leq 1}[ T^\star(e_i),n ]_\varphi = \liminf_{i} \sup_{n\in \mm_{\varphi}^+, \varphi(n)\leq 1}[ e_i,T(n) ]_\psi\leq \liminf_{i} \sup_{n\in \mm_{\varphi}^+, \varphi(n)\leq 1}\|e_i\|\psi(T(x))\leq 1. $
		\item If $ T $ is unital, let $ (e_i)  $ again be an increasing net in $\mm_{\varphi}^+$ converging to $1$, then for $ n\in N_+ $, $\varphi\circ T^{\star}(n) = [ 1,T^\star(n) ]_\varphi = [\lim e_i,T^\star(n)]_\varphi= \lim [T(e_i),n ]_\psi = [1,n]_\psi = \psi(n)$. If $T$ is weight preserving, then for all $n\in N_+$, $ [1,n]_\psi = \psi(n)= \varphi \circ T^{\star}(n)= [1,T^{\star}(n)]_\varphi = [T(1),n]_\psi $. And therefore, we must have $ T(1)=1 $. The other two implications then follows from the duality $ T^{\star\star} =T$.
	\end{enumerate}
\end{proof}

We now show the isometry $ \HHH_{T^{\star}} \simeq \overline{\HHH}_{T} $ for a normal completely positive $ T $.

\begin{prop}\label{H of T star is conjugate H T}
	Let $ T:M\to N $ be a normal completely positive map such that $ \psi\circ T\leq C\varphi $ for some $C\geq 0$. The linear map
	$$ \ell_T: n\otimes_{T^{\star}} J_\varphi \eta_{\varphi}(m)\mapsto \overline{ m\otimes_T J_\psi \eta_\psi(n) } $$
	for $n\in \nn_{\psi}$ and $m\in \nn_{\varphi}$ extends to an $N$-$M$-linear isometry from $ \HHH_{T^{\star}}  $ to $ \overline{\HHH}_{T} $.
\end{prop}
\begin{proof}
	Let $n_i\in \nn_{\psi}$ and $m_i\in \nn_{\varphi}$, we have 
	\begin{align*}
		&\langle n_1\otimes_{T^{\star} } J_\varphi \eta_{\varphi}(m_1),n_2\otimes_{T^{\star} } J_\varphi \eta_{\varphi}(m_2)\rangle = \langle J_\varphi \eta_{\varphi}(m_1),T^{\star}(n_1^*n_2 )J_\varphi \eta_{\varphi}(m_2) \rangle\\
		=&[T^{\star}(n_1^*n_2),m_2^*m_1 ]_\varphi = [n_1^*n_2,T(m_2^*m_1)]_\psi=\langle m_2\otimes_T J_\psi \eta_{\psi}(n_2), m_1\otimes_T J_\psi\eta_{\psi}(n_1) \rangle\\
		=& \langle \overline{ m_1\otimes_T \eta_{\psi}(n_1) },\overline{m_2\otimes_T \eta_{\psi}(n_2)} \rangle.
	\end{align*} 
	By the density of the linear combinations of such elements in $ \HHH_{T^{\star}} $ and $ \overline{\HHH}_{T} $, $ \ell_T  $ extends to an isometry from $ \HHH_{T^{\star}}  $ to $ \overline{\HHH}_{T} $. To see that $ \ell_T $ is $N$-linear, we observe that for $ n_0\in N $,
	$$ \ell_T( n_0n\otimes_{T^{\star}}J_\varphi \eta_{\varphi}(m)  ) = \overline{ m\otimes_{T} J_\psi n_0 \eta_{\psi}(n) }= \overline{ (m\otimes_{T} J_\psi \eta_{\psi}(n))n_0^* }= n_0\overline{ m\otimes_{T} J_\psi \eta_{\psi}(n) }. $$
	The same computation shows that $ \ell_T $ is $M$-linear.
\end{proof}

In particular, we also have $ X(\overline{\HHH}_T)\simeq X(\HHH_{T^{\star}}) \simeq X_{T^{\star}} $.

\section{Constructions of free Poisson algebras and Wick products}
	In this section, we construct the free Poisson algebra for a pseudo left Hilbert algebra $ \frakA $, or equivalently a pair $ (M,\varphi) $. When $\frakA$ is full, we always assume $ \frakA $ is a left Hilbert algebra with $(M,\varphi)$ the associated von Neumann algebra.
	
	\subsection{Construction}
	
	Consider the full Fock space $ \FF(\overline{\frakA})$. Recall that while the creation operator $ l(\xi) $ is linear on $ \xi $, the annihilation operator $ l^*(\xi)$ is antilinear on $ \xi $. For each $\xi\in \frakA$, with the help of the involution $S$, it is then therefore natural to define the (modified) creation and annihilation operator for $\xi$ as $ a^+(\xi):= l(\xi) $ and $ a^-(\xi):= (a^+(S\xi))^*=l^*(S\xi)$. Also, we define the preservation operator for $\xi$ as $a^0(\xi)=\Lambda(\pi_{l}(\xi)): f_1\otimes \cdots \otimes f_n\mapsto (\xi f_1)\otimes \cdots \otimes f_n  $. Define linear map $ X: \frakA\to B(\FF( \overline{\frakA} )) $, $$ X(\xi) = l(\xi)+ l^*(S\xi)+\Lambda(\pi_{l}(\xi))= a^+(\xi)+a^-(\xi)+a^0(\xi). $$ We have $ \|X(\xi)\|\leq \|\xi\|+\|S\xi\|+\|\pi_{l}(\xi)\| $ and $ X(\xi)^* = X(S\xi) $. If $\frakA = \frakA_\varphi$ for a n.s.f. weight, then for $ x\in \nn_\varphi\cap \nn_\varphi^* \subseteq M$, we will also write $ a^+(x):=a^+(\eta_\varphi(x)) $, $ X(x):= X(\eta_\varphi(x)) $  and etc..
	
	\begin{defn}\label{def free Poisson}
		Let $ \frakA $ be a pseudo left Hilbert algebra. The (left) \textit{free Poisson von Neumann algebra} of $\frakA$ is the von Neumann algebra $\Gamma(\frakA)$ generated by the operators $ X(\xi)= l(\xi)+ l^*(S\xi)+\Lambda(\pi_{l}(\xi)) = a^+(\xi)+a^-(\xi)+a^0(\xi)\in B(\FF(\overline{\frakA})) $ for $ \xi\in \frakA $, i.e.
		$ \Gamma(\frakA) := X(\frakA)''.$
		\newline
		If $ M $ is a von Neumann algebra with weight $\varphi$, then we also denote $ \Gamma(M,\varphi) := \Gamma(\frakA_\varphi) $. In this case, we call the linear map $ X: \nn_{\varphi}^*\cap \nn_{\varphi}\to \Gamma(M,\varphi) $ the \textit{centered free Poisson random weight} with intensity weight $ \varphi $.
	\end{defn}
	
	\begin{rmk}
		Since every element $\xi\in \frakA$ is a linear combination of two self-adjoint elements in $ \frakA $ (i.e. invariant under the involution $S$) and $X:\frakA\to \Gamma(\frakA)$ is linear, $ \Gamma(\frakA) $ is also generated by the self-adjoint elements $ \{X(\xi): S\xi=\xi \}$.
	\end{rmk}

	\begin{prop}\label{density results}
		For any pseudo left Hilbert algebra $\frakA$, $ \Gamma(\frakA) = \Gamma(\frakA'') $.
	\end{prop}
	\begin{proof}
		For any $ \xi\in \frakA'' $ with $ \|\pi_l(\xi)\|\leq 1 $, by Proposition \ref{approximate A'' by A}, there exists a sequence $ (\xi_n) $ in $\frakA$ such that $ \lim_{n\to \infty} \|\xi-\xi_n\|+\|S\xi-S\xi_n\|=0 $, $ \| \pi_{l}(\xi_n) \|\leq 1 $ and $ \pi_{l}(\xi_n) $ converges to $ \pi_{l}(\xi) $ in the $*$-strong operator topology. Therefore, we have $ \lim_{n\to \infty}\|a^+(\xi_n)+a^-(\xi_n)-a^+(\xi)-a^-(\xi)\|\leq \lim_{n\to \infty}\|\xi-\xi_n\|+\|S\xi-S\xi_n\| = 0 $, i.e. $ a^+(\xi_n)+a^-(\xi_n) $ uniformly converges to $ a^+(\xi)+a^-(\xi) $. For any $ \xi,\zeta \in \FF( \overline{\frakA} ) $ and $ \varepsilon>0 $, let $ \xi_0 = \alpha \Omega+\sum_{i=1}^{l}f^i_{1}\otimes \cdots \otimes f^i_{s_{i}} $ and $ \zeta_0 = \beta\Omega + \sum_{j=1}^{m}g^j_{1}\otimes \cdots \otimes g^j_{t_{j}} $ be vectors in $\FF_{\text{alg}}(\overline{\frakA})$ such that $ \|\xi-\xi_0\|\leq \varepsilon/ \|\zeta\| $, $\|\zeta - \zeta_0\|\leq \varepsilon/ \|\xi\| $, and $ \|\xi_0\|\leq \|\xi\| $, $ \|\zeta_0\|\leq \|\zeta\| $. Then we have
		\begin{align*}
			\langle \zeta,(a^0(\xi_n)-a^0(\xi))\xi \rangle &\leq \langle \zeta,(a^0(\xi_n)-a^0(\xi))(\xi-\xi_0)  \rangle+ \langle \zeta-\zeta_0,(a^0(\xi_n)-a^0(\xi))\xi_0  \rangle + \langle \zeta_0,(a^0(\xi_n)-a^0(\xi))\xi_0 \rangle \\
			&\leq 4\varepsilon + \sum_{i=1}^{l}\sum_{j=1}^{m} \delta_{ s_i,t_j }\langle g^j_1,(\pi_{l}(\xi_n-\xi) )f^i_1 \rangle \prod_{k=2}^{s_i}\langle g^j_{k},f^i_{k} \rangle.
		\end{align*}
		Since $ \pi_{l}(\xi_n) $ converges to $ \pi_l(\xi) $ in the $*$-strong operator topology, it also converges in the weak operator topology, and therefore we can choose a $N>0$ such that $\forall n>N $ the last term of the above inequality is less than $ \varepsilon $. This implies that $ a^0(\xi_n) $ converges to $ a^0(\xi) $ in the weak operator topology and therefore $ X(\xi_n) $ converges to $ X(\xi) $ in the weak operator topology.
	\end{proof}
	
	Note that for a n.s.f. weight $\varphi$, since $\frakA^2_\varphi$ is also a left Hilbert algebra with $ (\frakA_\varphi^2)''=\frakA_\varphi'' $, we have $ \Gamma(\frakA_\varphi^2)=\Gamma(\frakA_\varphi) $. For $ x\in \pi_{l}(\frakA^2)=(\nn_{\varphi}\cap \nn_{\varphi}^*)^2\subseteq \mm_\varphi $, $ \varphi(x) $ is well-defined. Thus if we define $ Y(x) := a^+(x)+a^-(x)+a^0(x)+ \varphi(x)= X(x)+\varphi(x), \quad x\in \mm_\varphi$ then $ \Gamma(M,\varphi)=\Gamma(\frakA_\varphi) $ is also generated by those $ Y(x) $'s with $ x\in (\nn_{\varphi}\cap \nn_{\varphi}^*)^2 $. In particular, $  \Gamma(M,\varphi) $ is generated by $ \{ Y(x):x\in \mm_{\varphi} \} $.
	
	\begin{defn}
		For a n.s.f. weight $ \varphi $ on $M$, the linear map $Y:\mm_{\varphi}\to \Gamma(M,\varphi)$,
		$$ Y(x) := a^+(x)+a^-(x)+a^0(x)+ \varphi(x)= X(x)+\varphi(x), \quad x\in \mm_\varphi, $$
		is called the \textit{free Poisson random weight} with intensity weight $ \varphi $.
	\end{defn}
	
	The proof for the following lemma is standard and is omitted.
	\begin{lem}
		For all $ \xi,\eta\in \frakA $, $ a^0(\xi)a^+(\eta)=a^+(\xi \eta) $, $ a^-(\xi)a^+(\eta)=\langle S \xi,\eta \rangle 1 = \varphi(\pi_l(\xi \eta))1 $ and $ a^-(\xi)a^0(\eta)= a^-(\xi \eta)$.
	\end{lem}
	
	Similarly, for the right pseudo Hilbert algebra $ \frakA' $, we can also define the linear map $ X_r:\frakA' \to B(\FF(\overline{\frakA})) $. For $ \eta\in \frakA' $, we define the right creation and annihilation operator $ a^+_r(\eta)\zeta_1\otimes \cdots\otimes \zeta_n=\zeta_1\otimes \cdots \otimes \zeta_n \otimes \eta $, $ a^-_r(\eta) = a^+_r(S^*_\varphi \eta)^* $, and the right preservation operator $ a^0_r(\eta)\zeta_1\otimes \cdots \otimes \zeta_n = \zeta_1\otimes \cdots \otimes (\zeta_n\eta) $. We also set $ X_r(\eta) = a^+_r(\eta)+a^-_r(\eta)+a^0_r(\eta) $. We will see later that those operators generate the commutant $ \Gamma(\frakA)' $.
	
	For $\xi_1,\cdots, \xi_n\in \frakA$, let $ \Psi(\xi_1\otimes\cdots \otimes\xi_n) \in \Gamma(\frakA)$ be the unique element such that $ \Psi(\xi_1\otimes\cdots \otimes \xi_n)\Omega = \xi_1\otimes \cdots \otimes \xi_n $. Such an element always exists since it can be constructed inductively by the rule:
	$$  \Psi(\xi_1\otimes\cdots\otimes\xi_n) = X(\xi_1)\Psi(\xi_2\otimes\cdots\otimes\xi_n)-\langle S \xi_1,\xi_2\rangle \Psi(\xi_3,\cdots,\xi_n)-\Psi(\xi_1\xi_2\otimes\xi_3\otimes\xi_n) . $$
	We call $ \Psi(\xi_1\otimes\cdots \otimes\xi_n) $ the Wick product of the tensor $ \xi_1\otimes\cdots \otimes\xi_n $. In particular, the vacuum vector $ \Omega $ is cyclic for $ \Gamma(\frakA) $. Similarly, $ \Omega $ is also cyclic for $ X_r(\frakA')'' $.
	
	\begin{lem}
		The vacuum vector $ \Omega $ is cyclic and separating for $ \Gamma(\frakA) $.
	\end{lem}
	\begin{proof}
		Since $\Omega$ is cyclic for both $ \Gamma(\frakA) $ and $ X_r(\frakA')'' $, it suffices to show that $ X(\xi) $ commutes with $ X_r(\eta)$ for all $ \xi\in \frakA $ and $ \eta\in \frakA' $. If $ \eta \in \overline{\frakA}^{\otimes n} $ for $ n\geq 2 $, then it is clear that the actions of $ X(\xi) $ and $ X_r(\eta) $ on $\eta$ are independent with each other and hence $ X(\xi)X_r(\eta)\eta = X_r(\eta)X(\xi)\eta  $. If $ \zeta\in \overline{\frakA} $, then \begin{align*}
			X(\xi)X_r(\eta)\zeta &= X(\xi)( \zeta\otimes \eta + \langle S^*_\varphi \eta,\zeta \rangle\Omega + \zeta\eta )\\
			&= \xi\otimes \zeta\otimes \eta + \langle S \xi ,\zeta \rangle \eta + \xi \zeta\otimes \eta + \langle S^* \eta, \zeta \rangle \xi + \xi\otimes \zeta\eta + \langle  S \xi,\zeta\eta \rangle\Omega + \xi \zeta \eta.  
		\end{align*}
		Similarly,
		\begin{align*}
			X_r(\eta)X(\xi)\zeta &= X_r(\eta)( \xi\otimes \zeta + \langle S \xi,\zeta \rangle\Omega + \xi \zeta )\\
			&= \xi\otimes \zeta\otimes \eta + \langle S^*_\varphi \eta ,\zeta \rangle \xi + \xi\otimes \zeta\eta + \langle S \xi, \zeta \rangle \eta + \xi \zeta\otimes \eta + \langle  S^*_\varphi \eta,\xi \zeta \rangle\Omega + \xi \zeta \eta.  
		\end{align*}
		Since $ \langle  S^*_\varphi \eta,\xi \zeta \rangle = \langle (S \xi) S^*_\varphi \eta,\zeta \rangle = \langle  S \xi,\zeta \eta \rangle $, we have $ X_r(\eta)X(\xi)\zeta = X(\xi)X_r(\eta)\zeta $. And finally,
		$$ [X(\xi),X_r(\eta)]\Omega = X(\xi)\eta - X_r(\eta)\xi = \xi\otimes \eta +\langle S \xi,\eta\rangle + \xi \eta - \xi\otimes \eta -\langle S^*_\varphi \eta,\xi\rangle - \xi \eta=0. $$
		This implies that $ X_r(\eta) $ and $ X(\xi) $ commutes and therefore $ \Omega $ is also separating for $ \Gamma(\frakA) $.
	\end{proof}
	
	To study the modular structure of $\Gamma(\frakA)$, it is convenient to have the following Wick formula.
	\begin{prop}\label{comb}
		$$ \Psi(\xi_1\otimes\cdots\otimes \xi_n)=\sum_{ s=1 }^{n+1}\left(\prod_{i\leq s-1}a^+(\xi_i)\right)\left(\prod_{j\geq s} a^-(\xi_{j})\right)+ \sum_{ s=1}^{n}\left(\prod_{i\leq s-1}a^+(\xi_i)\right)a^0(\xi_s) \left(\prod_{j\geq s+1} a^-(\xi_{j})\right),$$
		where we use the convention $ \prod_{i\leq 0}a^+(\xi_i) = 1 $, etc.
	\end{prop}
	\begin{proof}
		We observe that if we apply both sides to $ \Omega $, then we get the same vector $ \xi_1\otimes \cdots \otimes \xi_n \in \FF(\frakA) $. Since $\Omega$ is separating for $ \Gamma(\frakA) $, we only need to show that the right hand side is in indeed in $ \Gamma(\frakA) $. Note that also when $n=1$, $ \Psi(\xi_1) = X(\xi_1)=a^+(\xi_1)+a^-(\xi_1) $ and we have the recurrence relation $$  \Psi(\xi_1\otimes \cdots \otimes \xi_n) = X(\xi_1)\Psi(\xi_2\otimes \cdots\otimes\xi_n) - \langle S\xi_1,\xi_2 \rangle\Psi(\xi_3 \otimes \cdots\otimes\xi_n) - \Psi(\xi_1\xi_2\otimes \xi_3\otimes \cdots\otimes\xi_n).$$ The assertion then follows from induction and the relations $ a^+(\xi\eta) = a^0(\xi)a^+(\eta) $, $ a^0(\xi\eta) = a^0(\xi)a^0(\eta) $ and $ a^-(\xi\eta) = a^-(\xi)a^0(S\eta) $ for all $\xi,\eta\in \frakA$.
	\end{proof}
An important fact about this formula is that it does not involve inner product or multiplication, which allows us to choose morphisms that are not $*$-homomorphisms when constructing the second quantization in Section 5.

 In particular, we have $ \Psi(\xi_1,\cdots,\xi_n)^*=\Psi(S \xi_n,\cdots,S \xi_1) $ for $ \xi_i \in \frakA$.
\begin{lem}
	Let $ \frakA_\Omega = \Gamma(\frakA)\Omega $ be the left Hilbert algebra for the vacuum state $ \varphi_{\Omega} $, and $ S_\Omega = J_\Omega \Delta_\Omega^{1/2} $ the corresponding involution operator. Then $ \frakA^{\otimes n}\subseteq \frakA_{\Omega} $, $ D(S)^{\otimes n}\subseteq D(S_{\Omega}) $. Moreover, the following formulas hold:
	\begin{enumerate}[1)]
		\item $ S_\Omega(\xi_1\otimes \cdots \otimes \xi_n)= (S \xi_n)\otimes \cdots \otimes (S \xi_1) $ for all $ \xi_i\in D(S) $.
		\item $\Delta_\Omega \supseteq \FF(\Delta)$, i.e. $ \Delta_\Omega( \xi_1\otimes \cdots \otimes \xi_n ) = (\Delta \xi_1)\otimes \cdots \otimes (\Delta \xi_n) $ for all $ \xi_i\in D(\Delta) $.
		\item $ J_\Omega(\xi_1\otimes \cdots \otimes \xi_n)= (J \xi_n)\otimes \cdots \otimes (J \xi_1) $ for all $ \xi_i\in \overline{\frakA} $.
	\end{enumerate}
\end{lem}
\begin{proof}
	1) For all $\xi_i\in \frakA$, $S_\Omega(\xi_1\otimes \cdots \otimes \xi_n)= S_\Omega\Psi(\xi_1,\cdots,\xi_n)\Omega = \Psi(\xi_1,\cdots,\xi_n)^* \Omega =\Psi(S \xi_n,\cdots,S \xi_1)\Omega = (S \xi_n)\otimes \cdots \otimes (S \xi_1)$, hence the formula follows from the fact that $ \frakA$ is a core of $S$.
	
	2) Similarly, for all $ g_i\in D(S^*) $, since $ X_r(g_i)\in \Gamma(\frakA)' $, using the right version of the argument in 1), we obtain that $ g_1\otimes \cdots \otimes g_n\in D(S^*) $, and $ S^*_\Omega(g_1\otimes \cdots \otimes g_n)= (S^* g_n)\otimes \cdots \otimes (S^* g_1)$. The statement now follows from the fact that $ \Delta_\Omega$ is the closure of $S^*_\Omega S_\Omega $.
	
	3) The formula for $ J_\Omega $ follows directly from $ J_\Omega \supset S_\Omega \Delta_\Omega^{-1/2} $.
\end{proof}

In particular, we have $ X_r(\eta) = J_\Omega X(J \eta )J_\Omega $ for all $ \eta\in \frakA' $ and $ \Gamma(\frakA)' = X_r(\frakA')'' $.
\begin{cor}
	If $\varphi$ is a tracial n.s.f. weight, then $ \varphi_\Omega $ is a tracial state on $ \Gamma(M,\varphi) $.\qed
\end{cor}

\begin{rmk}
	Note that if $\varphi$ is a normal semifinite but non-faithful weight on $M$, then we can still define $ X(x) := a^+(x)+a^-(x)+a^0(x) = l(\eta_{\varphi}(x))+l^*(\eta_{\varphi}(x^*))+\Lambda(\pi_{l}(x)) $ for $x\in \nn_{\varphi}\cap \nn_{\varphi}^*$ and set $\Gamma(M,\varphi) = X( \nn_{\varphi}\cap \nn_{\varphi}^* )''$. In this case, however, $ \Omega $ is not separating for $ \Gamma(M,\varphi) $ in general. If we assume $\varphi$ is tracial, then considering the non-faithful situations will not give us any new algebras, since we can simply restrict $ \varphi $ to the subalgebra $ Mz $ with $ z\in \mathcal{Z}(M) $ the support of $ \varphi $. In this paper, we will mainly focus on the cases when $ \varphi $ is faithful.
\end{rmk}

\subsection{Free cumulants of $ X(u) $'s}
Applying Theorem \ref{cumulants of operators on Fock spaces}, we immediately get:
\begin{prop}\label{cumulants}
	Let $ \frakA $ be a pseudo left Hilbert algebra, and $ (\xi_i)_{i\in I} $ be a family of elements in $ \frakA $. The joint free cumulants of $X(\xi_i)_{i\in I}$ are
	$$ R_1(X(\xi_i))=0,\quad R_n(X(\xi_{i_1}),\cdots,X(\xi_{i_n})) = \langle S_\varphi \xi_{i_1},\xi_{i_2}\cdots \xi_{i_{n-1}}\xi_{i_n} \rangle,\quad \forall n\geq 2. $$
	In particular, if $\frakA=\frakA_\varphi$, and $ (x_i)_{i\in I}$ is a family of elements in $\nn_{\varphi}\cap \nn_{\varphi}^*$, then the joint free cumulants of $ {X(x_i)}_{i\in I} $ are
	$$  R_1(X(x_{i}))=0,\quad R_n(X(x_{i_1}),\cdots,X(x_{i_n})) = \varphi( x_{i_1} x_{i_2}\cdots x_{i_n} ),\quad \forall n\geq 2.$$
Similarly, if $ x_i \in \mm_{\varphi} \subseteq \nn_{\varphi}\cap \nn_{\varphi}^* $, let again $ Y(x_i) = X(x_i)+\varphi(x_i) $ then the free cumulants of $ Y(x_i)=Y(\eta_{\varphi}(x_i)) $ are
$$ R_n(Y(x_{i_1}),\cdots,Y(x_{i_n})) = \varphi(x_{i_1}x_{i_{2}}\cdots x_{i_n}),\quad \forall n\geq 1.$$ \qed
\end{prop}

\begin{cor}\label{direct sum}
	If $ \frakA_1 $ and $ \frakA_2 $ are two pseudo left Hilbert algebras, then $ \Gamma( \frakA_1\oplus \frakA_2 ) = \Gamma(\frakA_1)\ast \Gamma(\frakA_2) $ with respect to the vaccum state.
\end{cor}
\begin{proof}
	For a family of elements $ (X_i)_{i=1}^{n} $ in $ X(\frakA_1)\cup X(\frakA_2) $, Proposition \ref{cumulants} implies that the free cumulant $ R_n( X_1,\cdots,X_n ) $ is $0$ whenever $ (X_i)_{i=1}^n $ contains an element in $X(\frakA_1)$ and an element in $ X(\frakA_2) $. Thus, the two families $ X(\frakA_1) $ and $ X(\frakA_2) $ are freely independent and hence the von Neumann algebras $ \Gamma(\frakA_1)= X(\frakA_1)'' $ and $\Gamma(\frakA_2)= X(\frakA_2)''$ are freely independent with respect to the vacuum state.
\end{proof}

	\begin{cor}\label{multi Levy-Khin}
	If $ \frakA $ is a pseudo left Hilbert algebra such that $ \frakA^2 $ is not dense in $ \overline{\frakA} $, then $\Gamma(\frakA)$ is the free product of the free Poisson von Neumann algebra $ \Gamma(\frakA^2) $ and the free Araki-Woods algebra $ \Gamma\left( (\frakA^2)^\perp \cap D(S) \right) $,
	$$ \Gamma(\frakA)=\Gamma(\frakA'') = \Gamma\left( (\frakA^2)'' \oplus ( (\frakA^2)^\perp \cap D(S) ) \right)= \Gamma(\frakA^2)\ast \Gamma\left( (\frakA^2)^\perp \cap D(S) \right).$$
\end{cor}
\begin{proof}
	Note that $\Gamma\left( (\frakA^2)^\perp \cap D(S) \right)$ is generated by $ X(\xi) = a^+(\xi)+a^-(\xi)+0= l(\xi)+l^*(S\xi) = l(\xi)+l^*(\xi)$ for $\xi=S\xi\in (\frakA^2)^\perp\cap D(S) $. Consider the real Hilbert space $ \KK=\{\xi\in (\frakA^2)^\perp \cap D(S): \xi=S\xi\} $, then $  \Gamma\left( (\frakA^2)^\perp \cap D(S) \right) $ is exactly the free Araki-Woods algebra $ \Gamma_{AW}(\KK \subset (\frakA^2)^\perp) $ for the embedding $ \KK \subset (\frakA^2)^\perp $ as defined in \cite{Shl97}. The statement now follows from Proposition \ref{decomposition of pseudo hilbert algebra} and Corollary \ref{direct sum}.
\end{proof}

We note that a similar decomposition also occurred in \cite{AM21} Theorem 43, where they considered a more general Fock space over a fixed $*$-probability space $ (\mathcal{B},\varphi) $ while our construction is based on left Hilbert algebras.

The next example shows that the free Poisson random weight gives us natural examples of $R$-cyclic families.
\begin{eg}
	Let $ \varphi $ be a n.s.f weight on $ M $ and $ \{p_{i}\}_{i=1}^n $ be orthogonal projections in $M$ with finite weights. For each family $\{x_{k,ij}\}_{1\leq i,j\leq n,k\in I}$ with $x_{k,ij}\in p_i Mp_j\in \mm_{\varphi} $, consider the $ \Gamma(M,\varphi) $ valued matrices
	
	$$ y_k := (Y( x_{k,ij} ))_{i,j}\in \Gamma(M,\varphi) \otimes M_{n}(\CC). $$ Then $\{y_k\}_{k\in I}$ is a R-cyclic family \cite{NSS02} with respect to the state $ \varphi_{\Omega}\otimes \text{tr}_n$. This is because the free cumulants $$ R_m(Y(x_{k_1,i_1j_1}), \cdots, Y(x_{k_m,i_mj_m})) = \varphi( x_{k_1,i_1j_1}\cdots x_{k_m,i_mj_m} )  $$ are non-zero only if $ j_s=i_{s+1} $ for $ s=1,\cdots,m-1 $.
\end{eg}

\subsection{Wick product and a Haagerup type inequality for finite weights}
 Recall that from Proposition \ref{comb}, we have the Wick product formula
\begin{align*}
	\Psi(\xi_1\otimes \cdots\otimes \xi_n)&=\sum_{ s=1 }^{n+1}\left(\prod_{i\leq s-1}a^+(\xi_i)\right)\left(\prod_{j\geq s} a^-(\xi_{j})\right)+ \sum_{ s=1}^{n}\left(\prod_{i\leq s-1}a^+(\xi_i)\right)a^0(\xi_s) \left(\prod_{j\geq s+1} a^-(\xi_{j})\right)\\
	&= \sum_{ s=1 }^{n+1}\left(\prod_{i\leq s-1}\ell(\xi_i)\right)\left(\prod_{j\geq s} \ell^*(S_\varphi \xi_{j})\right)+ \sum_{ s=1}^{n}\left(\prod_{i\leq s-1}\ell (\xi_i)\right)\Lambda(\pi_l(\xi_s)) \left(\prod_{j\geq s+1} \ell^*(S_\varphi \xi_{j})\right)
\end{align*}
for $ \xi_i \in \frakA_\varphi$. From this, one can also show the following formula proved in \cite{BP14} for the commutative cases.

\begin{prop}\label{multiplication of Wick product}
	For all $\xi_i,\eta_j\in \frakA$, \begin{align*}
		\Psi(\xi_n\otimes \cdots\otimes \xi_1)\Psi(\eta_1\otimes \cdots \otimes \eta_m)=&\sum_{k=0}^{n\wedge m} \left(\prod_{i=1}^{k}\langle S\xi_i,\eta_i \rangle\right) \Psi( \xi_n\otimes \cdots \otimes \xi_{k+1}\otimes \eta_{k+1}\otimes\cdots \otimes \eta_{m} )\\
		&+\sum_{k=0}^{n\wedge m-1}  \left(\prod_{i=1}^{k}\langle S\xi_i,\eta_i \rangle\right) \Psi( \xi_n\otimes \cdots \otimes \xi_{k+1}\eta_{k+1}\otimes\cdots \otimes \eta_{m} )
	\end{align*}
\end{prop}
\begin{proof}
	By the Wick product formula, we obtain
	\begin{align*}
		&\Psi(\xi_n\otimes \cdots\otimes \xi_1)\Psi(\eta_1\otimes \cdots \otimes \eta_m)\Omega =\Psi(\xi_n\otimes \cdots\otimes \xi_1)\eta_1\otimes \cdots \otimes \eta_m\\ =&\sum_{k=0}^{n\wedge m} \left(\prod_{i=1}^{k}\langle S\xi_i,\eta_i \rangle\right) \xi_n\otimes \cdots \otimes \xi_{k+1}\otimes \eta_{k+1}\otimes\cdots \otimes \eta_{m} +\sum_{k=0}^{n\wedge m-1}  \left(\prod_{i=1}^{k}\langle S\xi_i,\eta_i \rangle\right) \xi_n\otimes \cdots \otimes \xi_{k+1}\eta_{k+1}\otimes\cdots \otimes \eta_{m}.
	\end{align*} 
	The statement now follows from the separability of $ \Omega $.
\end{proof}

Similar to Wick product for free Gaussian algebras (free Araki-Woods algebras), Wick product for free Poisson algebras is also closely related to Haagerup type inequality. For the rest of this subsection, we focus on free Poisson algebras over a (actual) left Hilbert algebra $\frakA_\varphi$ associated a finite weight $ \varphi $ on $ M $.

Assume that $ \varphi(1)<\infty $, so that $ \xi := \varphi^{1/2} \in L^2(M)$ is a vector with norm $ \|\xi\|=\varphi(1)^{1/2} $. If we consider the composition of $\Psi$ and the embedding $ M^{\odot n}\ni x\mapsto x\xi^{\otimes n} $, then we obtain the linear map $I_n:M^{\odot n} \to \Gamma(M,\varphi)$ (we will see later that $I_n$ actually extends to the whole von Neumann algebra $M^{\otimes n}$). A (operator-valued) Haagerup type inequality for the free Poisson algebra can then be considered as the completely boundedness of $ I_n $ with $ \|I_n\|_{cb}\leq Cn $. To show this, we will follow the notation about operator spaces in \cite{nou2004non}.

For a complex Hilbert space $\HHH$, we will denote $ \HHH_c := B( \CC,\HHH )$ the column operator space and $ \HHH_r:= B( \overline{\HHH},\CC ) $ the row operator space (where $ \overline{\HHH} $ is the conjugate Hilbert space of $\HHH$). Recall that the column and row operator spaces have the following properties under the Haagerup tensor product: if $ \HHH,\KK $ are two Hilbert space and $ S $ another operator space, then we have
$$ \HHH_c\otimes_h S = \HHH_c \otimes_{\min} S,\quad S\otimes_h \KK_r = S\otimes_{\min} \KK_r,\quad   \HHH_c\otimes_h \KK_r = \HHH_c\otimes_{\min} \KK_r = B( \KK,\HHH ).$$
Also, if one denote $ \HHH\otimes_2 \KK $ the usual tensor product of Hilbert spaces, then one also has $$ \HHH_c\otimes_h \KK_c = \HHH_c\otimes_{\min} \KK_c = (\HHH\otimes_2 \KK)_c,\quad \HHH_r\otimes_h \KK_r = \HHH_r\otimes_{\min} \KK_r = (\HHH \otimes_2 \KK)_r. $$

It is shown in \cite{nou2004non} that, for each $ k\geq 1 $, the creation operator $ \ell_k :=\ell: (\HHH^{\otimes k})_c\to B(\FF(\HHH)) $,
$$ \ell(\xi_1\otimes\cdots \otimes \xi_k) := \ell(\xi_1)\ell(\xi_2 )\cdots \ell(e_k),\quad \forall \xi_i\in \HHH $$
is a complete contraction $ \|\ell \|_{cb}\leq 1 $. And similarly, the annilation operator $ \ell^*_k:= \ell^*: (\overline{\HHH}^{\otimes k})_r\to B(\FF_0(\HHH)) $,
$$ \ell^*( \overline{\xi_1}\otimes \cdots \otimes \overline{\xi_k} ):= \ell^*( \xi_1 )\ell^*(\xi_2) \cdots \ell^*(\xi_k), $$
is a complete contraction $ \|\ell^*\|_{cb}\leq 1 $.

Let $ \mathcal{M}: B(\FF_0(\HHH))\otimes_{h} B(\FF_0(\HHH)) $ be the multiplication operator $ \mathcal{M}(A\otimes B):= AB $, then we have $ \|M\|_{cb}\leq 1 $ (See \cite{pisier2003introduction} Corollary 5.4).

We also need to consider another completely bounded map: the embedding of $ M $ into $ L^2(M,\varphi) $. For this, we note that the linear map $$ R: M\to L^2(M)_c = B( \CC,L^2(M) ),\quad 
x\mapsto x\xi $$ can be considered as the restriction of $ M\subseteq B(L^2(M)) $ to the subspace $ \CC\xi \subseteq L^2(M) $ (multiplied by the constant $ \|\xi\|$). But since restriction is always a complete contraction, we have $ \|R\|_{cb}\leq \|\xi\| $.
Similarly, we can consider the corestriction of $ M\subseteq B(L^2(M)) $ to $ \CC\xi $,
$$ cR: M\to \overline{L^2(M)}_r = B(L^2(M),\CC), \quad x\mapsto \overline{x^*\xi} $$
(note that for any $\eta\in \HHH$, we have $ \langle x^*\xi, \eta\rangle = \langle \xi,x\eta \rangle  $, so when considered as an operator in $ B(L^2(M),\CC) $, $ \overline{x^*\xi} $ is indeed the corestriction of $ x $ to $ \CC\xi $, multiplied again by $ \|\xi\| $.) In particular, we also have $ \|cR\|_{cb}\leq \|\xi\| $.

We can finally rewrite the Wick formula in terms of those complete bounded maps:
let $ x_1,\cdots ,x_n\in M $, then
\begin{align*}
	\varPsi( x_1\xi \otimes \cdots \otimes  x_n\xi ) =& \sum_{s=0}^{n}\mathcal{M}(\ell_s\otimes \ell^*_{n-s})( R^{\otimes s} \otimes cR^{\otimes (n-s)} )( x_1\otimes \cdots \otimes x_n )\\
	&+ \sum_{s=1}^{n}\mathcal{M}(\text{id}\otimes M)( \ell_{s-1}\otimes \text{id} \otimes \ell^*_{n-s} )( R^{\otimes s-1} \otimes \Lambda \otimes cR^{\otimes (n-s)} )( x_1\otimes \cdots \otimes x_n ).
\end{align*}

The following corollary then follows immediately from the formula above.
\begin{cor}
	Suppose that $ \varphi(1)< \infty $ and $ \xi:= \varphi^{1/2}\in L^2(M) $. Consider the linear map $ I_n:  M^{ \odot n }\to \Gamma(M,\varphi) $ defined on the algebraic tensor power of $ M $,
	$$ I_n( x_1\otimes \cdots \otimes x_n ):= 	\varPsi( x_1\xi \otimes \cdots \otimes  x_n\xi ), $$
	then $ I_n $ can be extended to a normal completely bounded map from the von Neumann algebra $ M^{\otimes n} $ to $ \Gamma(M,\varphi)$:
	$$ I_n:  M^{\otimes n} \to \Gamma(M,\varphi)\subseteq  B(\FF_0(L^2(M))), $$
	with $ \|I_n\|_{cb} \leq (n+1)\varphi(1)^{n/2}+n\varphi(1)^{(n-1)/2} $.
\end{cor}
\begin{proof}
	a) We first show that $ I_n $ is completely bounded on the minimal tensor product $ M^{\otimes_{\min} n} $.
	For this, we simply note that  $ R^{\otimes s-1} \otimes \Lambda \otimes cR^{\otimes (n-s)} $ (similarlly for $R^{\otimes s} \otimes cR^{\otimes (n-s)}$) can be considered as the cb map on the minimal tensor product $ M^{\otimes_{\min} n} \to ( L^2(M)^{\otimes (s-1)})_c \otimes_{\min} B(\FF_0(\HHH)) \otimes_{\min} (L^2(M)^{\otimes (n-s)})_r$. Identifying $ ( L^2(M)^{\otimes (s-1)})_c \otimes_{\min} B(\FF_0(\HHH)) \otimes_{\min} (L^2(M)^{\otimes (n-s)})_r $ with the Haagerup tensor product $ ( L^2(M)^{\otimes (s-1)})_c \otimes_{h} B(\FF_0(\HHH)) \otimes_{h} (L^2(M)^{\otimes (n-s)})_r $, we then consider $\ell_{s-1}\otimes \text{id} \otimes \ell^*_{n-s}$ as the cb map defined on the Haagerup tensor product and finally apply the boundedness of $ M $ on the Haagerup tensor product. In particular, the extension $I_n:M^{\otimes_{\min} n}\to \Gamma(M,\varphi)$ is well defined as $$ I_n:=  \sum_{s=0}^{n}\mathcal{M}(\ell_s\otimes \ell^*_{n-s})( R^{\otimes s} \otimes cR^{\otimes (n-s)} )+ \sum_{s=1}^{n}\mathcal{M}(\text{id}\otimes M)( \ell_{s-1}\otimes \text{id} \otimes \ell^*_{n-s} )( R^{\otimes s-1} \otimes \Lambda \otimes cR^{\otimes (n-s)} ).$$
	
	b) We now show that $I_n$ can be extended to a normal completely bounded map on $ M^{\otimes n} $. Indeed, for a given faithful normal positive linear functional $ \psi $ on a von Neumann algebra $ N $, the strong operator topology on the unit ball $N_1$ coincides with the $L^2$-topology (under the identification $ N\ni x\mapsto x\psi^{1/2}\in L^2(N) $). Since $I_n: M^{\odot n}\to \Gamma(M,\varphi) $ preserves the $L^2$-norm by the definition of Fock space, $ I_n $ is continuous on $ (M^{\odot n})_1 $ under the strong operator topology. As $ (M^{\odot n})_1 $ is strong operator dense in $ (M^{\otimes n})_1 $, we can extend $I_n$ to $ M^{\otimes n} $. Finally, we note that $ I_n $ is normal since a linear map between von Neumann algebras is normal if and only if it is strong operator continuous on the unit ball. To see that $I_n$ is completely bounded on $M^{\otimes n}$, one simply evoke the Kaplansky density theorem: for any $ X=(x_{ij})\in M^{\otimes n}\otimes M_d(\CC) $ with $ \|X\|\leq 1 $, we pick a net $ X^{(s)}= (x^{(s)}_{ij})\in M^{\odot n}\otimes M_d(\CC) $ converging strongly to $X$ with $ \|X^{(s)}\|\leq 1 $. We now have $ (I_n(x_{ij}))\in \Gamma(M,\varphi)\otimes M_d(\CC) $ is the strong limit of $ (I_n(x^{(s)}_{ij})) $. In particular, $ \|  (I_n(x_{ij}))\| \leq \sup_{s} \| (I_n(x^{(s)}_{ij})) \| \leq (n+1)\varphi(1)^{n/2}+n\varphi(1)^{(n-1)/2} $.
\end{proof}

Since $I_n$ is normal, we can also extends the formula in Proposition \ref{multiplication of Wick product} to $ M^{\otimes n} $.

\begin{cor}
	Assume that $\varphi(1)< \infty$. Then for all $ x\in M^{\otimes n} $ and $y\in M^{\otimes m}$, we have
	$$ I_n(x)I_m(y)= \sum_{k=0}^{n\wedge m} I_{m+n-2k}(c_{k}(x,y))+ \sum_{k=0}^{n\wedge m-1}I_{m+n-2k-1}(m_k(x,y)), $$
	where $ c_k: M^{\otimes n}\times M^{\otimes m}\to M^{\otimes (m+n-2k)} $ is the contraction operator
	$$ c_k(x_n\otimes \cdots \otimes x_1,y_1\otimes \cdots \otimes y_m):= \left(\prod_{i=1}^k \varphi( x_{i}y_{i} )\right) x_n\otimes \cdots \otimes x_{k+1} \otimes y_{k+1}\otimes \cdots \otimes y_m, $$
	and $ m_k: M^{\otimes n}\times M^{\otimes m}\to M^{\otimes (m+n-2k-1)} $ is the operator
	$$ m_k(x_n\otimes \cdots \otimes x_1,y_1\otimes \cdots \otimes y_m):= \left(\prod_{i=1}^k \varphi( x_{i}y_{i} )\right) x_1\otimes \cdots \otimes x_{k+1}y_{k+1}\otimes \cdots \otimes y_m.$$
	In particular, $ \text{span}\{ I_k(x):x\in M^{\otimes k},k\geq 0 \} $ is a dense $*$-subalgebra of $ \Gamma(M,\varphi) $.\qed
\end{cor}
\begin{proof}
	For $x\in M^{\odot n},y\in M^{\otimes m}$ the statement follows directly from Proposition \ref{multiplication of Wick product}. By density of $ M^{\odot n} $ in $ M^{\otimes n} $ and since $I_n$ is normal, it suffices to show that the $c_k$ and $m_k$ operators are normal (i.d. continuous w.r.t. the $\sigma$-weak topology). For $ 0\leq k\leq n\wedge m -1 $, $ m_{k}: M^{\otimes n}\times M^{\otimes m} \to M^{n+m-2k-1}$ can be written as the following composition map:
	\begin{align*}
		&M^{\otimes n}\times M^{\otimes m} \xrightarrow{ \text{id}\times (F_{k+1}\otimes \text{id}^{m-k-1}) } M^{\otimes n}\times M^{\otimes m} \simeq (M^{\otimes n} \otimes \CC^{\otimes m-k-1})\times (\CC^{\otimes n-k-1}\otimes M^{\otimes m})\\ &\xhookrightarrow{} M^{\otimes m+n-k-1}\times M^{\otimes m+n-k-1} \xrightarrow{\mathcal{M}} M^{\otimes m+n-k-1} \xrightarrow{ \text{id}^{ n-k}\otimes \varphi^{\otimes k}\otimes \text{id}^{m-k-1} }M^{\otimes m+n-2k-1},
	\end{align*}
	where $ F_{k+1}\otimes \text{id}^{m-k-1} $ is the map reversing the order of the first $ (k+1) $ components in $M^{\otimes n}$.
	Obviously, $ m_k $ is normal and $ \|m_k(x, y) \|\leq \|x\|\|y\| $.
	Similarly, we can also decompose $c_k: M^{\otimes n}\times M^{\otimes m} \to M^{n+m-2k}$ as
	\begin{align*}
		&M^{\otimes n}\times M^{\otimes m} \xrightarrow{ \text{id}\times (F_{k}\otimes \text{id}^{m-k}) } M^{\otimes n}\times M^{\otimes m} \simeq (M^{\otimes n} \otimes \CC^{\otimes m-k})\times (\CC^{\otimes n-k}\otimes M^{\otimes m})\\ &\xhookrightarrow{} M^{\otimes m+n-k}\times M^{\otimes m+n-k} \xrightarrow{\mathcal{M}} M^{\otimes m+n-k} \xrightarrow{ \text{id}^{ n-k}\otimes \varphi^{\otimes k}\otimes \text{id}^{m-k} }M^{\otimes m+n-2k}.
	\end{align*} 
	
	Therefore, all of the $c_k$'s, $m_k$'s, and $ I_k $'s are normal, and we can extend the product formula of Wick polynomials to $M^{\otimes n}$.
\end{proof}

\section{Isomorphism class of free Poisson algebras for finite weights}

Recall that a free Poisson distribution $ \pi_\lambda $ of rate $ \lambda>0 $ is a distribution with free cumulants $ R_n = \lambda $ for all $ n\geq 1$. The density of $ \pi_\lambda $ is
$$ d\pi_\lambda(x) = \frac{\sqrt{4\lambda - (x-(\lambda+1))^2}}{4\pi x}\chi_{[(\sqrt{\lambda}-1)^2, (\sqrt{\lambda}+1)^2 ]}(x)dx+\max{(1-\lambda,0)}\delta_0.$$
The support of $ \pi_\lambda $ is $ [(\sqrt{\lambda}-1)^2, (\sqrt{\lambda}+1)^2 ] $.

The \textit{free compound Poisson distribution} $ \pi_{\nu,\lambda} $ of rate $\lambda$ with respect to another distribution $ \nu $ is defined using the free convolution
$$ \pi_{\nu,\lambda}:= \lim_{n\to \infty}\left( (1-\frac{\lambda}{n})\delta_0+\frac{\lambda}{n}\nu) \right)^{ \boxplus n}. $$ When $ \nu $ has finite moments, the free cumulants of $ \pi_{\nu,\lambda} $ is exactly $\lambda$ times the moments of $ \nu $. Also, if $ \nu $ is a positive distribution (i.e. supported in $\RR_+$), then $ \pi_{\nu,\lambda} $ is also positive. Corollary \ref{cumulants} implies that if $ \varphi $ is a state, and the distribution of $ u\in M $ is $ \nu $, then the distribution of $ Y(u) $ is compound Poisson distribution $ \pi_{\nu,1} $ of rate $1$, while the distribution of $ X(u) $ is the centered compound Poisson distribution. Similarly, free compound Poisson distribution with rate $ \lambda\neq 1 $ can be obtained by considering the free Poisson algebra of the positive linear functional $ \lambda \varphi $ on $ M $.

While $ Y(u) $ is a realization of compound free Poisson distribution, there is another well-known realization for $\lambda=1$: let $ s $ be a self-adjoint variable freely independent from $u$ such that $s^2$ has distribution $ \pi_1 $, then $sus$ also has the compound free Poisson distribution $\pi_{\nu,1}$. From this, one easily can show that when $\varphi$ is a state, the two families $ \{Y(u)\}_{u\in M} $ and $ \{sus\}_{u} $ has the same joint distribution, and therefore we have the isomorphism $\Gamma(M,\varphi) \simeq L(\ZZ)\ast (M,\varphi)$. In this section, we generalize this observation to the cases when $ \varphi $ is a positive linear functional and determine the isomorphism class of $ \Gamma(M,\varphi ) $ by considering free Poisson distributions $ \pi_\lambda $ with $ \lambda\neq 1 $.

If $ \varphi $ is a faithful normal state on $ M $ and $ \alpha >0 $, we consider the rescaled positive linear functional $ \alpha\varphi $ on $M$ (every finite weight can be written in this form) and denote $ Y_\alpha:= Y_{(M,\alpha\varphi)}:M\to \Gamma(M,\alpha \varphi)  $.

\begin{lem}\label{support of Y}
	If $\varphi$ is a faithful normal state and  $ 0<\alpha<1 $, then the support $ z_\alpha = s(Y_\alpha(1))\in \Gamma(M,\alpha \varphi) $ is a central projection in $ \Gamma(M,\alpha \varphi) $ with $ \varphi_\Omega(z_\alpha)=\alpha $ and $ \Gamma(M,\alpha \varphi) = \Gamma(M,\alpha \varphi)z_\alpha \oplus \CC (1-z_\alpha) $.
\end{lem}
\begin{proof}
	Since $ Y_\alpha(M)\cup \{1\} $ generate $\Gamma(M,\alpha \varphi)$, we need to show that for all $ x\in M_+ $, $ s(Y_\alpha(x))\leq z_\alpha $. Let $ C = \|x\| $ and $ y = C -x $. Since $ Y_\alpha(x)+Y_\alpha(y)=Y_\alpha(C)= CX_\alpha(1) $ and $ Y_\alpha(x)\sim \pi_{\nu_{x},1} $ and $ Y_\alpha(y)\sim \pi_{\nu_{y},1} $ are positive elements, we have $ Y_\alpha(x)\leq C Y_\alpha(1) $ and hence $ s(Y_\alpha(x))\geq s(Y_\alpha(1))=z_\alpha $.
\end{proof}

If $ (A,\varphi) $ and $ (B,\psi) $ are two non-commutative probability space and $ 0\leq s,t\leq 1 $ such that $ s+t=1 $, we will use the notation
$$ \overset{p}{\underset{t}{A}}\oplus \overset{q}{\underset{s}{B}}$$
to denote the direct sum algebra $ A\oplus B $ equipped with state $ \varphi_0( a+b ) = s\varphi(a)+t\psi(b)  $ for all $ a\in A $ and $b\in B$. Here $ p $ and $q$ are the identities $1_A$ and $ 1_B $ respectively.

\begin{thm}\label{rescaling}
	Let $ (M,\varphi) $ be a von Neumann algebra with a faithful normal state. For $0<\alpha \leq 1$, let $ s$ be a generating positive element of $ L(\ZZ) $ such that $ s^2 $ has the free Poisson distribution $ \pi_{1/\alpha} $, then the linear map
	$$ i: Y_\alpha(u)z_\alpha \mapsto  \alpha sus,\quad \forall u\in M $$
	extends to a $*$-isomorphism $ i: \Gamma(M,\alpha\varphi)z_\alpha \to L(\ZZ)\ast(M,\varphi) $. In particular, we have
	$$ \Gamma(M,\alpha \varphi) \simeq \left(\underset{\alpha}{L(\ZZ)\ast (M,\varphi)}\right)\oplus \underset{1-\alpha}{\CC}.$$
\end{thm}
\begin{proof}
	For $ u_i\in M_{s.a.} $, the joint moments of $ (Y(u_i)) $ in $ \Gamma(M,\alpha\varphi) $ (with respect to the vacuum state $\langle \Omega, \cdot\Omega\rangle$) are
	\begin{align*}
		&M_{n,\Gamma(M,\alpha\varphi)}( Y(u_1),\cdots,Y(u_k) )=\sum_{\pi\in \mathcal{NC}(n)}R_{\pi,\Gamma(M,\alpha\varphi) }( Y(u_1),\cdots,Y(u_n) )\\ =& \sum_{\pi\in\mathcal{NC}(n)}M_{\pi,\alpha \varphi}( u_1,\cdots, u_n ) = \sum_{\pi\in\mathcal{NC}(n)}\alpha^{|\pi|}M_{\pi,\varphi}( u_1,\cdots, u_n )
	\end{align*}
	Now consider the state $ \frac{1}{\alpha}\langle \Omega, \cdot\Omega\rangle $ on $ \Gamma(M,\alpha \varphi)z_\alpha $. The joint moments of $ (Y(u_i)z_\alpha) $ with respect to this state are
	\begin{align*}
		M_{n,\Gamma(M,\alpha\varphi)z_\alpha}( Y(u_1)z_\alpha,\cdots,Y(u_n)z_\alpha )&= \frac{1}{\alpha}\langle \Omega,Y(u_1)\cdots Y(u_n)\Omega \rangle\\
		&= \frac{1}{\alpha}M_{n,\Gamma(M,\alpha\varphi)}( Y(u_1),\cdots,Y(u_n) )\\
		&=\sum_{\pi\in \mathcal{NC}(n)}\alpha^{|\pi|-1}M_{\pi,\varphi}( u_1,\cdots, u_n ).
	\end{align*}
	On the other hand, since $ s $ is in the centralizer of $ L(\ZZ)\ast (M,\varphi) $, applying Theorem \ref{thm moments of product}, we get the joint moments of $ \alpha su_is $ in $ L(\ZZ)\ast (M,\varphi) $
	\begin{align*}
		M_{n,L(\ZZ)\ast (M,\varphi)}(\alpha su_1s,\cdots,\alpha su_ns) =& \alpha^n M_{n,L(\ZZ)\ast (M,\varphi)}(u_1s^2,\cdots,u_ns^2) \\=& \alpha^n \sum_{\pi \in \mathcal{NC}(n)}M_{\pi,\varphi}(u_1,\cdots,u_n)R_{K(\pi)}(s^2,\cdots,s^2)\\
		=&\alpha^n \sum_{\pi \in \mathcal{NC}(n)}M_{\pi,\varphi}(u_1,\cdots,u_n)\frac{1}{\alpha^{|K(\pi)|}}\\
		=&\sum_{\pi \in \mathcal{NC}(n)}M_{\pi,\varphi}(u_1,\cdots,u_n)\alpha^{n-(n+1-|\pi|)}\\
		=&\sum_{\pi \in \mathcal{NC}(n)}M_{\pi,\varphi}(u_1,\cdots,u_n)\alpha^{|\pi|-1}.
	\end{align*}
This shows that the self-adjoint families $ (X(u_i)z_\alpha) $ and $ (\alpha su_is) $ have the same distribution and hence $ i: X(u)z_\alpha \mapsto \alpha sus $ extends to a $*$-isomorphism $ \Gamma(M,\alpha\varphi)z_\alpha \to L(\ZZ)\ast(M,\varphi) $. (Note that $ i( \Gamma(M,\alpha\varphi)z_\alpha )$ is exactly $L(\ZZ)\ast(M,\varphi) $. Indeed, we have $ s1s=s^2$ and since $ s $ is atomless, $ s^{-1} $ is affiliated with $ i(\Gamma(M,\alpha\varphi)z_\alpha)\cap L(\ZZ) $, and hence $ u = s^{-1}sus s^{-1} \in i(\Gamma(M,\alpha\varphi)z_\alpha)$.) 
\end{proof}

Similarly, when $ \alpha >1 $ we have another type of isomorphism for $ \Gamma(M,\alpha \varphi) $.
\begin{thm}\label{rescaling2}
		Let $ (M,\varphi) $ be a von Neumann algebra with a faithful normal state. For $\alpha > 1$, let $ s$ be a generating positive element of $ \overset{p}{\underset{1/\alpha}{L(\ZZ)}}\oplus\overset{ 1-p}{\underset{1-1/\alpha}{\CC}} $ with support $ p$ such that $ s^2 $ has the free Poisson distribution $ \pi_{1/\alpha} $, then the linear map
	$$ i: Y_\alpha(u) \mapsto  \alpha psus p,\quad\forall u\in M $$
	extends to a $*$-isomorphism $ i: \Gamma(M,\alpha\varphi) \to p\left((M,\varphi)\ast ( \overset{p}{\underset{1/\alpha}{L(\ZZ)}}\oplus\overset{ 1-p}{\underset{1-1/\alpha}{\CC}} )\right)p = L(\ZZ)\ast p \left((M,\varphi)\ast ( \overset{p}{\underset{1/\alpha}{\CC}}\oplus\overset{ 1-p}{\underset{1-1/\alpha}{\CC}} )\right) p $. In particular, we have
	$$ \Gamma(M,\alpha \varphi) \simeq L(\ZZ)\ast p \left((M,\varphi)\ast ( \overset{p}{\underset{1/\alpha}{\CC}}\oplus\overset{ 1-p}{\underset{1-1/\alpha}{\CC}} )\right) p.$$
\end{thm}
\begin{proof}
	Again, we only need to show that for self adjoint $ u_i\in M $, the two families $ (Y_\alpha(u_i)) $ and $ (\alpha psu_isp) $ have the same distribution. The joint moments for $ (Y_\alpha(u_i)) $ are as before:
	$$ M_{n,\Gamma(M,\alpha \varphi)}(Y_\alpha(u_1),\cdots, Y_\alpha(u_n)) = \sum_{\pi \in \text{NC}(n)} M_{\pi,\alpha\varphi}(u_1,\cdots,u_n)=\sum_{\pi \in \text{NC}(n)}\alpha^{|\pi|} M_{\pi,\varphi}(u_1,\cdots,u_n)$$
	Let $\tau$ be the free product state on $(M,\varphi)\ast ( \overset{p}{\underset{1/\alpha}{L(\ZZ)}}\oplus\overset{ 1-p}{\underset{1-1/\alpha}{\CC}} )$ and $ \tau_p = \alpha\tau|_{p} $ be the compressed state. Again applying Theorem \ref{thm moments of product}, we get the joint moments for $ (\alpha psu_isp) $ in $ p\left((M,\varphi)\ast ( \overset{p}{\underset{1/\alpha}{L(\ZZ)}}\oplus\overset{ 1-p}{\underset{1-1/\alpha}{\CC}} )\right)p $ with respect to $\tau_p$
	\begin{align*}
		M_{n,\tau_p}( \alpha psu_1sp,\cdots, \alpha psu_n sp )	&= \alpha^n \alpha M_{n,\tau}( su_1s,\cdots,su_ns )=\alpha^{n+1}M_{n,\tau}( u_1s^2,\cdots,u_ns^2 )\\
		&=\alpha^{n+1}\sum_{\pi \in \mathcal{NC}(n)} M_{\pi,\varphi}(u_1,\cdots,u_n)R_{K(\pi),\tau}(s^2,\cdots,s^2)\\ &= \alpha^{n+1}\sum_{\pi \in \mathcal{NC}(n)} M_{\pi,\varphi}(u_1,\cdots,u_n)(\frac{1}{\alpha})^{|K(\pi)|}\\
		&=\sum_{\pi \in \mathcal{NC}(n)} M_{\pi,\varphi}(u_1,\cdots,u_n)\alpha^{|\pi|}.
	\end{align*}
	The two families have the same joint distribution, hence $ i $ extends to an injective $*$-homomorphism $ i:  \Gamma(M,\alpha\varphi) \to p\left((M,\varphi)\ast ( \overset{p}{\underset{1/\alpha}{L(\ZZ)}}\oplus\overset{ 1-p}{\underset{1-1/\alpha}{\CC}} )\right)p $. To show that $i$ is also surjective, we simply note that since $ i( X_\alpha(1) )=\alpha ps^2p=\alpha s^2 $ and $ p $ is the support of $ s $, we again have an (bounded!) inverse $(sp)^{-1}$ in $ i( \Gamma(M,\alpha \varphi) ) \cap (p( \overset{p}{\underset{1/\alpha}{L(\ZZ)}}\oplus\overset{ 1-p}{\underset{1-1/\alpha}{\CC}})p)$. Therefore for any $ u\in M $, $ pup = s^{-1}suss^{-1} \in i( \Gamma(M,\alpha \varphi) )$. It is now easy to show that $ \{pMp,s\} $ generate $p\left((M,\varphi)\ast ( \overset{p}{\underset{1/\alpha}{L(\ZZ)}}\oplus\overset{ 1-p}{\underset{1-1/\alpha}{\CC}} )\right)p$. (For example, we have $ ps^lu(1-p)vs^tp = s^l( puvp -puppvp )s^t  $.)
	
	Finally, the identification $p\left((M,\varphi)\ast ( \overset{p}{\underset{1/\alpha}{L(\ZZ)}}\oplus\overset{ 1-p}{\underset{1-1/\alpha}{\CC}} )\right)p = L(\ZZ)\ast p \left((M,\varphi)\ast ( \overset{p}{\underset{1/\alpha}{\CC}}\oplus\overset{ 1-p}{\underset{1-1/\alpha}{\CC}} )\right) p$ follows from Theorem 1.2 \cite{Dyk93}.
\end{proof}

\begin{rmk}
	Note that the two isomorphisms proved above in fact also hold for nonfaithful normal state $\varphi$ when we consider the von Neumann algebra $\Gamma(M,\alpha \varphi)$ generated by $ \ell( \eta_{\alpha\varphi}(x)) + \ell^*(\eta_{\alpha\varphi}( x^*)) + \Lambda( x ) \in B(\FF( L^2(M,\alpha\varphi) ))$ for $x\in M$.
\end{rmk}

\begin{rmk} If we consider $ X^\lambda(u):= a^+(u)+a^-(u)+\lambda a^0(u) $ ($\lambda \neq 0$) and $ \Gamma^\lambda(M,\varphi):= X^\lambda(\frakA)'' $. Then the same calculation shows that $ \Gamma^\lambda(M,\alpha\varphi) \simeq \Gamma(M,\frac{\alpha}{\lambda^2}\varphi) $ via the identification $ X^\lambda_\alpha(u)\mapsto \lambda X_\alpha(u) $. Therefore, considering the algebra generated by the rescaled field operators $ X^\lambda(u) $ is equivalent to rescale the weight $ \varphi $ on the algebraic level.
\end{rmk}

As a quick example, we may apply the above isomorphism to compute the isomorphism class of $ \Gamma(L^{\infty}( [0,\alpha] )) $ which can also be considered as the filtration of algebra of the free Poisson process with rate $\lambda>0$, $ (Y(\mbox{1}_{[0, \lambda t]}))_{t\geq 0} $ (i.e. a process $ Y_t $ with freely independent increment and $ Y_t-Y_s \sim \pi_{\lambda(t-s)}$ for each $t>s$). 
\begin{cor}\label{filtration of free poisson}
	For $ \alpha >0 $, let $ L^\infty([0,\alpha]) $ be the $ L^\infty $ space with respect to the Lebesgue measure on $ [0,\alpha] $, then
	$$ \Gamma(L^\infty([0,\alpha])) \simeq \begin{cases} \underset{\alpha}{L(\FFF(2))}\oplus \underset{1-\alpha}{\CC}, \quad 0<\alpha<1,\\
		L(\FFF(2\alpha)),\quad \alpha\geq 1. 
		
	\end{cases} $$
	In particular, the filtration von Neumann algebras $ \{\mathcal{A}_t\}_{t>0} $ for a free Poisson process of rate $ \lambda>0 $ are
	$$ \mathcal{A}_t \simeq \begin{cases} \underset{\lambda t}{L(\FFF(2))}\oplus \underset{1-\lambda t}{\CC}, \quad 0<\lambda t<1,\\
		L(\FFF(2\lambda t)),\quad \lambda t\geq 1. 
		
	\end{cases} $$
\end{cor}
\begin{proof}
	The cases for $ 0<\alpha \leq 1 $ follows directly from Theorem \ref{rescaling}. Assume now $ \alpha > 1 $ and $ n < \alpha \leq n+1$ for some integer $ n\geq 1 $. Then by Corollary \ref{direct sum},
	\begin{align*}
		\Gamma(L^\infty([0,\alpha])) &= \Gamma(L^\infty([0,1]))\ast \cdots \ast \Gamma(L^\infty([n-1,n]))\ast \Gamma(L^\infty([n,\alpha]))\\
		&\simeq L(\FFF(2n))\ast ( \underset{\alpha-n}{L(\FFF(2))}\oplus \underset{n+1-\alpha}{\CC} ).
	\end{align*}
	Finally, by Proposition 1.7 \cite{Dyk93}, $$L(\FFF(2n))\ast ( \underset{\alpha-n}{L(\FFF(2))}\oplus \underset{n+1-\alpha}{\CC} )\simeq L(\FFF( 2n+2(\alpha -n)^2  +2(\alpha -n)(n+1-\alpha) )=L(\FFF(2\alpha)).$$
	
	For the second statement, we note that $ t\mapsto Y(\mbox{1}_{[0,\lambda t]}) $ is a free Poisson process of rate $\lambda$. Since for each fixed $t>0$, $ \frakA_t := \mbox{span}\{ \mbox{1}_{[0, \lambda s]}, 0<s\leq t \} $ is a dense Hilbert subalgebra of $ L^2([0,\lambda t])\cap L^\infty([0,\lambda t])=L^\infty([0,\lambda t]) $, by Proposition \ref{density results}, $ \mathcal{A}_t = X(\frakA_t)'' = \Gamma(\frakA_t) = \Gamma( L^\infty[0,\lambda t] ) $.
\end{proof}

If $M$ is finite dimensional, then it is known that $ L(\ZZ) \ast (M,\psi) $ is a free Araki-Woods algebra (with almost periodic vacuum state) for any faithful state $ \psi $, see for example Theorem 3.13 \cite{hartglass2021free}. As a consequence, $ \Gamma(M,\varphi) $ is always a free Araki-Woods algebra (possibly direct sum a copy of $\CC$) when $ M$ is nontrivial and finite dimensional.
\begin{cor}
	If $M $ is finite dimensional and $ M\neq \CC $, then $ \Gamma(M,\varphi) $ is a either an interpolated free group factor or a free Araki-Woods algebra. Furthermore, if $\varphi$ is nontracial, then we have
	$$ \Gamma(M,\varphi) \simeq \begin{cases}
		\underset{\varphi(1)}{T}\oplus \underset{1-\varphi(1)}{\CC},\quad &\text{if } \varphi(1)<1;\\
		T,\quad &\text{if } \varphi(1)\geq 1,
	\end{cases}  $$
	where $ T $ is the free Araki-Woods algebra with almost periodic vacuum state whose $\text{Sd}$-invariant is the subgroup of $\RR_+$ generated by the spectrum of $ \Delta_{\varphi} $.
\end{cor}

\section{Second quantization of normal subunital weight decreasing completely positive maps}
In this section, we construct the second quantization for free Poisson algebras for normal subunital and weight decreasing completely positive maps between two von Neumann algebras.

\begin{defn}
	Let $ \varphi $ be a n.s.f. weight on $ M $, $ \psi $ be a n.s.f. weight on $ N $, and $ T:M\to N $ be a normal subunital and weight decreasing completely positive map. We call a unital completely positive map $ \Gamma(T): \Gamma(M,\varphi)\to \Gamma(N,\psi) $ the \textit{second quantization} of $ T $ if it preserves the vacuum state and
	$$ (\Gamma(T)X)\Omega = \FF(T_2)(X\Omega),\quad \forall X\in \Gamma(M,\varphi), $$
	where $T_2:L^2(M,\varphi)\to L^2(N,\psi)$ is again the contraction determined by $ T_2(\eta_{\varphi}(m) ) = \eta_{\psi}(T(m)) $ for all $m\in \nn_\varphi$.
\end{defn}
In particular, since $ \Omega $ is separating, $\Gamma(T)$ is uniquely determined by the above relation (if it exists). Also, $ \Gamma(T) $ preserves the Wick product in the sense that
$$ \Gamma(T)\varPhi(\xi_1\otimes\cdots \xi_n ) = \varPhi(T_2\xi_1\otimes\cdots T_2\xi_n ),\quad \forall n\geq 0,\xi_i\in \frakA_\varphi. $$

The following proposition says that second quantization preserves the dual of completely positive maps.
\begin{prop}
	If $ \Gamma(T): \Gamma(M,\varphi)\to \Gamma(N,\psi) $ is the second quantization of $ T:M\to N $, then $ \Gamma(T)^\star: \Gamma(N,\psi)\to \Gamma(M,\varphi) $ is the second quantization of $ T^\star:N\to M $.
\end{prop}
\begin{proof}
	This follows from direct computation: let $ J_{\Omega,1} =\FF(J_\varphi) $ (resp. $  J_{\Omega,2}=\FF(J_\psi)$) be the Tomita conjugation for the vaccum state of $ \Gamma(M,\varphi) $ (resp. $\Gamma(N,\psi)$), then for all $ Y\in \Gamma(N,\psi) $,
	\begin{align*}
		(\Gamma(T)^\star Y)\Omega =& (\Gamma(T)^{\star})_2 (Y\Omega)=J_{\Omega,1}\left( \Gamma(T)\right)_2^*J_{\Omega,2} Y\Omega = J_{\Omega,1}\FF(T_2)^* J_{\Omega,2}Y\Omega \\ =&  J_{\Omega,1} \FF(T_2^* )J_{\Omega,2}Y\Omega  = J_{\Omega,1}\FF(J_{\varphi} (T^{\star})_2 J_{\psi})J_{\Omega,2}Y\Omega = \FF( (T^{\star}_2) )(Y\Omega ),
	\end{align*}
	where we used the fact that $ \left( \Gamma(T)^{\star}\right)_2 = J_{\Omega,1}\left( \Gamma(T)\right)_2^*J_{\Omega,2} = J_{\Omega,1}\FF(T_2)^*J_{\Omega,2} $ and $ (T^{\star})_2 = J_\varphi T_2^* J_\psi$ by Theorem \ref{duality of positive maps} (3).
\end{proof}

We now consider correspondence $\HHH_T $ from $M$ to $N$ and its associated $M$-$N$ bimodule $X_T\simeq \mathcal{L}_N( L^2(N),\HHH_T ) $ as defined in Section 2.4. We are interested in the specific element $ i_N:= 1_M\otimes_{T} 1_N \in X_T $. Since $ \langle 1_M\otimes_{T} 1_N,1_N\otimes_{T} 1_M\rangle_N = T(1_M)\leq 1_N $, $ i_N \in X_T \simeq \mathcal{L}_N( L^2(N),\HHH_T )$ is a $ N $-linear contraction from $ L^2(N) $ to $ \HHH_T  $.

Similarly, as $T$ is weight decreasing, $ T^{\star} $ is subunital. Therefore, if we consider the $ N $-$M$ bimodule $ X_{T^{\star}}\simeq \mathcal{L}_M(L^2(M),\HHH_{T^\star}) $, then $ i_M:= 1_N\otimes_{T^{\star}}1_M \in X_{T^{\star}}\simeq \mathcal{L}_M(L^2(M),\HHH_{T^\star}) $ is also a $M$-linear contraction from $ L^2(M) $ to $ \HHH_{T^{\star}} $. By Proposition \ref{H of T star is conjugate H T}, we have $ \HHH_{T^{\star}}\simeq \overline{\HHH}_T $. If we let $J_T: \HHH_T\to \overline{\HHH}_T $ be the conjugation map, then $ J_T i_M J_\varphi: L^2(M)\to \HHH_{T} $ is also a $M$-linear contraction. We denote $j_M := J_T i_M J_\varphi$ in short.

\begin{lem}\label{lem for second quantization}
	For $\xi\in L^2(M,\varphi)$, $ i_N^*( j_M(\xi) ) = T_2\xi $. And for $ m\in M $, $ i_N^* m \,i_N =  T(m) \in N$.
\end{lem}
\begin{proof}
	1) Since $ T_2 $ and $ i^*_N\circ j_M $ are both contractions, we may assume that $ \xi \in \frakA_\varphi $. Let $ e_i\in \nn_{\varphi}\cap \nn_{\varphi}^* $ be a net in $ N $ increasing to $1_N$, then $ i_MJ_\varphi(\xi) = 1_N\otimes_{T^{\star}} J_\varphi\xi   = \lim_{i} e_i\otimes_{T^{\star}}J_\varphi\xi \in \HHH_{T^{\star}}$. Under the isomorphism $ \HHH_{T^{\star}}\simeq \overline{\HHH_T} $ in Proposition \ref{H of T star is conjugate H T}, this becomes $ \lim_i \overline{ \pi_l(\xi)\otimes_T J_\varphi \eta_{\psi}(e_i) } $. Therefore $ j_M(\xi) = \lim_i \pi_l(\xi)\otimes_T J_\varphi \eta_{\psi}(e_i)\in \HHH_{T} $. For $\eta\in L^2(N,\psi)$, 
	\begin{align*}
		&\langle \eta,i_N^*(j_M(\xi)) \rangle= \langle i_N(\eta) ,j_M(\xi) \rangle_{\HHH_{T}} = \langle 1_M\otimes_T \eta,\lim_i \pi_l(\xi)\otimes_T J_\psi \eta_{\psi}(e_i) \rangle_{\HHH_{T}}\\
		=&\lim_i \langle \eta, T(\pi_l(\xi))J_\varphi \eta_\psi(e_i) \rangle = \lim_i \langle \eta, J_\psi e_i J_\psi T_2 \xi \rangle = \langle \eta, T_2\xi \rangle.
	\end{align*}
	\newline
	2) For all $ \eta\in L^2(N,\psi) $,
	\begin{align*}
		i_N^* m \, i_N(\eta) &= i_N^* m(1_M\otimes_T \eta ) = i_N^*( m\otimes_{T} \eta ) = T(m)\eta.
	\end{align*}
\end{proof}

\begin{thm}\label{main}
	Let $(M,\varphi)$, $(N,\psi)$ be two von Neumann algebras with n.s.f. weights. For any normal subunital weight decreasing completely positive map $ T:M\to N $, the second quantization $ \Gamma(T):\Gamma(M,\varphi)\to \Gamma(N,\psi) $ exists.
\end{thm}
\begin{proof}

	Consider the (degenerate) correspondence $ \tilde{\HHH}_T := L^2(M)\oplus \HHH_{T}\oplus L^2(N) $ from $ M $ to $N$ with $ M $ acting on $ L^2(N) $ as $0$ and $ N^{op} $ acting on $ L^2(M) $ as $0$ (rigorously speaking, since the action of $M$ and $N^{op}$ is degenerate, $ \tilde{\HHH}_T$ is not an actual correspondence). For a $\xi\in \frakA_\varphi$, we denote again $ \pi_T(\xi):= \pi_T(\pi_l(\xi)) $ the action of $\pi_l(\xi)\in M  $ on $ \tilde{\HHH}_T $. Let $ \FF(\tilde{\HHH}_T ) $ be the full Fock space of $ \tilde{\HHH}_T  $, we then have the preservation operator $ \Lambda( \pi_{T}(\xi) ) \in B( \FF( \tilde{\HHH}_{T} ) )$.

	Let $ k_M: L^2(M)\to \tilde{\HHH}_T $ be the $ M $-linear map $$ k_M( \xi ) := (\sqrt{1-j_M^* j_M}(\xi), j_M(\xi),0 ),\quad \forall \xi\in L^2(M), $$
	then $ k_M $ is an isometry as $$ k_M^*k_M = ( \sqrt{1-j_M^* j_M},j_M^*,0)\begin{pmatrix}
		\sqrt{1-j_M^* j_M}\\
		j_M\\
		0
	\end{pmatrix} =1. $$

	Similarly, let $ p_N : \tilde{\HHH_{T}}\to L^2(N) $ be the $N^{op}$-linear map $$ p_N( \zeta,\xi,\eta ) = i_N^*(\xi)+\sqrt{ 1-i_N^* i_N }(\eta),\quad \forall \zeta\in L^2(M),\xi\in \HHH_{T}, \eta\in L^2(N), $$
	and one can similarly check $ p_N $ is a projection, i.e. $ p_Np_N^*=1 $.
	
	As $k_M: L^2(M)\to \tilde{\HHH}_T $ is isometric and $M$-linear, by Theorem \ref{cumulants of operators on Fock spaces}, the family $ (X(\xi))_{\xi\in \frakA_\varphi} $ and $ ( \ell(k_M(\xi))+ \ell^*(k_M(S_\varphi \xi)) +\Lambda( \pi_{l}(\xi) ) )_{\xi\in \frakA_\varphi} \subseteq B( \FF( \tilde{\HHH}_{T} ) )$ have the same joint free cumulants and hence have the same joint distribution with respect to the vacuum state. Therefore, the map
	$$ I: X(\xi)\mapsto  \ell(k_M(\xi))+ \ell^*(k_M(S_\varphi \xi)) +\Lambda( \pi_{l}(\xi) ),\quad \forall \xi\in \frakA_\varphi $$
	extends to an injective $*$-homomorphism $ I: \Gamma(M,\varphi)\to B( \FF( \tilde{\HHH}_{T} ) ) $ preserving the vacuum state.

	We now claim that the second quantization is \begin{align*}
		\Gamma(T):& \Gamma(M,\varphi) \to B(\FF(L^2(N,\varphi)))\\
		&X \mapsto \FF(p_N)I(X)\FF(p_N)^*.
	\end{align*}
	First, we need to show that $ \Gamma(T)X $ is indeed in $ \Gamma(N,\psi) $ for all $ X\in \Gamma(M,\varphi) $. By Propositon \ref{comb}, we have for all $ \xi_1,\cdots,\xi_l\in \frakA_\varphi $,
	\begin{align*}
		&\FF(p_N)I(\Psi(\xi_1,\cdots ,\xi_l))\FF(p_N)^*\\ =&\FF(p_N) I\left(\sum_{ s=1 }^{l+1}a^+(\xi_1)\cdots a^+(\xi_{s-1})a^-(\xi_{s})\cdots a^-(\xi_l)+ \sum_{ s=1}^{l}a^+(\xi_1)\cdots a^+(\xi_{s-1})a^0(\xi_s)a^-(\xi_{s+1})\cdots a^-(\xi_l)\right) \FF(p_N)^*\\ =&\FF(p_N) \Big(\sum_{ s=1 }^{l+1}l(k_M(\xi_1))\cdots l(k_M(\xi_{s-1}))l^*( k_M(S_\varphi\xi_{s}))\cdots l^*(k_M(S_\varphi \xi_l))\\ & + \sum_{ s=1}^{l}l(k_M(\xi_1))\cdots l(k_M(\xi_{s-1}))\Lambda(\pi_{l}(\xi_s))l^*(k_M(S_\varphi \xi_{s+1}))\cdots l^*(k_M(S_\varphi \xi_l))\Big) \FF(p_N)^*
	\end{align*}
	Recall that for any contraction $ C$ of Hilbert spaces, we have $ \FF(C)l(h)=l(Ch)\FF(C) $ and $ l^*(h)\FF(C)^* = \FF(C)^*l^*(Ch) $. Therefore, using the equation $ p_N(k_M(\xi)) = i^*_N( j_M(\xi) ) = T_2\xi $ for all $\xi\in L^2(M,\varphi)$ (Lemma \ref{lem for second quantization}) and that $T_2S_\varphi = S_\psi T_2$ on $ \frakA_\varphi $ (as $ T $ is positive), we get \begin{align*}
		\FF(p_N)I(\Psi(\xi_1,\cdots ,\xi_l))\FF(p_N)^* =& \sum_{ s=1 }^{l+1}l(T_2 \xi_1)\cdots l(T_2 \xi_{s-1})l^*(S_\psi T_2 \xi_{s})\cdots l^*(S_\psi T_2  \xi_l)\\ &+ \sum_{ s=1}^{l}l(T_2 \xi_1)\cdots l(T_2 \xi_{s-1})\FF(p_N)\Lambda(\pi_{l}(\xi_s))\FF(p_N)^*l^*(S_\psi T_2 \xi_{s+1})\cdots l^*(S_\psi T_2  \xi_l)
	\end{align*} 
	Notice also that since $ p_Np_N^* = 1 $ and $ p_N\pi_l(\xi)p_N^*= i_N^*\pi_l(\xi)i_N = T(\pi_l(\xi))$ for all $ \xi\in \frakA_\varphi$, we have
	\begin{align*}
		\FF(p_N)\Lambda(\pi_{l}(\xi_s))\FF(p_N)^* = \Lambda(p_N \pi_{l}(\xi_s) p_N^*) = \Lambda( T(\pi_l(\xi_s)) ) = a^0(T_2\xi_s).
	\end{align*}
	 We can now continue the calculation:
	\begin{align*}
			&\FF(p_N)I(\Psi(\xi_1,\cdots ,\xi_l))\FF(p_N)^*\\
			=&\sum_{ s=1 }^{l+1}a^+(T_2\xi_1)\cdots a^+(T_2\xi_{s-1})a^-(T_2\xi_{s})\cdots a^-(T_2\xi_l)+ \sum_{ s=1}^{l}a^+(T_2\xi_1)\cdots a^+(T_2\xi_{s-1})a^0(T_2\xi_s)a^-(T_2\xi_{s+1})\cdots a^-(T_2\xi_l)\\
			=&\Psi(T_2\xi_1,\cdots ,T_2\xi_l)\subseteq \Gamma(N,\psi).
	\end{align*}
	Finally, for $X= \Psi(\xi_1,\cdots ,\xi_l)\in \Gamma(M,\varphi) $, $ (\Gamma(T)X)\Omega=(\Gamma(T)\Psi(\xi_1,\cdots ,\xi_l))\Omega = \Psi(T_2\xi_1,\cdots ,T_2\xi_l)\Omega = T_2\xi_1\otimes \cdots \otimes T_2\xi_l = \FF(T_2)( X \Omega) $, which also implies that $ \Gamma(T) $ is state preserving. In particular, $\Gamma(T)$ is indeed the second quantization of $ T $.
\end{proof}

\begin{eg}
	If $ M\subseteq N $, $ \varphi = \psi|_{M} $ is semifinite, and $ T: M\to N $ is the inclusion map, then the second quantization $ \Gamma(T) $ is also an inclusion. (Indeed, the two families $ (X(\frakA_\varphi)) $ and $ (\Gamma(T)(X(\frakA_\varphi) )) $ has the same joint distribution.) On the other hand, if $ T: M\to N$ is a generalized conditional expectation (in the sense of \cite{AC82}\cite{petz1984dual}), i.e. $ T^\star $ is a weight preserving inclusion, then $ \Gamma(T) $ is also a generalized conditional expectation, since $ \Gamma(T)^\star = \Gamma( T^{\star} ) $ is a state preserving inclusion. Note that, however, $ \Gamma(T) $ may not be a homomorphism even when $ T $ is a homomorphism (but not weight preserving).
\end{eg}

\begin{eg}
	For each $ \lambda\in [0,1] $, we can consider the second quantization of the constant multiplication $ \Gamma(\lambda): \Gamma(M,\varphi)\to \Gamma(M,\varphi) $. This gives us the Ornstein–Uhlenbeck semigroup $ P_t:= \Gamma(e^{-t}):\Gamma(M,\varphi)\to \Gamma(M,\varphi) $ on the free Poisson algebra.
\end{eg}

\section{Factoriality and type classification}
First, let us recall the factoriality and type results for free product algebras, see \cite{UEDA20112647}\cite{MR1426835}\cite{MR1273959}\cite{MR1224611} for more details about general nontracial free product von Neumann algebras.
\begin{thm}[Thm. 3.4 \cite{UEDA20112647}]\label{factoriality free product}
	Let $ M_1 $, $ M_2 $ be two nontrivial von Neumann algebras with faithful normal state $ \varphi_1 $ and $\varphi_2$, and $ (M,\varphi) = (M_1,\varphi_1)\ast(M_2,\varphi_2)$. If either $ M_1 $ or $M_2$ is diffuse, then $ M $ is a factor of type $\text{II}_1$ or $ \text{III}_\lambda $ ($\lambda\neq 0$), whose $T$-invariant is $ T(M) = \{ t\in \RR|\sigma_{t}^{\varphi_1} =\text{id} = \sigma_{t}^{\varphi_2}\} $.
\end{thm}

\begin{thm}[Thm. 3.7 \cite{UEDA20112647}]\label{free product extreme ergodic}
	Let $ M_1 $, $ M_2 $ be two nontrivial von Neumann algebras with faithful normal state $ \varphi_1 $ and $\varphi_2$, and $ (M,\varphi) = (M_1,\varphi_1)\ast(M_2,\varphi_2)$. If either $ (M_1)_{\varphi_1} $ or $(M_2)_{\varphi_2}$ is diffuse and the other $(M_i)_{\varphi_i}\neq \CC$, then $ (M_\varphi)'\cap M^\omega $, where $ M^\omega $ is a free ultrapower of $ M $.
\end{thm}

Using the factoriality of free product algebras, we can directly show the factoriality of $ \Gamma(M,\varphi) $ when $ \varphi(1)<\infty $. Thus the remaining interesting cases are when $ \varphi(1)=\infty $, for which we divide further into the cases: 1) $ \varphi $ is strictly semifinite; 2) $ \varphi $ is not strictly semifinite. Case 1) follows from a standard inductive limit argument, while Case 2) is proved via locating the centralizer $ \Gamma(M,\varphi)_{\varphi_\Omega} $.

The key step for proving the factoriality of $ \Gamma(M,\varphi) $ when $ \varphi $ is not strictly semifinite is to show that the centralizer of $ \Gamma(M,\varphi) $ in fact belongs to the free Poisson algebra of the strictly semifinite part of $(M,\varphi)$. We note that similar phenomenon also happens for free Araki-Woods algebras where the centralizer is always contained in the almost periodic part (see \cite[Theorem 3.4]{bikram2017generator} and \cite{Fum01}.)

We now take a slight detour and discuss the strictly semifinite part of the weight $\varphi $.
\subsection{Strictly semifinite part of $(M,\varphi)$ }
Recall that a n.s.f. weight $\varphi$ is called strictly semifinite if and only if its restriction to the centralizer $ M_\varphi $ is still semifinite. Equivalently, $\varphi$ is strictly semifinite if only if it can be written as a sum of positive normal linear functionals with mutually orthogonal supports. For a general $\varphi$, we may consider the following definition describing the strictly semifinite part of $\varphi$.

\begin{defn}
	Let $ \varphi $ be a n.s.f. weight on $M$. The \textit{strictly semifinite support} $z_{ssf}(\varphi)$ of $ \varphi $ is the largest projection in $\mathcal{Z}(M_\varphi)$ such that $ \varphi $ is semifinite on $ M_\varphi z_{ssf}(\varphi) $. We also call the subalgebra $ z_{ssf}(\varphi)Mz_{ssf}(\varphi) $ the \textit{strictly semifinite part} of $ (M,\varphi) $.
\end{defn}

Note that as $  z_{ssf}(\varphi)\in M_\varphi $, $ \varphi $ is also semifinite (in fact strictly semifinite) on $ z_{ssf}(\varphi)Mz_{ssf}(\varphi) $.

Since $ \varphi|_{M_\varphi} $ is tracial on $ M_\varphi $, $ \nn_{\varphi}\cap M_\varphi $ is an ideal of $ M_\varphi $, and one can check that $ z_{ssf}(\varphi)\in \mathcal{Z}(M_{\varphi}) $ is the unique projection such that $ M_\varphi z_{ssf}(\varphi) = ( \nn_{\varphi}\cap M_\varphi  )'' $. On the other hand, we can also describe the $ L^2 $-closure of $ \nn_{\varphi}\cap M_\varphi $ (i.e. the closure of $\eta_{\varphi}(\nn_{\varphi}\cap M_\varphi)  $ in $ L^2(M,\varphi) $) as the space of fixed vector of $\Delta_\varphi$. Similarly, one can also locate the eigenvectors of $\Delta_\varphi$ inside $ L^2(  z_{ssf}(\varphi)Mz_{ssf}(\varphi), \varphi) $. To see this, we recall the following well-known polar decomposition on standard form, see for example \cite[Exercise IX.1 2)]{Tak03}.

\begin{lem}
	Let $(M,L^2(M),J,\mathfrak{P} )$ be the standard form of $M$, and a vector $\xi \in L^2(M)$. There exists a unique decomposition $ \xi= u|\xi| $ with $|\xi|\in \mathfrak{P}$ and $u\in M$ partial isometry such that $u^*u=[M'|\xi|]$ and $uu^*= [M'\xi ]$. If $\phi$ is a n.s.f weight on $M$, and $ \xi $ is a eigenvector of $\Delta_\varphi$, then $ \Delta_\varphi^{it}|\xi| = |\xi| $ and $\sigma_{t}^\varphi(u)= \lambda^{it}u$ where $ \lambda $ is the eigenvalue for $\xi$.
\end{lem}

\begin{lem}\label{lem: approximating fixed vector}
	If $ \xi \in L^2(M) $ and $ \Delta_\varphi^{it}\xi= \lambda^{it}\xi  $ for a n.s.f weight $\varphi$, then there exists an positive unbounded operator $ h $ affiliated with $ M_\varphi $ such that for each $ \varepsilon>0 $, $ h_\varepsilon := {h}/{(1+\varepsilon h)}\in \mm_\varphi$ and $ \lim_{\varepsilon \to 0_+}\eta_\varphi(h_\varepsilon^{1/2}) = |\xi| $.
	
	In particular, the space of fixed vectors of $\Delta_\varphi$ is the closure of $ \eta_{\varphi}( \nn_{\varphi}\cap M_\varphi ) \subseteq L^2(M)$.
\end{lem}
\begin{proof}
	Let $ \omega_{\xi}:=\langle \xi,\cdot\xi\rangle  $, then $ \omega_{\xi}\circ \sigma_t^\varphi = \langle \lambda^{it}\xi,\cdot \lambda^{it}\xi \rangle =  \omega_{\xi}$ for all $t\in \RR$. Therefore $ (D\omega_{\xi}:D\varphi)_t $ is a one parameter unitary group in $ M_\varphi $. Let $ h $ be the generator of $  (D\omega_{\xi}:D\varphi)_t = h^{it} $, then $ \omega_{\xi}(x) = \lim_{\varepsilon\to 0_+}\varphi( h^{1/2}_\varepsilon  x h^{1/2}_\varepsilon ) $ on $M$ (See VIII Corollary 3.6. and Lemma 2.8. \cite{Tak03}. See also Theorem 4.10 \cite{stratila2020modular}.) In particular, we have $ \varphi(E_{[s,\infty)}(h))\leq \frac{1+\varepsilon s}{s} \varphi(h_\varepsilon) <\infty$ for every $s>0$, where $ E_{[s,\infty)}(h) $ is the spectrum projection of $h$. And from the limit $ \lim_{\varepsilon\to 0_+}\varphi(h_\varepsilon) = \omega_{\xi}(1) = 1 $, one have $$ \int_{\RR_+} s d\varphi\circ E(h) = 1.$$
	For $ \varepsilon,\varepsilon'>0 $,
	\begin{align*}
		&\|\eta_\varphi( h_\varepsilon^{1/2} )-\eta_\varphi( h_{\varepsilon'}^{1/2} )\|_{L^2(M)}= \varphi( (h_\varepsilon^{1/2}-h_{\varepsilon'}^{1/2})^2 )^{1/2} = \left(\int_{\RR_+}  \left(  \frac{s^{1/2}}{\sqrt{1+\varepsilon s}}- \frac{s^{1/2}}{\sqrt{1+\varepsilon' s}} \right)^2 d\varphi\circ E(h)\right)^{1/2}\\
		\leq & \left(    \int_{\RR_+}  \left(  \frac{s^{1/2}}{\sqrt{1+\varepsilon s}}- s^{1/2} \right)^2 d\varphi\circ E(h)   \right)^{1/2}+\left(    \int_{\RR_+}  \left(  \frac{s^{1/2}}{\sqrt{1+\varepsilon' s}}- s^{1/2} \right)^2 d\varphi\circ E(h)   \right)^{1/2}
	\end{align*}
	which by dominant convergence theorem converges to $0$ as $ \varepsilon,\varepsilon'\to 0 $. In particular, $ \eta_\varphi( h_\varepsilon^{1/2} ) $ is Cauchy as $ \varepsilon\to 0 $ hence $ \lim \eta_\varphi( h_\varepsilon^{1/2} ) = \eta  $ for some $ \eta\in \mathfrak{P} $ (as $ \eta_\varphi( h_\varepsilon^{1/2} ) \in \mathfrak{P} $ for each $\varepsilon$). But now, one have $ \omega_{\xi} = \lim_{\varepsilon\to 0_+} \omega_{ \eta_\varphi( h_\varepsilon^{1/2} ) } = \omega_{\eta} $, we must have $ \eta = |\xi| $.
\end{proof}

\begin{lem}\label{eigenvectors are in z}
	 All the eigenvectors of $ \Delta_\varphi  $ are in $ z_{ssf}L^2(M)z_{ssf} = L^2(z_{ssf}Mz_{ssf},\varphi|_{z_{ssf}Mz_{ssf}} ) $.
\end{lem}
\begin{proof}
	Let $ \xi $ be an eigenvector of $\Delta_\varphi$, then $ |\xi| = \lim_{\varepsilon \to 0_+} \eta_{\varphi}(h_\varepsilon) $ for some positive $h$ affiliated with $M_\varphi$ and $ \varphi(E_{[s,\infty)}(h)) <\infty$ for every $s>0$. In particular, $E_{[s,\infty)}(h) \in \nn_{\varphi}\cap M_\varphi\subseteq M_\varphi z_{ssf}(\varphi)$ hence $E_{[s,\infty)}(h)z_{ssf}=E_{[s,\infty)}(h)$ or equivalantly $E_{[s,\infty)}(h)\leq z_{ssf}$. As the support of $h_\varepsilon$ is contained in $ \vee_{s>0}E_{[s,\infty)}(h)\leq z_{ssf} $, $h_\varepsilon=z_{ssf}h_\varepsilon=h_\varepsilon z_{ssf}$, so $ |\xi| = \lim_{\varepsilon \to 0_+} \eta_{\varphi}(h_\varepsilon) = \lim_{\varepsilon \to 0_+} \eta_{\varphi}(h_\varepsilon z_{ssf})= \lim_{\varepsilon \to 0_+} \eta_{\varphi}(h_\varepsilon) z_{ssf}= |\xi|z_{ssf}  $. This shows that $ \xi z_{ssf} = u|\xi|z_{ssf}= u|\xi|= \xi $. Applying the same argument to $J\xi$ we also obtain $ z_{ssf}\xi=\xi $.
\end{proof}

\begin{rmk}
	Note that however, this does not imply that eigenoperators of $ \sigma_{t}^\varphi $ are always inside $  z_{ssf}Mz_{ssf}$. For example, $1\in M $ has eigenvalue $1$, but is not in $ z_{ssf}Mz_{ssf}  $ unless $\varphi$ strictly semifinite.
\end{rmk}

\subsection{Proof of factoriality}
We can now locate the almost periodic part of $ (\Gamma(M,\varphi),\varphi_\Omega) $.
\begin{lem}\label{lem: location of almost periodic part}
	If $x\in \Gamma(M,\varphi)$ is an eigenoperator of $ \sigma^{\varphi_\Omega} $, i.e. there exists a $ \lambda>0 $, $ \sigma_{t}^{\varphi_\Omega}(x)=\lambda^{it}x $ for all $ t\in \RR $, then we have $ x \in \Gamma( z_{ssf}Mz_{ssf},\varphi|_{ z_{ssf}Mz_{ssf} } ) $.
\end{lem}
\begin{proof}
	Note that $ x\Omega\in \FF(L^2(M,\varphi))$ is an eigenvector of $ \Delta_\Omega $. Since $ \Delta_{\Omega}^{it} $ is simply the tensor power of $ \Delta_\varphi^{it} $ on each $ L^2(M)^{\otimes n} $, the eigenspace of $\Delta_\Omega|_{ L^2(M)^{\otimes n}}  $ is contained in $ L^2(z_{ssf}Mz_{ssf})^{\otimes n} $. In particular, the eigenspace of $\Delta_\Omega  $ is contained in $ \FF(L^2(z_{ssf}Mz_{ssf})) $. Therefore, we must have $ x\Omega\in \FF(L^2(z_{ssf}Mz_{ssf})) $, hence $ x\in \Gamma( z_{ssf}Mz_{ssf},\varphi|_{z_{ssf}Mz_{ssf} } ) $.
\end{proof}

\begin{thm}
	The free Poisson algebra $\Gamma(M,\varphi)$ is a factor if and only if $\varphi(1)\geq 1$ and $M\neq \CC$. Moreover, when $ \Gamma(M,\varphi) $ is a factor, we have also
	\begin{enumerate}[1)]
		\item 	$ \Gamma(M,\varphi)_{\varphi_\Omega} $ is a factor if $ \varphi(z_{ssf}(\varphi))\geq 1$ and $  M_\varphi z_{ssf}(\varphi)\neq \CC z_{ssf}(\varphi) $.
		\item $ \Gamma(M,\varphi)_{\varphi_\Omega} = \CC $ if and only if $ \varphi $ has no strictly semifinite part, i.e. $z_{ssf}(\varphi) = 0$.
	\end{enumerate}
\end{thm}
\begin{proof}
	\begin{enumerate}[a)]
		\item If $ \varphi(1)<\infty $, the statement follows immediately from Theorem \ref{rescaling}, Theorem \ref{rescaling2}, and Theorem \ref{factoriality free product}.
		\item If $\varphi(1) = +\infty$, and $ \varphi $ is strictly semifinite, then $ \varphi|_{M_\varphi} $ is a semifinite trace on $ M_\varphi $. The set $ \mathcal{P}_{fi}:= \{ e\in P(M_\varphi): \varphi(p)<\infty \}$ is an increasing net such that $ \displaystyle\lim_{e\in \mathcal{P}_{fi} } e = 1 $. We then have that $ \Gamma(M,\varphi) $ is the inductive limit of $ \Gamma_e:= \Gamma(eMe,\varphi_{eMe}) $.
		Suppose $ e $ is not minimal in $ M_\varphi $ with $ \varphi(e)=\alpha\geq 1 $. By Theoerem \ref{rescaling2}, $ \Gamma_e\simeq L(\ZZ)\ast p \left((eMe,\frac{1}{\alpha}\varphi|_{eMe})\ast ( \overset{p}{\underset{1/\alpha}{\CC}}\oplus\overset{ 1-p}{\underset{1-1/\alpha}{\CC}} )\right) p $. As the centralizer of $ p \left((eMe,\frac{1}{\alpha}\varphi|_{eMe})\ast ( \overset{p}{\underset{1/\alpha}{\CC}}\oplus\overset{ 1-p}{\underset{1-1/\alpha}{\CC}} )\right) p  $ is nontrivial, by Theoerem \ref{free product extreme ergodic}, we have $$  \left((\Gamma_e)_{\varphi_\Omega}\right)'\cap \Gamma_e=\CC. $$ Therefore we immediately have $ \Gamma(M,\varphi) = \underrightarrow{\lim} \Gamma_e $ is a factor by the standard argument: Let $ E_e: \Gamma(M,\varphi)\to \Gamma_e $ be the faithful state preserving conditional expectation, for any $ x\in (\Gamma(M,\varphi)_{\varphi_\Omega})'\cap \Gamma(M,\varphi) $, $ E_e(x)\in \left((\Gamma_e)_{\varphi_\Omega}\right)'\cap \Gamma_e=\CC $ for all $e$. Thus $ x\in \CC $ and $ (\Gamma(M,\varphi)_{\varphi_\Omega})'\cap \Gamma(M,\varphi) = \CC $.
		\item If $\varphi$ is not strictly semifinite, for any $ x\in \mathcal{Z}(\Gamma(M,\varphi)) $, by Lemma \ref{lem: location of almost periodic part}, we have $ x \in \Gamma( z_{ssf}Mz_{ssf},\varphi|_{ z_{ssf}Mz_{ssf} } )$. But as $ \varphi $ is also semifinite on $ z_{ssf}^{\perp}Mz_{ssf}^{\perp}\neq \CC $ (since $z_{ssf}\in M_\varphi$), take a non-constant $ y\in \Gamma( z_{ssf}^{\perp}Mz_{ssf}^{\perp},\varphi|_{ z_{ssf}^{\perp}Mz_{ssf}^{\perp}} ) $, we have that $ x $ and $y$ are freely independent. But as $ x\in \mathcal{Z}(\Gamma(M,\varphi)) $, $x$ and $y$ are both commuting and freely independent, which is impossible unless $ x\in \CC $. Therefore $ \mathcal{Z}(\Gamma(M,\varphi  )) = \CC $.
	\end{enumerate}

	For 1), we have $ \Gamma(M,\varphi)_{\varphi_\Omega}\subseteq  \Gamma(z_{ssf}Mz_{ssf},\varphi|_{z_{ssf}Mz_{ssf}})$ and thus $  \Gamma(M,\varphi)_{\varphi_\Omega}= \Gamma(z_{ssf}Mz_{ssf},\varphi|_{z_{ssf}Mz_{ssf}})_{\varphi_\Omega} $. If $ \varphi(z_{ssf})=\infty $, then the factoriality of $ \Gamma(z_{ssf}Mz_{ssf},\varphi|_{z_{ssf}Mz_{ssf}})_{\varphi_\Omega} $ follows from b).  If $ 1\leq \varphi(1)<\infty $, then $\Gamma(z_{ssf}Mz_{ssf},\varphi|_{z_{ssf}Mz_{ssf}})$ is again a free product algebra by Theorem \ref{rescaling2}, and $\Gamma(z_{ssf}Mz_{ssf},\varphi|_{z_{ssf}Mz_{ssf}})_{\varphi_\Omega}$ is a factor by Theorem \ref{free product extreme ergodic}.
	
	For 2), if $ z_{ssf}(\varphi)=0 $, then $ \Gamma(M,\varphi)_{\varphi_\Omega}\subseteq  \Gamma(z_{ssf}Mz_{ssf},\varphi|_{z_{ssf}Mz_{ssf}})= \Gamma(0,\varphi|_{0}) = \CC $. On the other hand, if $ z_{ssf}(\varphi)\neq 0 $, then as $ \varphi $ is semifinite on $ M_\varphi z_{ssf}(\varphi) $, one can choose a projection $ z_{ssf}\geq p\in  M_\varphi $ such that $ \varphi(p)<\infty $ and hence the free Poisson variable $ Y(p) $ belongs to $ \Gamma(M,\varphi)_{\varphi_\Omega} $.
\end{proof}

	We now discuss the type of $ \Gamma(M,\varphi) $. Recall that for a n.s.f. weight $\varphi$, the Connes spectrum of $ \sigma^\varphi $ is defined as
	$$ {\mit\Gamma}(\sigma^\varphi):= \bigcap_{p\in M^\varphi}\text{Sp}(\sigma^{\varphi_p}), $$
	where $ \text{Sp}(\sigma^\varphi) $ is the Arveson spectrum of $ \sigma^\varphi $, $p$ ranges from nonzero projections in $ M^\varphi $, and $ \varphi_p$ is the n.s.f. weight $ \varphi_p = \varphi|_{pMp} $ on $pMp$. When $ \varphi $ is a state and $ M_\varphi $ is a factor, we have $ {\mit\Gamma}(\sigma^\varphi)= \text{Sp}(\sigma^{\varphi}) $. When $ M $ is a factor and $ \varphi $ is a state, we have also that $ S(M)\cap \RR_+^{\times}= {\mit\Gamma}(\sigma^\varphi)$, where $ S(M) $ is the $S$-invariant (which describes the type of $ M $), defined as the intersection of of the spectrum of $ \Delta_\psi $ for all n.s.f. weight $ \psi $.
	\begin{cor}
		If $ \varphi(1)\geq 1 $ and $ M\neq \CC $, then $ \Gamma(M,\varphi) $ is a factor with type $ \text{II}_1 $ or $ \text{III}_\lambda $ with $\lambda\neq 0$. And the type of $ \Gamma(M,\varphi) $ can be determined as following.
		\begin{enumerate}[a)]
			\item If $ 1\leq \varphi(1)<\infty $, then the $ T $-invariant $ T(\Gamma(M,\varphi)) = \{ t\in \RR: \sigma_t^{\varphi}=\text{id} \} $.
			\item If $ \varphi(1)=\infty $ and $ \varphi $ is strictly semifinite, then the Connes spectrum $ \Gamma(\sigma_{\varphi_\Omega}) = \text{Sp}( \sigma_{\varphi_\Omega}) $ is the closure of the subgroup of $ \RR^{\times}_{+} $ generated by $ \text{Sp}( \Delta_{\varphi})$.
			\item If $ \varphi $ is not strictly semifinite, then $ \Gamma(M,\varphi)  $ has type $ \text{III}_1 $.
		\end{enumerate}
	\end{cor}
\begin{proof}
	a) Since $ \Gamma(M, \varphi) \simeq L(\ZZ)\ast p \left((M,\frac{1}{\varphi(1)}\varphi)\ast ( \overset{p}{\underset{1/\varphi(1)}{\CC}}\oplus\overset{ 1-p}{\underset{1-1/\varphi(1)}{\CC}} )\right) p$, we have $$ T(\Gamma(M,\varphi)) = \{ t\in \RR: \sigma_t^{ p\left((\frac{1}{\varphi(1)}\varphi)\ast \tau\right)p  }=\text{id} \} = \{ t\in \RR: \sigma_t^{\varphi}=\text{id}  \} $$ by Theorem \ref{factoriality free product}. Also, by Theorem \ref{factoriality free product}. Therefore $ \Gamma(M,\varphi) $ is not of type $ \text{III}_0 $, the $ T $-invariant determines the type of $ \Gamma(M,\varphi) $.

	b) If $ \varphi(1)=\infty $ and $ \varphi $ is strictly semifinite, then $ (\Gamma(M, \varphi))_{\varphi_\Omega} $ is a factor. Therefore $ \Gamma(\sigma_{\varphi_\Omega}) = \text{Sp}( \sigma_{\varphi_\Omega})  =  \text{Sp}( \Delta_{\Omega }) $ with is precisely the closure of the subgroup of $ \RR_+^{\times} $ generated by $ \text{Sp}( \Delta_{\varphi}) $ by the Fock space structure.

	c) For any n.s.f. weight $\psi$, we denote in short the corner weight $ \psi_q := \psi|_{qMq} $. First, let us assume that $ z_{ssf}(\varphi) = 0 $, then $\Gamma(M, \varphi)$ has trivial centralizer $ \Gamma(M,\varphi)_{\varphi_\Omega} = \CC $ and thus is a factor of type $ \text{III}_1 $ by \cite[proof of Theorem 3]{longo1979notes}, see also \cite{marrakchi2024ergodic}. Now, if $ z_{ssf}(\varphi)\neq 0 $, then we have $ \Gamma(M,\varphi)_{\varphi_\Omega} = \Gamma(z_{ssf}Mz_{ssf},\varphi_{z_{ssf}})_{\varphi_\Omega}$. So, the Connes spectrum of $ \sigma_{\varphi_\Omega} $ is
	$$ \Gamma(\sigma^{\varphi_\Omega}) := \bigcap \{ \text{Sp}( \sigma^{(\varphi_\Omega)_p  } ): p \text{ nonzero projection of } \Gamma(z_{ssf}Mz_{ssf},\varphi_{z_{ssf}})_{\varphi_\Omega}\}.$$
	For each fixed such $ p $, $ p $ is freely independent from $ \Gamma(z_{ssf}^\perp Mz_{ssf}^\perp,\varphi_{z_{ssf}^{\perp}}) $. Let \begin{align*}
		\HHH_1 &= L^2( \Gamma(z_{ssf}^\perp Mz_{ssf}^\perp,\varphi_{z_{ssf}^{\perp}}), (\varphi_{z_{ssf}^{\perp}})_{\Omega} )= \mathring{\HHH}_1\oplus \CC\Omega \subseteq L^2( \Gamma(M,\varphi),\varphi_\Omega )\\
		 \HHH_2 &= pL^2( \Gamma(z_{ssf}^\perp Mz_{ssf}^\perp,\varphi_{z_{ssf}^{\perp}},(\varphi_{z_{ssf}^{\perp}})_{\Omega} ) )p = \mathring{\HHH}_2\oplus \CC p\Omega p \subseteq L^2( p\Gamma(M,\varphi)p,(\varphi_\Omega)_p ).
	\end{align*}
	By the definition of freeness, one has isometry $ L : \mathring{\HHH}_1 \to \mathring{\HHH}_2 $, $ x\Omega \to \varphi_\Omega(p)^{-1} p x\Omega p $ that intertwines $ \Delta^{it}_{\Omega}|_{ \mathring{\HHH}_1 } $ and $ \Delta^{it}_{(\varphi_\Omega)_p}|_{ \mathring{\HHH}_2 } $. Therefore,
	$$  \text{Sp}( \sigma^{(\varphi_\Omega)_p  }   ) = \text{Sp}( \Delta_{(\varphi_\Omega)_p} )\supseteq \text{Sp}( \Delta_{(\varphi_\Omega)_p}|_{ \mathring{\HHH}_2 } ) = \text{Sp}( \Delta_{\Omega}|_{ \mathring{\HHH}_1 } ).$$
	But since $ \varphi_{z_{ssf}^\perp} $ has no strictly semifinite part, $  \Gamma(z_{ssf}^\perp Mz_{ssf}^\perp,\varphi_{z_{ssf}^{\perp}}) $ has type $ \text{III}_1 $ with trivial centralizer, hence $\text{Sp} (\Delta_{\Omega}|_{ \HHH_1 }) = [0,\infty) $. As $ \Omega \in \HHH_1 $ is a fixed vector of $ \Delta_{\Omega} $, one obtain also $ \text{Sp} (\Delta_{\Omega}|_{ \mathring{\HHH}_1 }) = [0,\infty) $ and therefore $ \text{Sp}( \sigma^{(\varphi_\Omega)_p  } )  = [0,\infty)  $ for all $p\neq 0$. This implies that $ \Gamma(\sigma_{\varphi_\Omega}) =  [0,\infty) $.
\end{proof}

\begin{eg}
	Let $M=B(\HHH)$. Any n.s.f. weight $\varphi$ on $ M $ is of the form $ \text{Tr}(h\cdot) $ with $h\geq 0$ the density operator affiliated with $M$. It is shown in \cite{GP15} that $\varphi$ is strictly semifinite if and only if $h$ has pure point spectrum. Therefore, $ \Gamma(B(\HHH),\varphi) $ is a type $\text{III}_1$ factor as long as $ h $ does not have pure point spectrum. On the other hand, if $h$ has pure point spectrum, then $ \Gamma(\sigma_{\varphi_\Omega}) $ is simply the closure of the subgroup generated by the spectrum of $h$. If we furthermore assume that $ \{h\}'' $ is a maximal abelian subalgebra of $ B(\HHH) $, then $ \mathcal{Z}(M_\varphi)=M_\varphi = \{h\}'' $ and $z_{ssf}( \varphi )$ is exactly the projection onto the eigenvector space of $h$. In this case, $ (\Gamma(B(\HHH), \varphi),\varphi_\Omega) $ has trivial centralizer if and only if $ h $ has no point spectrum.
\end{eg}

\begin{rmk}
		We may compare some of our results with the corresponding results in \cite{chen2023noncommutative}. It was shown in \cite{chen2023noncommutative} that the classical Poisson algebra of a finite weight $\varphi$ (i.e. $\varphi(1)<1$) on $M$ is isomorphic to the direct sum of symmetric tensor powers $ \sum_{n=0}^{\infty} M^{\otimes_s n} $ with the state $ e^{-\varphi(1)}\bigoplus_{n\geq 0} \frac{\varphi^{\otimes n}}{n!} $. In particular, the classical Poisson algebra can never be a factor unless $ \varphi(1)=\infty $, while in the free case $ \Gamma(M,\varphi) $ is a factor as long as $ \varphi(1)\geq 1$ and $ M $ is nontrivial. This phenomenon is similar to the case of $q$-Araki-Woods algebras, where the free cases ($q=0$) are much easier to have factoriality due to freeness, while for the symmetric cases ($ q=1 $), we get a factor only when the underlining standard subspace has trivial intersection with its symplectic complement. \cite{chen2023noncommutative} also showed that a sufficient condition for the classical Poisson algebra to be a factor is that $\varphi$ has the form $ \varphi = \text{Tr}\otimes \omega $ with $ \text{Tr} $ the trace on $ B(\HHH) $ for a infinite dimensional $\HHH$ and $ \omega $ is another n.s.f. weight. For second quantization, \cite{chen2023noncommutative} discussed the second quantization construction of positive maps, while in our construction we require the maps to be completely positive. Finally, \cite{chen2023noncommutative} defined the Poisson algebra using a direct combinatorial description, while our approach emphasize on the Fock space construction with the help of preservation operators. 
\end{rmk}

\section{Infinitely divisible families, and L\'{e}vy-It\^{o}'s decomposition via pseudo Hilbert algebras}

Recall that for a Borel measure $\mu$ on $\RR$, its free cumulants generating function is defined as $ C_\mu(z):=zG_\mu^{\langle -1\rangle}(z)-1 $, where $ G_\mu(z) $ is the Cauchy transform of $\mu$: $ G_\mu(z):= \int_{\RR} \frac{1}{z-t}d\mu(t) $. When $ \mu $ is compactly supported, we have $ C_\mu(z) = \sum_{n=1}^{\infty}\kappa_{n}z^n $ with $ \{\kappa_{n}\}_{n\geq 1} $ the free cumulants of $ \mu $. By the free L\'{e}vy-Khintchine decomposition (see \cite{BOT02}\cite{BT05}), for any (possibly not compactly supported) freely infinitely divisible (F.I.D. in short) distribution $ \mu $, there exists a $\sigma$-finite Borel measure $\rho$ on $\RR$, $ a\in \RR $ and $b\geq 0$ such that $ \int_{\RR} \min\{ 1,t^2 \} d\rho(t) < \infty $, $ \rho(\{0\})=0 $ and its free cumulants generating function $C_{\mu}(z)$ is given by
\begin{equation}\label{unbounded Levy-Khinchine}
	C_{\mu}(z) = az + b^2z^2 + \int_{\RR} (\frac{1}{1-zt}-1 - zt\chi_{[-1,1]}(t))d\rho(t).
\end{equation}
The measure $ \rho $ is called the (free) L\'evy measure of $ \mu $, and $ (a,b,\rho) $ is called the (free) L\'evy triple of $ \mu $.

Let $ \rho' = \chi_{[-1,1]}\cdot\rho $ and $ \rho'' = \chi_{\RR/[-1,1]}\cdot \rho $, then we can rewrite $ C_{\mu}(z) $ as
$$ C_{\mu}(z) = az+b^2z^2 + \int_{\RR} (\frac{1}{1-zt}-1 - zt)d\rho'(t) + \int_{\RR} (\frac{1}{1-zt}-1)d\rho''(t). $$
This implies that, similiar to the classical case, we have the following free L\'{e}vy-It\^{o}'s decomposition: $\mu$ is the free additive convolution of a semicircular distribution $ \mu_1 $ with $ C_{\mu_1}(z) = az+b^2z^2 $, a compensated centered compound free Poisson distribution $\mu_2$ with $ C_{\mu_2}(z) =  \int_{\RR} (\frac{1}{1-zt}-1 - zt)d\rho'(t)$, and a compound free Poission distribution $\mu_3$ with $ C_{\mu_3}(z) = \int_{\RR} (\frac{1}{1-zt}-1)d\rho''(t) $. (In general, we say $\mu$ is compensated centered compound free Poisson if it has L\'{e}vy triple $ (0,0,\rho) $ with $ \rho $ support on $ [1,1] $. And we say $\mu$ is compound free Poisson if it has L\'{e}vy triple $ (a,0,\rho)$ with $ a= \rho([-1,1])<\infty $.)

Note that the semicircular distribution $ \mu_1 $ can be naturally realized via Voiculescu's free Gaussian functor whose underline information is a Hilbert space, while $\mu_2$ and $ \mu_3 $ can be instead realized using the free Poisson functor whose underline information is a measure or equivalently a (commutative) Hilbert algebra. It is therefore natural to consider pseudo left Hilbert algebras which is a combination of both Hilbert spaces and Hilbert algebras. We will see that for a (compactly supported) freely infinitely divisible distribution, its free cumulants provide us a natural pseudo Hilbert algebra whose decomposition matches the L\'{e}vy-It\^{o}'s decomposition (Theorem \ref{thm Levy Ito decomposition}).

\subsection{Multivariable free L\'{e}vy-It\^{o}'s decomposition in tracial probability space}
Recall that a family $ \{y_i\}_{i\in I} $ of bounded self-adjoint random variables of a non-commutative probability space $ (\mathcal{A},\varphi) $ is jointly freely infinitely divisible (J.F.I.D. in short) if for any $ t\in \RR_+ $, there is another family $ \{y_i^t\}_{i\in I} $ of random variables of a non-commutative probability space $ (\mathcal{A}_t,\varphi_t) $ such that for all $ i_1,\cdots, i_n\in I $.
$$ tR_n(y_{i_1},\cdots,y_{i_n})=R_n(y^t_{i_1},\cdots,y^t_{i_n}).$$
In particular, this implies that $ R_n(y_{i_1},\cdots,y_{i_n}) =\lim_{t\to 0} \frac{1}{t}M_n( y^t_{i_1},\cdots,y^t_{i_n} ) $ for $ n\geq 1 $, which then implies that $ R_n: y_{i_1}\cdots y_{i_n}\mapsto R_n(y_{i_1},\cdots,y_{i_n}  ) $ defines a positive linear functional on the algebra of noncommutative polynomials $\CC_{0}\langle \{ x_i \}_{i\in I} \rangle$ without constant terms.

From Proposition \ref{cumulants}, one can immediately see that the free poisson random weights $X$ and $Y$ always provide us a J.F.I.D. family.
\begin{cor}
	Let $ \frakA $ be a pseudo left Hilbert algebra, then the family $ \{ X(\xi): \xi\in \frakA, S\xi = \xi \} $ is J.F.I.D.
\end{cor}
\begin{proof}
	We simply need to consider the pseudo left Hilbert algebra $ \frakA_t $ with $ \frakA_t $ and $ \frakA $ being same involution algebra but with the new inner product $ t\langle \cdot,\cdot\rangle $ for each $ t>0 $, and then apply Proposition \ref{cumulants}.
\end{proof}

In general, if $ \{y_i\}_{i\in I} $ is a family of unbounded self-adjoint random variables affiliated with some tracial $W^*$-probability space  $ (\mathcal{A},\tau) $, we say that $ \{y_i\}_{i\in I} $ is J.F.I.D. if for any $ m\in \NN $, there is another family $ \{z_i\}_{i\in I} $ of unbounded random variables affiliated with a tracial $ W^* $-probability space $ (\mathcal{B},\tau_1) $ such that there is a trace preserving embedding
$$ \ell: (\mathcal{A},\tau) \to \ast_{j=1}^{m}(\mathcal{B},\tau_1), $$
and $ \ell({y_i}) = \sum_{j=1}^{m} z_i^{(j)}$, where $ z_i^{(j)} $ is the copy of $ z_i $ for each free product components of $ (\mathcal{B},\tau_1) $.

\begin{lem}
	Let $ \{y_i\}_{i\in I} $ be a J.F.I.D. family of bounded self-adjoint variables in a tracial $W^*$-probability space $ (\mathcal{A},\tau) $. Consider the (non-unital) $*$-algebra of non-commutative polynomials without constant terms $ P = \CC_{0}\langle \{x_i\}_{i\in I} \rangle $ with the involution $ \alpha x_{i_1}\cdots x_{i_n}\mapsto \bar{\alpha}x_{i_n}\cdots x_{i_1} $ ($\forall\alpha\in \CC$) and the positive sesquilinear form $$ \langle x_{i_1}\cdots x_{i_n},x_{j_1}\cdots x_{j_{m}} \rangle:= R_{m+n}(y_{i_n},\cdots,y_{i_{1}},y_{j_1},\cdots,y_{j_m}).$$ Let $ N:= \{ p\in P:\langle p,p\rangle = 0 \}\subseteq P$ be the null $*$-ideal of this positive form, and $\frakA:= P/N$ be the quotient $*$-algebra. We have that $ \frakA $ is a unimodular (i.e. $\Delta = 1$) pseudo left Hilbert algebra.
\end{lem}
\begin{proof}
	Since $ \tau $ is tracial, $R_n$ is cyclic symmetric, hence the involution $S:\frakA\to \frakA$ is anti-unitary on $\frakA$. In particular, $S$ is closable. It remains to show that for each polynomial $p$, the left multiplication by $p$ is bounded (we denote the image of $p$ in $\frakA$ still by $p$). For this, we only need to check the cases when $p=x_i$.
	
	We now follow the proof of Theorem 3 \cite{GSS92}. We notice that as the domain of $ \pi_{l}(x_i) $ contains $ \frakA $, we can still define the preservation operator $ \Lambda( \pi_{l}(x_i) ) $ as an unbounded operator on $ \FF(\overline{\frakA}) $. Note also that Theorem \ref{cumulants of operators on Fock spaces} still holds for unbounded $\Lambda( \pi_{l}(x_i) ) $ as long the right-hand side of the equation in Theorem \ref{cumulants of operators on Fock spaces} is well-defined. Similar to Proposition \ref{cumulant transform for X(x)}, we have that $ \{ \ell(x_i)+\ell^*(x_i)+ \Lambda( \pi_{l}(x_i) )+\tau(y_i)\}_{i\in I} $ and $ \{ y_i \}_{i\in I} $ has the same joint free cumulants. As $ y_i $ is bounded, $ \ell(x_i)+\ell^*(x_i)+ \Lambda( \pi_{l}(x_i) )+\tau(y_i) $ must also be bounded. But since $ \ell(x_i),\ell^*(x_i) $ are bounded, we must have that $  \Lambda( \pi_{l}(x_i) ) $ is bounded, which then implies that $ \pi_{l}(x_i) $ is bounded. In particular, this shows that $\frakA$ is a (unimodular) pseudo left Hilbert algebra.
\end{proof}

\begin{thm}\label{thm Levy Ito decomposition}
	Let $ \{y_i\}_{i\in I} $ be a J.F.I.D. family of self-adjoint variables affiliated with a tracial $W^*$-probability space $ (\mathcal{A},\tau) $. Assume that $ \mathcal{A} $ is generated by $ \{y_i\}_{i\in I} $, then there exists a larger tracial $W^*$-probablity space $ (\tilde{\mathcal{A}},\tau)  $, and a J.F.I.D. compound free Poisson family $ \{p_i\}_{i\in I} $, a J.F.I.D. compensated centered free Poisson family $ \{z_i\}_{i\in I} $, a semicircular family $ \{s_i\}_{i\in I} $, and a family $\{\lambda_i\}_{i\in I}$ of constants, all of which affiliated with $ (\tilde{\mathcal{A}},\tau)  $, such that the families $ \{p_i\}_{i\in I} $, $ \{z_i\}_{i\in I} $, and $ \{s_i\}_{i\in I} $ are freely independent from each other, and for each $ i\in I $, $$ y_i = \lambda_i + s_i + p_i+z_i.$$
\end{thm}
\begin{proof}
	First, for a general (unbounded) F.I.D. variable $ y $, if $ y $ has L\'evy triple $(a,b,\rho)$, then one can always write $y$ as a sum of freely independent (bounded) variables $ \{h_j\}_{j\geq 0} $, such that $ h_0 $ is semicircular and each $ h_n $ has L\'{e}vy triple $ (0,0,\chi_{(j-1,j]}\rho) $ for $j\geq 1$. Therefore, every $ y_i $ can be written as an (infinite) sum of free independently bounded variables $ \{h_{i,j}\}_{j\geq 0} $. Now the new family $ \{h_{i,j}\}_{i,j} $ is also a J.F.I.D. family of bounded variables. Therefore, by replacing $ \{y_i\}_{i\in I} $ with $ \{h_{i,j}\}_{i,j} $, it suffices to prove the statement for $ \{y_i\}_{i\in I} $ with each $ y_i $ bounded.
	
	Let $\frakA$ be the pseudo left Hilbert algebra from the previous Lemma, and consider its free Poisson algebra $\Gamma(\frakA)$. By Theorem \ref{multi Levy-Khin}, $$\Gamma(\frakA)=\Gamma\left( (\frakA^2)'' \oplus ( (\frakA^2)^\perp \cap D(S) ) \right) = \Gamma((\frakA^2)''\oplus (\frakA^2)^\perp)$$ where $ (\frakA^2)^\perp $ is equipped with the trivial multiplication $ uv =0 $. Since $ (\frakA^2)'' $ is unimodular, it is the (left) Hilbert algebra of a von Neumann algebra $M$ is equipped with a tracial weight $ \tau_0 $.
	
	Again, since the two families $\{ y_i\} $ and $ \{X(x_i)+ \tau(y_i) \in \Gamma(\frakA)\} $ has the same joint distribution, we can consider $ (\mathcal{A},\tau) $ as a subalgebra of $ (\Gamma(\frakA),\varphi_{\Omega}) $. For each $i\in I$, let $ {x_i}=a_i + b_i $ with $ a_i\in (\frakA^2)'' $ and $ b_i\in (\frakA^2)^{\perp} $. Now $ p_i := Y(a_i\cdot \chi_{(-\infty,1)\cup (1,\infty)}(a_i)) = X(\chi_{(-\infty,1)\cup (1,\infty)}(a_i)a_i)+\tau_0( a_i\cdot\chi_{(-\infty,1)\cup (1,\infty)}(a_i) ) \in \Gamma( \frakA^2)=\Gamma(M,\tau_0)$ is a J.F.I.D. compound free Poisson family, $z_i:=  X(a_i\cdot\chi_{[-1,1]}(a_i)) $ is a J.F.I.D. compensated centered free Poisson family, and $ s_i = X(b_i)\in \Gamma( (\frakA^2)^{\perp} ) $ is a semicircular family since the multiplication on $ (\frakA^2)^{\perp} $ is trivial. Finally, let $ \lambda_i = \tau(y_i)- \tau_0( a_i\cdot\chi_{(-\infty,1)\cup (1,\infty)}(a_i) )$, we then get
	$$   y_i = X(x_i)+\tau(y_i) = X(a_i)+X(b_i) +\tau(y_i) = \lambda_i + s_i + p_i+z_i. $$
\end{proof}

\begin{eg}
	If $ x $ is a bounded freely infinitely divisible random variables, then $ (M,\tau_0) $ is a separable abelian von Neumann algebra generated by a self-adjoint element $ a $, and thus we can assume $ M = L^\infty(\RR,\mu) $ for some measure $ \mu $ on $ \RR $ and $ a = t $ such that $ \int_{\RR}t^n d\mu(t) = R_n(x)$ for $n\geq 3$ and $ \int_{\RR} td\mu(t) =0 $. Therefore we can consider the measure $ \mu $ as the centered free L\'{e}vy measure of $x$.
\end{eg}

Note that for non-tracial cases, the argument in the above theorem still works if the null left-ideal $ N=\{p\in P: \langle p,p\rangle = 0 \} $ is a $*$-ideal of $ P = \CC_0\langle \{x_i\}_{i\in I} \rangle$. The simplest case is the free Araki-Woods algebra. However, in general, it is not clear whether $N$ is always a $*$-ideal.

\begin{eg}
	Consider the free Araki-Woods algebra $ \Gamma_{AW}(\RR^2,U_t) $ for two dimensional space with $$ U_t= \begin{pmatrix}
		\cos\log(\lambda)t\;& -\sin\log(\lambda)t\\
		\sin\log(\lambda)t\;& 	\cos\log(\lambda)t		
	\end{pmatrix} $$
	with $ 0<\lambda\leq 1 $ as in \cite{Shl97}. Then $ \Gamma_{AW}(\RR^2,U_t) $ is generated by two elements $ X_1 = l(e_1)+l^*(e_1) $ and $ X_2 = l(e_2)+l^*(e_2) $ where $ (e_1,e_2) $ is an orthogonal basis of $ \RR^2$. The free cumulants of $ X_i $ are $ R_{2}(X_1,X_1)=R_{2}(X_2,X_2)=1 $, $ R_2(X_1,X_2)=i(1-\lambda)/(1+\lambda)=-R_2(X_2,X_1) $, and $ R_{n}(X_{i_1},\cdots,X_{i_n}) =0 $ for any $ n\neq 2 $. Therefore for any $p\in P$ with no degree $1$ term, we have $ \langle p,p\rangle=0 $. So we get $ N = \{ p\in P: p \text{ has no degree } 1 \text{ term} \} $ is a $*$-ideal and $ \frakA = P/N $ is a two dimensional algebra $ \frakA = \CC {y_1} +\CC {y_2} $ with trivial multiplication. Obviously, the linear map $ L: {y_i}\mapsto e_i $ is an isometry from $ \frakA $ to $ (\RR^2\otimes \CC,\langle\cdot,\cdot \rangle_{U_t}  ) $. Moreover, $ L $ preserves the involutions.
\end{eg}

\subsection{Fock space realization of unbounded freely infinitely divisible variables}
For a bounded self-adjoint variable $x$, its Cauchy transform is $ G_x(z) = \varphi( 1/(z-x) ) = \sum_{n=0}^{\infty} m_n/z^{n+1} $, and its free cumulant transform and Voiculescu transform are $ C_x(z) = z(G_x)^{-1}(z)-1 = \sum_{n=1}^{\infty}\kappa_{n} z^n $ and $ \varphi_x(z) = zC_x(1/z) = \sum_{n=0}^{\infty}\kappa_{n+1}/z^{n} $, where $ m_n $ and $ \kappa_n $ are the moments and cumulants of $x$. Therefore, we can easily compute the free cumulant generating function of our field operators $ X(x) $ and $ Y(x) $. If $x\in (\nn_{\varphi}^*\cap \nn_{\varphi})_{s.a.}$ for some n.s.f. weight $ \varphi $, then the free cumulants transform is
\begin{equation}\label{cumulant transform for X(x)} C_{X(x)}(z) = \sum_{n=1}^\infty \kappa_{n}(X(x))z^n= \sum_{n=2}^{\infty} \varphi(x^n)z^n = \varphi( \frac{1}{1-zx}-1-zx ),\quad \varphi_{X(x)}(z) = \sum_{n=0}^{\infty}\frac{\varphi(x^{n+1})}{z^{n}}= \varphi( \frac{zx}{z-x} ).\end{equation}
If we furthermore assume that $ x\in \mm_\varphi $, then $ \varphi(x) $ exists and the free cumulant generating function of $Y(x)$ is
\begin{equation}\label{cumulant transform for Y(x)} C_{Y(x)}(z) = \sum_{n=1}^{\infty} \varphi(x^n)z^n = \varphi( \frac{1}{1-zx}-1) \end{equation} 
In particular, if $ \varphi $ is a state and $ \alpha >0 $, then $ Y_{\alpha}(x)\in \Gamma( M,\alpha\varphi ) $ has compound free Poisson distribution $ \pi_{\mu_{x},\alpha} $, and
$$ C_{Y_{\alpha}(x)}(z) = \alpha \varphi(\frac{1}{1-zx}-1) = \alpha ( G_x(1/z)/z-1 )$$

For a F.I.D. $\mu$, consider its L\'{e}vy It\^{o}'s decomposition $ \mu = \mu_1\boxplus\mu_2\boxplus \mu_3 $ as before. To realize $\mu_1$, $\mu_2$ and $\mu_3$ as variables on a single Fock space, we only need to consider the pseudo commutative Hilbert algebra $ \frakA_\rho^+ = \frakA_\rho\oplus \CC\xi =( L^\infty(\RR,\rho)\cap L^2(\RR,\rho) )\oplus \CC\xi $, where $\CC\xi$ is the one dimensional trivial algebra: $\xi = S\xi$, $ \xi^2=0 $, and $ \langle \xi,\xi \rangle=1$. Indeed, $ \mu_1$ and $\mu_2 $ can be realized in $ \Gamma(\frakA_\rho^+) $ as $ bX(\xi)+a $ and $ X(t\cdot\chi_{[-1,1]}(t)) $ (here $t\cdot\chi_{[-1,1]}(t)\in L^\infty(\RR,\rho)\cap L^2(\RR,\rho) $ since $ \int_{-1}^{1}t^2d\rho(t)\leq \int_{\RR} \min\{ 1,t^2 \}d\rho(t) < \infty $). If $ \rho $ is compactly supported, then $ \mu_3 $ can also be realized as $ Y( t\cdot \chi_{\RR/[-1,1]}(t) ) $ since $ t\cdot \chi_{\RR/[-1,1]} $ is bounded on the support of $ \rho $. Therefore, the only remained difficulty is to define $ Y(t\cdot \chi_{\RR/[-1,1]}(t)) $ when $ \rho $ has unbounded support and show that the free cumulants transform of $ Y(t\cdot \chi_{\RR/[-1,1]}(t)) $ is still of the form \eqref{cumulant transform for Y(x)}.

\begin{lem}\label{extension of Y}
	If $ \tau $ is a tracial n.s.f. weight on $M$, then the positive linear map $ Y: \mm_\tau \ni x\mapsto Y(x)\in \Gamma(M,\tau)$ can be extended to a positive linear map $ \tilde{Y}: (\mm_\tau+\mbox{FS}(\tilde{M}))\to \widetilde{\Gamma(M,\tau)} $, where $ \tilde{M} $ and $ \widetilde{\Gamma(M,\tau)} $ are the algebras of affiliated unbounded operators, and $ \mbox{FS}(\tilde{M}) $ is the algebra of operators in $ \tilde{M} $ with finite support (i.e. the support has finite trace). Moreover, for self-adjoint $x\in \mbox{FS}(\tilde{M})$, the free cumulants transform of $\tilde{Y}(x)$ is $$ C_{\tilde{Y}(x)}(z) = \tau( \frac{1}{1-zx}-1 ), \quad \forall z\notin \RR.$$
\end{lem}
\begin{proof}
	Assume that the support $ p=s(x) $ of $ x\in \tilde{M} $ is finite, i.e. $ \tau(p)< \infty $. Then we can consider the subalgebra $ pMp \subseteq M$, and thus $ \Gamma(pMp,\tau) $ is a subalgebra of $ \Gamma(M,\tau) $. But since $ \tau(p)<\infty $, by the concrete form of $Y$ as in Theorem \ref{rescaling} and Theorem \ref{rescaling2}, we can extend $Y$ to the whole $ \widetilde{pMp} $, and we have $ \tilde{Y}(\widetilde{pMp}) \subseteq \widetilde{\Gamma(pMp,\tau)}\subseteq \widetilde{\Gamma(M,\tau)}$ as $ \tau $ is tracial.
	
	We now turn to prove the formula for $ C_{\tilde{Y}(x)}(z)$. Let $ x_n = xe_{x}([-n,n] ) $ where $ e_{x} $ is the spectral measure of $ x $, then $ x_n $ converges to $ x $ in probability and hence by the concrete form of $Y$ in Theorem \ref{rescaling} and Theorem \ref{rescaling2}, we also have $ Y(x_n) $ converges to $ \tilde{Y}(x) $ in probability. For each variable $u$, let $ \phi_u:=zC_u(1/z) $ be the Voiculescu transform of $u$. By equation \eqref{cumulant transform for Y(x)}, we have
	$$ \phi_{Y(x_n)}(z)= zC_{Y_{x_n}}(1/z) = \varphi(\frac{z^2}{z-x_n}) = \int_{-n}^{n}\frac{z^2}{z-t}d(\tau\circ e_x)(t)+ z\tau( e_x((-\infty,n)\cup (n,\infty)) ).$$
	By Proposition 5.7 \cite{BV93}, the Voiculescu transform of $ Y(x_n) $ also converges to the Voiculescu transform of $ \tilde{Y}(x) $: $$ \phi_{\tilde{Y}(x)}(z) = \lim_{n\to \infty} \phi_{Y(x_n)}(z) = \int_{\RR}\frac{z^2}{z-t}d(\tau\circ e_x)(t) = \tau(\frac{z^2}{z-x}).$$ Thus $\displaystyle C_{{\tilde{Y}(x)}}(z)  = z\phi_{\tilde{Y}(x)}(1/z) = \tau(\frac{1}{1-zx}-1 ).$
\end{proof}

The above discussion about the realization of $\mu$ can then be summarized into the following corollary.
\begin{cor}\label{realization of FID}
	For any freely infinitely divisible distribution $ \mu $ with L\'{e}vy triple $ (a,b,\rho) $, we can realize $ \mu $ as an affiliated operator $z$ of the free Poisson algebra $ \Gamma(\frakA_\rho^+) $ for the pseudo Hilbert algebra $ \frakA_\rho^+ := ( L^\infty(\RR,\rho)\cap L^2(\RR,\rho) )\oplus \CC\xi$ such that
	$$ z = a + bX(\xi) + X(x\cdot\chi_{[-1,1]}(x)) + \tilde{Y}( x\cdot\chi_{\RR/[-1,1]}(x) ).$$\qed
\end{cor}

\begin{rmk}
	We can see from the above formula that, unlike the classical case, a large family of freely infinitely divisible distributions have compact support. In particular, the compensated centered compound free Poisson distribution always has compact support, and the only possible unbounded part of $x$ is the compound free Poisson part, which, however, has a rather concrete representation as given in Theorem \ref{rescaling} and Theorem \ref{rescaling2}.
\end{rmk}

\section{Filtration algebras of free L\'{e}vy process}

With the help of Corollary \ref{realization of FID}, we can now realize a (one-parameter additive) free L\'{e}vy process in a single von Neumann algebra. For a free L\'{e}vy process with L\'{e}vy triplet $ (a,b,\rho) $, consider the Hilbert algebra $ \frakA_{\rho}^+\otimes \frakA_{\mathcal{L}} = \frakA_{\rho\times \mathcal{L}}\oplus (\CC\xi\otimes \frakA_\mathcal{L})  $ where $ \mathcal{L} $ is the Lebesgue measure on $\RR_+$, then by Corollary \ref{realization of FID}, we can realize the free L\'{e}vy process in $ \Gamma( \frakA_{\rho}^+\otimes \frakA_{\mathcal{L}} ) $ as
\begin{equation}\label{representation of free levy}
	 Z_t = at + bX(\xi\otimes\chi_{[0,t]}(y) ) + X(x\cdot\chi_{[-1,1]}(x) \chi_{[0,t]}(y)) + \tilde{Y}( x\cdot\chi_{\RR/[-1,1]}(x) \chi_{[0,t]}(y) ),
\end{equation}
where a point of $ \RR\times \RR_+ $ is denoted by $ (x,y) $. Note that since $ \CC\xi $ is the trivial algebra with zero multiplication, the multiplication of $\CC\xi\otimes \frakA_\mathcal{L}$ is also trivial and thus its free Poisson algebra $ \Gamma(  \CC\xi\otimes \frakA_\mathcal{L}) $ is simply the standard free Gaussian algebra $ \Gamma_{Gau}(L^2(\RR_+,\mathcal{L})) $.

In this section, we are mainly interested in the filtration algebras of $Z_t$, i.e. the von Neumann algebra $ M_t $ generated by $ \{Z_s\}_{0\leq s\leq t} $. We will show that if the free Gaussian part is nonzero, then for each $t>0$, $M_t$ is nothing other than $ \Gamma( \frakA_{\rho}^+\otimes \frakA_{\mathcal{L}|_{[0,t]}}   ) $. In particular, the algebra $ M_\infty $ generated by the whole process $\{Z_t\}_{t\geq 0}$ is exactly the entire algebra $ \Gamma( \frakA_{\rho}^+\otimes \frakA_{\mathcal{L}} ).$ In order to prove this, we first need to recall the higher variation processes for a free L\'{e}vy process introduced by \cite{AW18}. In terms of Fock space, those processes can be defined as follows.

\begin{defn}[\cite{AW18}]
	Let $ (Z_t)_{t\geq 0} $ be a free L\'{e}vy process with L\'{e}vy triple $(a,b,\rho)$ realized as \eqref{representation of free levy} in $\Gamma( \frakA_{\rho}^+\otimes \frakA_{\mathcal{L}} )$, then the $ k $-th variation of the process is the free L\'{e}vy process (in fact free compound Poisson)
	$$ Z^{\{k\}}_t := \delta_{2,k}b^2 t + \tilde{Y}(x^k \cdot\chi_{[0,t]}(y)), \quad \forall k\in \NN_+, k\geq 2.$$
\end{defn}

In \cite{AW18}, it is shown that for each $t>0$, $ \lim_{N\to \infty}\sum_{i=0}^{N-1}(Z_{{(i+1)t}/{N}}-Z_{{it}/{N}})^k$ converges to $ Z^{\{k\}}_t $ in distribution. And when $ Z_t $ is a free compound Poisson process with rate $1$, then the limit actually converges in probability. We will improve this result by showing that actually for all free L\'{e}vy processes, the limit converges in probability. Therefore, this implies that the $k$-th variation $Z_t^{\{k\}}$ is also affiliated with the filtration algebra $M_t$ of $Z_t$.

The following two lemmas will be useful to prove the convergence in probability. First, we need the following well known Khintchine inequality in \cite{Voi86} for the norm of sum of freely independent variables.
\begin{lem}\label{norm of sum}
	If $ X_1,\cdots, X_n $ are freely independent bounded self-adjoint random variables with identical distribution in a $C^*$ probability space $ (M,\varphi) $, then
	$$ \|\sum_{i=1}^{n} X_i\|\leq \|X_1\|+ 2\sqrt{n}\|X_1\|_{L^2}+ (n+1)| \varphi(X_1) | $$
\end{lem}
\begin{proof}
	If $ \varphi(X_1) =0$, then the claim follows from Lemma 3.2 \cite{Voi86}. In general, apply Lemma 3.2 \cite{Voi86} to $ (X_i-\varphi(X_1))_{i=1}^n $, we have \begin{align*}
		\|\sum_{i=1}^{n} X_i\|&\leq \| \sum_{i=1}^{n} (X_i-\varphi(X_1)) \|+n|\varphi(X_1)|\leq \|X_1-\varphi(X_1)\|+2\sqrt{n}\|X_1-\varphi(X_1)\|_{L^2}+n|\varphi(X_1)|\\
		&\leq \|X_1\|+2\sqrt{n}\|X_1\|_{L^2}+(n+1)|\varphi(X_1)|
	\end{align*}
\end{proof}

In order to deal with unbounded operator, we also need the following lemma from \cite{BV93}.

\begin{lem}[Lemma 4.4 \cite{BV93}]\label{projection lemma}
	Let $ (M,\varphi) $ be a tracial $W^*$-probability space, $ X_1,X_2 $ be unbounded operators affiliated with $M$, $p_1,p_2\in M$ be two projections. Denote $ X'_1 = X_1p_1 $ and $ X'_2 = X_2p_2 $, then
	\begin{enumerate}[1)]
		\item There exists projection $ p\in M $, such that $ (X_1+X_2)p = (X'_1+X'_2)p $ and $ \varphi(p^\perp)\leq \varphi(p_1^\perp)+\varphi(p_2^\perp) $;
		\item There exists projection $ q\in M $, such that $ (X_1X_2)q = (X'_1X'_2)q $ and $ \varphi(q^\perp)\leq \varphi(p_1^\perp)+\varphi(p_2^\perp) $.
	\end{enumerate}
\end{lem}

\begin{prop}\label{square variation}
	Let $ (Z_t)_{t\geq 0} $ be a free L\'{e}vy process with L\'{e}vy triple $(a,b^2,\rho)$ realized as \eqref{representation of free levy} in $\Gamma( \frakA_{\rho}^+\otimes \frakA_{\mathcal{L}} )$.
	\begin{enumerate}[a)]
		\item For a fixed $ t >0$ and $ k\geq 2 $,
		$$ \lim_{N\to \infty}\sum_{i=0}^{N-1}(Z_{{(i+1)t}/{N}}-Z_{{it}/{N}})^k = Z^{\{k\}}_t,$$
		with the limit being taken in probability in $ (\Gamma(\frakA_{\rho}^+\otimes \frakA_{\mathcal{L}} ),\varphi_\Omega) $.
		\item If furthermore $ \rho $ is compactly supported, then the limit actually converges in norm.
	\end{enumerate}
\end{prop}
\begin{proof}
	We will use the short notation $ Z_{i,N} := Z_{{(i+1)t}/{N}}-Z_{{it}/{N}} $, and we need to show that $\displaystyle  \lim_{N\to \infty}\sum_{i=0}^{N-1} (Z_{i,N})^k = Z^{\{k\}}_t $.
	
	\begin{enumerate}[1)]
		\item 	We first prove the second claim. Assume that  $ \rho $ is supported in $ [-M,M] $ for some $0<1<M$, then in particular $ x\in L^2(\RR,\rho)\cap L^\infty(\RR,\rho) $. So, we can rewrite \eqref{representation of free levy} as
		$$ Z_t = t(a+\int_{\RR/[-1,1]} xd\rho(x)) +  bX(\xi\otimes\chi_{[0,t]}(y) ) + X( x\cdot \chi_{[0,t]}(y) ). $$
		Let $ c = a+\int_{\RR/[-1,1]} xd\rho(x) $, $ f_t= (b\xi+x)\otimes \chi_{[0,t]}(y)\in \frakA_{\rho}^+\otimes \frakA_{\mathcal{L}} $, and $ f_{i,N} = f_{(i+1)t/N}-f_{it/N} $. Then $ Z_t = tc+ X(f_t) $ and $ Z_{i,N} = tc/N + X(f_{i,N}) $. We claim that the constant $c$ does not contribute to the limit $\lim_{N\to \infty}\sum_{i=0}^{N-1} (Z_{i,N})^k $. Indeed, we have $$ \sum_{i=0}^{N-1} (Z_{i,N})^k = \sum_{i=0}^{N-1}X(f_{i,N})^k + \sum_{i=0}^{N-1}\sum_{p=1}^{k}{k\choose p}(tc/N)^pX(f_{i,N})^{k-p}.$$
		Notice that $ I_{i,N}:=(\sum_{p=1}^{k}{k\choose p}(tc/N)^pX(f_{i,N})^{k-p})_{i=0}^{N-1} $ are freely independent variables with the same distribution for $ i=0,\cdots,N-1 $, by Lemma \ref{norm of sum}, $$ \|\sum_{i=0}^{N-1} (Z_{i,N})^k - \sum_{i=0}^{N-1}X(f_{i,N})^k\|=\|\sum_{i=0}^{N-1}I_{i,N}\| \leq \|I_{1,N}\|+2\sqrt{N}\| I_{1,N}\|_{L^2}+ (N+1)|\varphi_{\Omega}(I_{1,N})|.$$
		So, we only need to estimate those three terms. Assume that $ N>tc $. Then since $ \|X(f_{i,N})\|\leq \|a^+(f_{i,N})\|+\|a^-(f_{i,N})\|+\|a^0(f_{i,N})\|\leq 2\|f_{i,N}\|+ M = 2\sqrt{1/N}\|f_t\|+M$, we have $$ \|I_{1,N}\|\leq \| \sum_{p=1}^{k}{k\choose p}(tc/N)^pX(f_{1,N})^{k-p} \| \leq (tc/N)2^k(2\|f_{t}\|+M)^k=O(1/N).$$ Also \begin{align*}
			\|I_{1,N}\|_{L^2}&\leq  \sum_{p=1}^{k}{k\choose p}(tc/N)^p\|X(f_{1,N})^{k-p}\|_{L^2}\leq (tc/N)^k+\sum_{p=1}^{k-1}{k\choose p}(tc/N)^p\|X(f_{1,N})\|^{k-p-1}\|f_{1,N}\|\\&\leq (tc/N)2^k (2\|f_t\|+M)^{k-p-1}\sqrt{1/N}\|f_t\|=O(1/N^{3/2}).
		\end{align*}
		Now, we have
		\begin{align*}
			&\|\sum_{i=0}^{N-1} (Z_{i,N})^k - \sum_{i=0}^{N-1}X(f_{i,N})^k\|\leq \|I_{1,N}\|+2\sqrt{N}\| I_{1,N}\|_{L^2}+ (N+1)|\varphi_{\Omega}(I_{1,N})|\\ \leq& \|I_{1,N}\|+2\sqrt{N}\| I_{1,N}\|_{L^2}+ (N+1)\| I_{1,N}\|_{L^1}
			\leq \|I_{1,N}\|+(2\sqrt{N}+N+1)\| I_{1,N}\|_{L^2}\\ \leq& O(1/N)+(2\sqrt{N}+N)O(1/N^{3/2}) = O(1/N^{1/2}).
		\end{align*}
		
		Therefore, the constant $c$ does not contribute to the limit, and we only need to show that $ \sum_{i=0}^{N-1}X(f_{i,N})^k $ converges to $ \delta_{2,k}b^2t+ Y(x^k \cdot \chi_{[0,t]}(y)) $ in norm. Let $ d_{i.N} = X(f_{i,N})^k - (\delta_{2,k}b^2t/N+ Y(x^k \cdot \chi_{[it/N,(i+1)t/N]}(y)))$, then again $ (d_{i,N})_{i=0}^{N-1} $ are freely independent and have the same distribution. Apply Lemma \ref{norm of sum}, we have
		$$ \|\sum_{i=0}^{N-1}X(f_{i,N})^k- \delta_{2,k}b^2t+ Y(x^k \cdot \chi_{[0,t]}(y))\|= \|\sum_{i=0}^{N-1}d_{i,N}\|\leq \| d_{1,N} \|+2\sqrt{N}\|d_{1,N}\|_{L^2}+(N+1)|\varphi_{\Omega}(d_{i,N})| . $$
		It remains to estimate $ \|d_{1,N}\|_{L^2} $, $\|d_{1,N}\|$ and $ \varphi_{\Omega}(d_{1,N}) $. Notice that since $ f_{i,N}^k = x^k\chi_{[it/N,(i+1)t/N]}(y) $ for $k\geq 2$, we can write $ d_{i,N} $ in terms of $ a^+(f_{i,N}) $, $ a^-(f_{i,N}) $, and $ a^0(f_{i,N}) $ as
		\begin{align*}
			d_{i,N} = &(a^+(f_{i,N})+a^0(f_{i,N})+ a^-(f_{i,N}))^k - (a^+(f_{i,N}^k)+a^0(f_{i,N}^k)+ a^-(f_{i,N}^k)+ \langle f_{i,N},f_{i,N}^{k-1}\rangle)\\
			=& (a^+(f_{i,N})+a^0(f_{i,N})+ a^-(f_{i,N}))^k\\&-  a^0(f_{i,N})^{k-1}a^+(f_{i,N})-a^0(f_{i,N})^k - a^-(f_{i,N})a^0(f_{i,N})^{k-1}- a^-(f_{i,N})a^0(f_{i,N})^{k-2}a^+(f_{i,N})\\
			=& \sum_{\substack{\vec{q}\in \{+,0,-\}^k\\ \vec{q}\notin T}}a^{q_1}(f_{i,N})\cdots a^{q_k}(f_{i,N})
		\end{align*}
		where $ T = \{ (0,\cdots,0,+),(0,\cdots,0),(-,0,\cdots,0),(-,0,\cdots,0,+) \} \subseteq \{+,0,-\}^k$. Note $ \|a^+(f_{i,N})\|=\|a^-(f_{i,N})\| =\|f_{i,N}\|= \sqrt{1/N}\|f_t\| $ and $ \|a^{0}(f_{i,N}) \|= M $. Assume that $ N $ is large enough so that $ \|a^{0}(f_{i,N}) \|= M> \|a^+(f_{i,N})\|=\|a^-(f_{i,N})\| = \sqrt{1/N}\|f_t\|  $, then $ \|a^{q_1}(f_{i,N})\cdots a^{q_k}(f_{i,N})\|\leq \sqrt{1/N}\|f_t\|M^{k-1} $ for all $\vec{q}\neq (0,\cdots,0)$. Therefore, $$ \|d_{i,N}\| = \| \sum_{\vec{q}\in \{+,0,-\}^k/T}a^{q_1}(f_{i,N})\cdots a^{q_k}(f_{i,N}) \|\leq \sqrt{1/N}3^k\|f_t\|M^{k-1}=O(1/N^{1/2}). $$
		On the other hand, $\|a^{q_1}(f_{i,N})\cdots a^{q_k}(f_{i,N})\|_{L^2} = \|a^{q_1}(f_{i,N})\cdots a^{q_k}(f_{i,N})\Omega\|$ is nonzero only when the last term is a creation operator, i.e. $ q_k = + $. But for $\vec{q}\notin T$, $\vec{q}\neq (0,\cdots,0,+)$ and thus $ a^{q_1}(f_{i,N}),\cdots, a^{q_{k-1}}(f_{i,N}) $ contains at least one creation or annihilation operator. Therefore, we have for all $\vec{q}\notin T  $, \begin{align*}
			&\|a^{q_1}(f_{i,N})\cdots a^{q_k}(f_{i,N})\|_{L^2}\leq \| a^{q_1}(f_{i,N})\cdots a^{q_{k-1}}(f_{i,N}) \|\|f_{i,N}\|\\ =& M^{k-1}\sqrt{1/N}\|f_t\|\sqrt{1/N}\|f_t\|= (1/N)M^{k-1}\|f_t\|^2 .
		\end{align*}  Thus,$$ \|d_{i,N}\|_{L^2}\leq \sum_{\vec{q}\notin T}\|a^{q_1}(f_{i,N})\cdots a^{q_k}(f_{i,N})\|_{L^2}\leq (1/N)3^k M^{k-1}\|f_t\|^2 = O(1/N).$$
		Finally, for $ \varphi(d_{i,N}) = \langle \Omega,   \sum_{\vec{q}\notin T}a^{q_1}(f_{i,N})\cdots a^{q_k}(f_{i,N})\Omega\rangle = \sum_{\vec{q}\notin T}\langle  \Omega,  a^{q_1}(f_{i,N})\cdots a^{q_k}(f_{i,N})\Omega\rangle$, we note that $ \langle  \Omega,  a^{q_1}(f_{i,N})\cdots a^{q_k}(f_{i,N})\Omega\rangle $ is nonzero only when $ q_1 = - $ and $q_k=+$. But again since $ (-,0,\cdots,0,+)\in T $, for $\vec{q}\notin T$, $ a^{q_{2}}(f_{i,N}),\cdots,a^{q_{k-1}}(f_{i,N}) $ must contains at least one creation or annihilation operator. Therefore, we have for all $\vec{q}\notin T$, $ | \langle \Omega,  a^{q_1}(f_{i,N})\cdots a^{q_k}(f_{i,N})\Omega\rangle |\leq \|a^{q_2}(f_{i,N})\cdots a^{q_{k-1}}(f_{i,N}) \|\|f_{i,N}\|^2 \leq \sqrt{1/N}\|f_t\|M^{k-2}(1/N)\|f_t\|^2 =  (1/N^{3/2})M^{k-2}\|f_t\|^3 $, and thus
		$$ |\varphi_{\Omega}(d_{i,N})| \leq \sum_{\vec{q}\notin T}\|\langle  \Omega,  a^{q_1}(f_{i,N})\cdots a^{q_k}(f_{i,N})\Omega\rangle\|\leq (1/N^{3/2})3^kM^{k-2}\|f_t\|^3= O(1/N^{3/2}). $$
		Finally, combine all the estimates, we obtain
		$$ \|\sum_{i=0}^{N-1}X(f_{i,N})^k- \delta_{2,k}b^2t+ Y(x^k \cdot \chi_{[0,t]}(y))\|= \|\sum_{i=0}^{N-1}d_{i,N}\|\leq \| d_{1,N} \|+2\sqrt{N}\|d_{1,N}\|_{L^2}+(N+1)|\varphi_{\Omega}(d_{i,N})|= O(1/N^{1/2}).$$

		\item For a general measure $\rho$ with $  \int \min\{1,x^2\}d\rho(x) <\infty $, we will reduce the case to 1) by multiplying the term we want to estimate by a large enough projection $q$. For each $ \varepsilon>0 $, we can choose $ M>1>0 $ such that $ \rho( (-\infty, -M)\cup (M,\infty) )< \varepsilon/(t(k+1)) $. For each $i\leq N-1$, let $ q_{M,i,N} $ be the orthogonal completement of the support of $ Y( \chi_{ (-\infty, -M)\cup (M,\infty) }(x)\cdot \chi_{[it/N,(i+1)t/N]}(y) ) $, then $ \varphi_{\Omega}(q_{M,i,N}) \geq 1-t/N\rho(  (-\infty, -M)\cup (M,\infty) ) \geq 1-\varepsilon/(N(k+1))  $ (by Lemma \ref{support of Y}). Denote $S_{M,i,N}$ the product space $ (  (-\infty, -M)\cup (M,\infty) ) \times [it/N,(i+1)t/N] $. Apply Lemma \ref{support of Y} to the subalgebra $ \Gamma( L^{\infty}( S_{M,i,N}  , \rho\otimes \mathcal{L}) )\subseteq \Gamma(\frakA_{\rho}^+\times \frakA_{\mathcal{L}|_{[0,t]}} ) $, we have $ \tilde{Y}(f ) q_{M,i,N}= q_{M,i,N}\tilde{Y}(f) = 0 $ for any $ f\in L^{0}( S_{M,i,N}  , \rho\otimes \mathcal{L})  $. Therefore, we have \begin{align*}
			&Z_{i,N}q_{M,i,N} \\= &[at/N + bX(\xi\otimes\chi_{[\frac{it}{N},\frac{(i+1)t}{N}]}(y) ) + X(x\cdot\chi_{[-1,1]}(x) \chi_{[\frac{it}{N},\frac{(i+1)t}{N}]}(y)) + \tilde{Y}( x\cdot\chi_{\RR/[-1,1]}(x) \chi_{[\frac{it}{N},\frac{(i+1)t}{N}]}(y) )]q_{M,i,N}\\
			=&[at/N + bX(\xi\otimes\chi_{[\frac{it}{N},\frac{(i+1)t}{N}]}(y) ) + X(x\cdot\chi_{[-1,1]}(x) \chi_{[\frac{it}{N},\frac{(i+1)t}{N}]}(y)) + {Y}( x\cdot\chi_{[-N,-1)\cup (1,M]}(x) \chi_{[\frac{it}{N},\frac{(i+1)t}{N}]}(y) )]q_{M,i,N},
		\end{align*}
		Let $ Z'_t= at + bX(\xi\otimes\chi_{[0,t]}(y) ) + X(x\cdot\chi_{[-1,1]}(x) \chi_{[0,t]}(y)) + {Y}( x\cdot\chi_{[-M,1)\cup (1,M]}(x) \chi_{[0,t]}(y) $ be the free L\'{e}vy process with L\'{e}vy triple $ (a,b^2,\rho|_{[-M,M]}) $, and set again $Z'_{i,N} = Z'_{(i+1)t/N}-Z'_{it/N}$. Then the equation above says exactly
		$$ Z_{i,N}q_{M,i,N} = Z'_{i,N}q_{M,i,N},\quad \forall 0\leq i\leq N-1.$$
		For the same reason, we have $ Z^{\{k\}}_{i,N}q_{M,i,N}= [ \delta_{2,k} b^2t/N+ Y(x^k\cdot \chi_{[it/N,(i+1)t/N]}(y))]q_{M,i,N} =  [ \delta_{2,k} b^2t/N+ Y(x^k\cdot \chi_{[-M,M]}(x)\chi_{[it/N,(i+1)t/N]}(y))]q_{M,i,N} =  {Z'}^{\{k\}}_{i,N}q_{M,i,N}$
		
		Now, since $ \sum_{i=0}^{N-1}(Z_{i,N})^k - Z_t^{\{k\}} = \sum_{i=0}( (Z_{i,N})^k - Z_{i,N}^{\{k\}} ) $ is a polynomial of $ Z_{i,N} $'s and $ Z_{i,N}^{\{k\}} $'s, apply Lemma \ref{projection lemma}, for each $N\geq 1$, there exists a projection $q_{M,N} \in \Gamma(\frakA_{\rho}^+\otimes \frakA_{\mathcal{L}} )$ such that $$ [\sum_{i=0}^{N-1}((Z_{i,N})^k - Z_{i,N}^{\{k\}})]q_{M,N} = [\sum_{i=0}^{N-1}((Z'_{i,N})^k - Z'^{\{k\}}_{i,N} )]{q_{M,N}}, $$ and $\varphi_{\Omega}( q_{M,N}^{\perp} )\leq \sum_{i=0}^{N-1} (k+1)\varphi_{\Omega}( q_{M,i,N}^\perp ) \leq N(k+1)\varepsilon/(N(k+1))=\varepsilon $. But since $Z'_t$ is a bounded free L\'{e}vy process, by 1), for fixed $M$, $\|\sum_{i=0}^{N-1}(Z'_{i,N})^k - Z'^{\{k\}}_t \| = O(1/N^{1/2})$, we can choose a $ N_0 $ such that $\forall N\geq N_0$, $ \|\sum_{i=0}^{N-1}(Z'_{i,N})^k - Z'^{\{k\}}_t \|\leq \varepsilon $, and in particular $ \|[\sum_{i=0}^{N-1}((Z_{i,N})^k - Z_{i,N}^{\{k\}})]q_{M,N}\| = \|[\sum_{i=0}^{N-1}((Z'_{i,N})^k - Z'^{\{k\}}_{i,N} )]{q_{M,N}}\| \leq \varepsilon$. Thus, we have shown that $ \forall \varepsilon $, there exists a $ N_0\in \NN_+ $ such that $ \forall N\geq N_0$, there exists a projection $ q=q_{M,N} $ such that $  \|[\sum_{i=0}^{N-1}(Z_{i,N})^k - Z_{t}^{\{k\}}]q\|\leq \varepsilon $, and $ \varphi_\Omega(q^\perp)\leq \varepsilon $, which is exactly saying that $ \sum_{i=0}^{N-1}(Z_{i,N})^k $ converges to $Z_{t}^{\{k\}}  $ in probability.
	\end{enumerate}
\end{proof}

\begin{thm}
	Let $ Z_t $ be the free L\'{e}vy process in $ \Gamma( \frakA_\rho^+\otimes \frakA_{\mathcal{L}} ) $ as \eqref{representation of free levy}. If the free Gaussian part of $Z_t$ is nonzero, i.e. $b \neq 0$, then the von Neumann algebra $M_t$ generated by $ (Z_s)_{s\leq t} $ is $ \Gamma( \frakA_\rho^+\otimes \frakA_{\mathcal{L}|[0,t]} ) $. If $Z_t$ has no free Gaussian part (i.e. $b=0$), then $M_t = \Gamma( \frakA_\rho\otimes \frakA_{\mathcal{L}|[0,t]} ) )$. Here $\frakA_{\mathcal{L}|[0,t]} $ is the Hilbert algebra $ L^\infty([0,t],\mathcal{L}|_{[0,t]}) $.
\end{thm}
\begin{proof}
	We first show that for any $t>0$ and a continous function $ f\in C_{c}(\RR/\{0\}) $ (i.e. the support of $f$ is compact and away from $0$), we have $ Y(f(x)\chi_{[0,t]}(y))\in \Gamma( \frakA_\rho^+\otimes \frakA_{\mathcal{L}|[0,t]}) $. For any $ \varepsilon>0 $, choose a $ M_{\varepsilon}>2 $ such that $ t\rho(\RR/[-M_{\varepsilon},M_{\varepsilon}])<\varepsilon $, and let $ p_{\varepsilon}(x) $ be a polynomial in $x$ such that $ |p_{\varepsilon}(x) - f(x)/x^2|<\varepsilon/M_{\varepsilon}^2<\varepsilon $ for all $ -M\leq x\leq M $. We now want to show that $ \tilde{Y}(x^2p_{1/n}(x)\chi_{[0,t]}(y))) $ converges to $ Y(f(x)\chi_{[0,t]}(y)) $ in probability as $n\to \infty$. $ \tilde{Y}(x^2p_{1/n}(x)\chi_{[0,t]}(y))) - Y(f(x)\chi_{[0,t]}(y))  = t\int_{[-1,1]}x^2[p_{1/n}(x)-f(x)/x^2]d\rho(x) + X(x^2[p_{1/n}(x)-f(x)/x^2]\chi_{[-1,1]}(x)\chi_{[0,t]}(y))) + \tilde{Y}(x^2[p_{1/n}-f(x)/x^2](x)\chi_{\RR/[-1,1]}(x)\chi_{[0,t]}(y))):=I_1+I_2+I_3 $. Since $ |p_{1/n}(x) - f(x)/x^2|<1/(nM_{1/n}^2)<1/n $, $|I_1|\leq  (t\int_{[-1,1]}x^2d\rho(x) )/n$ and $ \|I_2\|\leq 1/n + 2\sqrt{t( \int_{[-1,1]}x^4 d\rho(x) )}/n $. To estimate $ I_3 $, let $ q_{n} $ be the orthogonal complement of the support of $ Y( \chi_{\RR/[-M_{1/n},M_{1/n}]}(x)\chi_{[0,t]}(y) ) $, then $ \varphi_{\Omega}(1-q_n) = t\rho( \RR/[-M_{1/n},M_{1/n}] )<1/n $ and $  I_3q_{n} = Y(x^2[p_{1/n}-f(x)/x^2](x)\chi_{[-M_{1/n},M_{1/n}]/[-1,1]}(x)\chi_{[0,t]}(y)))q_n $. Therefore $ \|I_3q_n\|\leq t\rho(\RR/[-1,1])/n+ 1/n + 2\sqrt{ t\rho(\RR/[-1,1]) }/n $. So, overall, for each $ n $, we have $$ \| [\tilde{Y}(x^2p_{1/n}(x)\chi_{[0,t]}(y))) - Y(f(x)\chi_{[0,t]}(y)) ]q_n \|\leq C_{\rho,t}/n$$ where $ C_{\rho,t} $ is a constant only depending on $ \rho $ and $t$ and therefore $ \tilde{Y}(x^2p_{1/n}(x)\chi_{[0,t]}(y))) $ converges to $ Y(f(x)\chi_{[0,t]}(y)) $ in probability.
	
	Let $ M_t $ be the von Neumann algebra generated by $ (Z_s)_{s\leq t} $, then by the previous paragraph $ Y(C_c(\RR/\{0\})\otimes \chi_{[0,t]}(y)) \subseteq M_t$. Let $ A\subseteq \frakA_{\mathcal{L}} $ be the Hilbert subalgebra generated by the projections $ (\chi_{[s_1,s_2](y)})_{s_1,s_2\leq t} $, then we have $ Y(\frakA_{\rho}\otimes A)\subseteq M_t $. By the Proposition \ref{density results}, we have $ \Gamma(  C_c(\RR/\{0\})\otimes A )=\Gamma( [ C_c(\RR/\{0\})\otimes A]'' )=\Gamma(\frakA_{\rho}\otimes\frakA_{\mathcal{L}|[0,t]}) \subseteq M_t $. Finally, since for each $ t>0 $, $ bX(\xi\otimes\chi_{[0,t]}(y) )  = Z_t -(at + X(x\cdot\chi_{[-1,1]}(x) \chi_{[0,t]}(y)) + \tilde{Y}( x\cdot\chi_{\RR/[-1,1]}(x) \chi_{[0,t]}(y) )) $, if $b\neq 0$, then $ \Gamma(\CC\xi \otimes \frakA_{\mathcal{L}|[0,t]}) $ is also contained in $ M_t $ and hence $ M_t = \Gamma(\frakA_\rho \otimes \frakA_{\mathcal{L}|[0,t]})\ast \Gamma( \CC\xi \otimes \frakA_{\mathcal{L}|[0,t]} ) =\Gamma(\frakA^+_{\rho} \otimes \frakA_{\mathcal{L}}) $. If instead $b=0$, then $ Z_t $ is already affiliated with $ \Gamma(\frakA_\rho \otimes \frakA_{\mathcal{L}|[0,t]})\subseteq M_t $ and thus $ M_t = \Gamma(\frakA_\rho \otimes \frakA_{\mathcal{L}|[0,t]}) $.
\end{proof}
\begin{cor}\label{final}
	Let $ Z_t $ be a free L\'{e}vy process with L\'{e}vy triplet $ (a,b,\rho) $ in some $W^*$-probability space, then for any $t>0$, the filtration von Neumann algebra $ M_{t} = W^*((Z_s)_{s\leq t}) $ is an interpolated free group factor (with a possible additional atom)
	$$ M_{t} \simeq \begin{cases}
		L(\FFF_{\infty}), \quad \mbox{if }b\neq 0 \mbox{ or } \rho(\RR)=\infty,\\
		L(\FFF_{ 2t\rho(\RR) }), \quad \mbox{if }b= 0 \mbox{ and } 1\leq t\rho(\RR)<\infty\\
		\underset{t\rho(\RR)}{L(\FFF_{ 2})}\oplus \underset{1-t\rho(\RR)}{\CC}, \quad \mbox{if }b= 0 \mbox{ and } t\rho(\RR)<1.
	\end{cases} $$
\end{cor}
\begin{proof}
	If $b=0$, then $ M_t \simeq \Gamma(\frakA_{\rho}\otimes \frakA_{\mathcal{L}|[0,t]}) = \Gamma(\frakA_{\rho\times \mathcal{L}|[0,t]}) $. Since $ \rho \times \mathcal{L}|_{[0,t]} $ is a continuous measure with total variation $ t\rho(\RR) $, we have $  M_t\simeq \Gamma(L^\infty([0,t\rho(\RR)])) $, and the statement follows from Corollary \ref{filtration of free poisson}. If $ b\neq 0 $, then $ M_t \simeq \Gamma(\frakA_{\rho}\otimes \frakA_{\mathcal{L}|[0,t]})\ast \Gamma(\CC\xi \otimes \frakA_{\mathcal{L}|[0,t]}) $, but since $ \Gamma(\CC\xi \otimes \frakA_{\mathcal{L}|[0,t]})$ is simply the free Gaussian algebra $\Gamma_{Gau}(L^2([0,t]))\simeq L(\FFF_{\infty})$, we obtain $ M_t\simeq  L(\FFF_{\infty})$.
\end{proof}

\section*{Acknowledgments}
I thank Michael Anshelevich for suggesting this topic to me, and for many valuable discussions. I also thank Brent Nelson for his helpful comments.

\bibliographystyle{alpha}
\bibliography{test3}
\end{document}